\input ./style/arxiv-general.cfg
\documentclass[MSNbibl,number,citesort,seceqn,dvips]{arxbj}
\makeatletter
   \@ifpackageloaded{graphicx}{}{\usepackage{graphicx}}
\makeatother

\volume{22}
\issue{4}
\pubyear{2016}
\firstpage{2029}
\lastpage{2079}
\doi{10.3150/15-BEJ719}
\docsubty{FLA}

\makeatletter

\newcommand{\rrvert}{\vert}

\newcommand{\llvert}{\vert}
\renewcommand{\mid}{|}
\newcommand{\hatk}{\hat{k}}
\newcommand{\barxi}{\bar{\xi}}
\newcommand{\mathds}{\mathbb}
\newcommand{\var}{\operatorname{Var}}
\newcommand{\La}{\Lambda}
\newcommand{\De}{\Delta}
\newcommand{\e}{\mathbb{E}}
\renewcommand{\P}{\mathbb{P}}
\newcommand{\mc}{\mathcal{C}}
\newtheorem{lemma}{Lemma}[section]
\newtheorem{proposition}{Proposition}[section]
\newtheorem{theorem}{Theorem}[section]
\newremark{remark}{Remark}[section]
\makeatother

\begin{document}
\begin{frontmatter}

\title{Cram\'er type moderate deviation theorems for self-normalized processes}
\runtitle{Self-normalized moderate deviations}

\begin{aug}
\author[A]{\inits{Q.-M.}\fnms{Qi-Man}~\snm{Shao}\thanksref
{A}\ead[label=e1]{qmshao@cuhk.edu.hk}}
\and
\author[B]{\inits{W.-X.}\fnms{Wen-Xin}~\snm{Zhou}\corref{}\thanksref{B,C}\ead
[label=e2]{wenxinz@princeton.edu}}
\address[A]{Department of Statistics, The Chinese University of Hong
Kong, Shatin, NT, Hong Kong.\\
\printead{e1}}

\address[B]{Department of Operations Research and Financial Engineering, Princeton University, Princeton, NJ~08544, USA.
\printead{e2}}

\address[C]{School of Mathematics and Statistics, University of
Melbourne, Parkville, VIC 3010, Australia}
\end{aug}

%
\received{\smonth{9} \syear{2013}}
%
\revised{\smonth{8} \syear{2014}}

%
\begin{abstract}
Cram\'er type moderate deviation theorems quantify the accuracy of the
relative error of the normal approximation
and provide theoretical justifications for many commonly used methods
in statistics. In this paper,
we develop a new randomized concentration inequality and establish a
Cram\'er type moderate deviation
theorem for general self-normalized processes which include many
well-known Studentized nonlinear
statistics. In particular, a sharp moderate deviation theorem under
optimal moment conditions is established for Studentized $U$-statistics.
\end{abstract}

%
\end{frontmatter}

\section{Introduction}
\label{intro}

Let $T_n$ be a sequence of random variables and assume that $T_n$
converges to $Z$ in distribution. The problem we are interested in is
to calculate the tail probability of $T_n$, $\P(T_n \geq x)$, where
$x$ may also depend on $n$ and can go to infinity. Because the true
tail probability of $T_n$ is typically unknown, it is common practice
to use the tail probability of $Z$ to estimate that of $T_n$. A natural
question is how accurate the approximation is? There are two major
approaches for measuring the approximation error. One approach is to
study the absolute error via Berry--Esseen type bounds or Edgeworth
expansions. The other is to estimate the relative error of the tail
probability of $T_n$ against the tail probability of the limiting
distribution, that is,
\[
{
\P(T_n \geq x) \over\P(Z\geq x)},\qquad x \geq0. %
\]
A typical result in this direction is the so-called \textit{Cram\'er type
moderate deviation}. The focus of this paper is to find the largest
possible $a_n$ ($a_n \to\infty$) so that
\[
{\P(T_n \geq x) \over\P(Z\geq x)} = 1 +o(1) %
\]
holds uniformly for $ 0 \leq x \leq a_n$.

The moderate deviation, and other noteworthy limiting properties for
self-normalized sums are now well-understood. More specifically, let
$X_1, X_2, \ldots, X_n$ be independent and identically distributed
(i.i.d.) non-degenerate real-valued random variables with zero means,
and let
\[
S_n = \sum_{i=1}^nX_i\quad\mbox{and}\quad V_n^2 = \sum_{i=1}^nX_i^2
\]
be, respectively, the partial sum and the partial quadratic sum. The
corresponding self-normalized sum is defined as $S_n/V_n$. The study of
the asymptotic behavior of self-normalized sums has a long history.
Here, we refer to \cite{LoganMallowsRiceShepp1973} for weak
convergence and to \cite{GriffinKuelbs1989,GriffinKuelbs1991} for the
law of the iterated logarithms when $X_1$ is in the domain of
attraction of a normal or stable law. \cite{BentkusGotze1996} derived
the optimal Berry--Esseen bound, and \cite{GineGotzeMason1997} proved
that $S_n/V_n$ is asymptotically normal if and only if $X_1$ belongs to
the domain of attraction of a normal law. Under the same necessary and
sufficient conditions, \cite{CsorgoSzyszkowiczWang2003} proved a
self-normalized analogue of the weak invariance principle. It should be
noted that all of these limiting properties also hold for the
standardized sums. However, in contrast to the large deviation
asymptotics for the standardized sums, which require a finite moment
generating function of $X_1$, \cite{Shao1997} proved a self-normalized
large deviation for $S_n/V_n$ without any moment assumptions. Moreover,
\cite{Shao1999} established a self-normalized Cram\'er type moderate
deviation theorem under a finite third moment, that is, if $\e
\llvert X_1\rrvert ^{3}<\infty$, then
\begin{equation}
\frac{\P( S_n/V_n \geq x )}{1-\Phi(x)} \rightarrow1\qquad\mbox{holds
uniformly for } 0\leq x \leq o
\bigl(n^{1/6}\bigr), \label{ld-t0}
\end{equation}
where $\Phi(\cdot)$ denotes the standard normal distribution
function. Result (\ref{ld-t0}) was further extended to independent
(not necessarily identically distributed) random variables by \cite
{JingShaoWang2003} under a Lindeberg type condition. In particular, for
independent random variables with $\e X_i=0$ and $\e\llvert
X_i\rrvert ^3<\infty$,
the general result in \cite{JingShaoWang2003} gives
\begin{equation}
{ \P({S_n/ V_n} \geq x) \over
1-\Phi(x)} = 1 + O(1) (1+x )^3 { \sum_{i=1}^n \e\llvert X_i\rrvert ^3
\over
(\sum_{i=1}^n \e X_i^2)^{3/2}}
\label{md-t}
\end{equation}
for $0\leq x\leq(\sum_{i=1}^n \e X_i^2)^{1/2}/ (\sum_{i=1}^n \e
\llvert X_i\rrvert ^3)^{1/3}$.

Over the past two decades, there has been significant progress in the
development of the self-normalized limit theory. For a systematic
presentation of the general self-normalized limit theory and its
statistical applications, we refer to \cite{PenaLaiShao2009}.

The main purpose of this paper is to extend (\ref{md-t}) to more
general self-normalized processes, including many commonly used
Studentized statistics, in particular, Student's $t$-statistic and
Studentized $U$-statistics. Notice that the proof in \cite
{JingShaoWang2003} is lengthy and complicated, and their method is
difficult to adopt for general self-normalized processes. The proof in
this paper is based on a new randomized concentration inequality and
the method of conjugated distributions (also known as the change of
measure method), which opens a new approach to studying self-normalized
limit theorems.

The rest of this paper is organized as follows. The general result is
presented in Section~\ref{mainresultssec}. To illustrate the
sharpness of the general result, a result similar to (\ref{ld-t0}) and
(\ref{md-t}) is obtained for Studentized $U$-statistics in
Section~\ref
{Self-Usec}. Applications to other Studentized statistics will be
discussed in our future work. To establish the general Cram\'er type
moderation theorem, a novel randomized concentration inequality is
proved in Section~\ref{con-ineqsec}. The proofs of the main results
and key technical lemmas are given in Sections~\ref{proof1sec} and
\ref{proof2sec}. Other technical proofs are provided in the \hyperref[appe]{Appendix}.

\section{Moderate deviations for self-normalized processes}
\label{mainresultssec}

Our research on self-normalized processes is motivated by Studentized
nonlinear statistics. Nonlinear statistics are the building blocks in
various statistical inference problems. It is known that many of these
statistics can be written as a partial sum plus a negligible term.
Typical examples include $U$-statistics, multi-sample $U$-statistics,
$L$-statistics, random sums and functions of nonlinear statistics. We
refer to \cite{ChenShao2007} for a unified approach to uniform and
non-uniform Berry--Esseen bounds for standardized nonlinear statistics.

Assume that the nonlinear process of interest can be decomposed as a
standardized partial sum of independent random variables plus a
remainder, that is,
\[
{ 1 \over\sigma} \Biggl(\sum_{i=1}^n
\xi_i + D_{1n} \Biggr),
\]
where $\xi_1,\ldots,\xi_n$ are independent random variables satisfying
\begin{equation}
\e\xi_i =0\qquad\mbox{for } i =1, \ldots, n\quad\mbox{and}\quad \sum
_{i=1}^n\e\xi_i^2
=1, \label{200}
\end{equation}
and where $D_{1n}=D_{1n}(\xi_1,\ldots,\xi_n)$ is a measurable
function of $\{\xi_i\}_{i=1}^n$. Because $\sigma$ is typically
unknown, a self-normalized process\vspace{-2pt}
\[
T_n = { 1 \over\widehat{\sigma}} \Biggl( \sum
_{i=1}^n \xi_i + D_{1n}
\Biggr) %
\]
is more commonly used in practice, where $\widehat{\sigma}$ is an
estimator of $\sigma$. Assume that $\widehat{\sigma}$ can be written as\vspace{-2pt}
\[
\widehat{\sigma} = \Biggl\{ \Biggl(\sum_{i=1}^n
\xi_i^2 \Biggr) (1+D_{2n})\Biggr
\}^{1/2}, %
\]
where $D_{2n}$ is a measurable function of $\{\xi_i\}_{i=1}^n$.
Without loss of generality and for the sake of convenience,
we assume $\sigma=1$. Therefore, under the assumptions in (\ref
{200}), we can rewrite the self-normalized process $T_n$ as\vspace{-2pt}
\begin{equation}
T_n=\frac{W_n+D_{1n}}{V_n(1+D_{2n})^{1/2}}, \label{stu-stat}
\end{equation}
where\vspace{-2pt}
\[
W_n=\sum_{i=1}^n
\xi_i,\qquad V_n = \Biggl( \sum_{i=1}^n
\xi_i^2 \Biggr)^{1/2}. %
\]

Essentially, this formulation (\ref{stu-stat}) states that, for a
nonlinear process that be can written as a linear process plus a
negligible remainder, it is natural to expect that the corresponding
normalizing term is dominated by a quadratic process. To ensure that
$T_n$ is well-defined, it is assumed implicitly in (\ref{stu-stat})
that the random variable $D_{2n}$ satisfies $1+D_{2n}>0$. Examples
satisfying (\ref{stu-stat}) include the $t$-statistic, Studentized $U$-
and $L$-statistics. See \cite{WangJingZhao2000} and the references
therein for more details.

In this section, we establish a general Cram\'er type moderate
deviation theorem for a self-normalized process $T_n$ in the form of
(\ref{stu-stat}). We start by introducing some of the basic notation
that is frequently used throughout this paper. For $x \geq1$, write\vspace{-2pt}
\begin{equation}
L_{n,x} = \sum_{i=1}^n
\delta_{i,x} , \qquad I_{n,x} = \e\exp\bigl(
xW_n-x^2V_n^2/2 \bigr) = \prod
_{i=1}^n \e\exp\bigl( \xi_{i,x}-
\xi_{i,x}^2/2 \bigr), \label{deix}
\end{equation}
where $ \delta_{i,x} = \e\xi_{i,x}^2 I(\llvert \xi_{i,x} \rrvert >1)
+ \e\llvert \xi
_{i,x}\rrvert ^3 I( \llvert \xi_{i,x} \rrvert \leq1) $ with $\xi
_{i,x } := x \xi_i$. For
$i=1,\ldots, n$, let $D_{1n}^{(i)}$ and $D_{2n}^{(i)}$ be arbitrary
measurable functions of $\{\xi_j\}_{j=1, j\neq i}^n$, such that $\{
D_{1n}^{(i)}, D_{2n}^{(i)}\}$ and $\xi_i$ are independent. Moreover, define
%
\begin{eqnarray}\label{rn1}
R_{n,x} & =& I_{n,x}^{-1} \times\Biggl( \e\bigl\{
\bigl(x\llvert D_{1n} \rrvert+x^2 \llvert
D_{2n} \rrvert\bigr) e^{\sum_{j=1}^n ( \xi_{j,x} - \xi_{j,x}^2 /2) }
\bigr\} \nonumber\\[-8pt]\\[-8pt]\nonumber
&&{} +\sum_{i=1}^n\e\bigl[ \min\bigl(
\llvert\xi_{i,x} \rrvert,1 \bigr) \bigl\{ \bigl\llvert
D_{1n}-D_{1n}^{(i)}\bigr\rrvert+x \bigl\llvert
D_{2n}-D_{2n}^{(i)}\bigr\rrvert\bigr\}
e^{\sum_{j \neq i}
(\xi_{j,x} - \xi_{j,x}^2/2)} \bigr] \Biggr).
\nonumber
\end{eqnarray}
Here, and in the sequel, we use $\sum_{j \neq i} = \sum_{j=1, j\neq
i}^n$ for brevity.

Now we are ready to present the main results.

%
\begin{theorem} \label{t21}
Let $T_n$ be defined in (\ref{stu-stat}) under condition (\ref
{200}). Then
there exist positive absolute constants $C_1$--$C_4$ and $c_1$ such that\vspace{-3pt}
\begin{equation}
\P(T_n\geq x) \geq\bigl\{1-\Phi(x)\bigr\} \exp\{ -C_1
L_{n,x} \} ( 1- C_2 R_{n,x} ) \label{cmd-lbd}
\end{equation}
and\vspace{-3pt}
%
\begin{eqnarray}\label{cmd-ubd}
\P(T_n\geq x) & \leq&\bigl\{ 1-\Phi(x) \bigr\} \exp\{ C_3
L_{n,x}\} ( 1+ C_4 R_{n,x} )
\nonumber\\[-8pt]\\[-8pt]\nonumber
&&{} + \P\bigl( x\llvert D_{1n} \rrvert> V_n / 4
\bigr) + \P\bigl( x^2 \llvert D_{2n} \rrvert>1/ 4 \bigr)
\end{eqnarray}
for all $x\geq1$ satisfying\vspace{-3pt}
\begin{equation}
\max_{1\leq i\leq n} \delta_{i,x} \leq1 \label{c1}
\end{equation}
and\vspace{-3pt}
\begin{equation}
L_{n,x} \leq c_1 x^2 . \label{rc}
\end{equation}
\end{theorem}

%
\begin{remark} \label{comment1}
The quantity $L_{n,x}$ in (\ref{deix}) is essentially the same as
the factor $\De_{n,x}$ in \cite{JingShaoWang2003}, which is the
leading term that describes the accuracy of the relative normal
approximation error. To deal with the self-normalized nonlinear process
$T_n$, first we need to ``linearize'' it in a proper way, although at
the cost of introducing some complex perturbation terms. The linearized
term is $x W_n - x^2 V_n^2/2$, and its exponential moment is denoted by
$I_{n,x}$ as in (\ref{deix}). A randomized concentration inequality is
therefore developed (see Section~\ref{con-ineqsec}) to cope with
these random perturbations which lead to the quantity $R_{n,x}$ given
in (\ref{rn1}). Similar quantities also appear in the Berry--Esseen
bounds for nonlinear statistics. See, for example, Theorems~2.1 and 2.2
in \cite{ChenShao2007}.
\end{remark}

Theorem~\ref{t21} provides the upper and lower bounds of the relative
errors for $x\geq1$. To cover the case of $0 \leq x \leq1$, we
present a rough estimate of the absolute error in the next theorem, and
refer to \cite{ShaoZhangZhou2014} for the general Berry--Esseen bounds
for self-normalized processes.

%
\begin{theorem} \label{t23}
There exists an absolute constant $C>1$ such that for all $x\geq0$,
\begin{equation}
\bigl\llvert\P(T_n \leq x) -\Phi(x) \bigr\rrvert\leq C
\breve{R}_{n,x} , \label{sc1}
\end{equation}
where
%
\begin{eqnarray}\label{rn2}
\breve{R}_{n,x} & :=& L_{n, 1+ x} + \e\llvert D_{1n}
\rrvert+x \e\llvert D_{2n} \rrvert
\nonumber\\[-8pt]\\[-8pt]\nonumber
&&{}+ \sum_{i=1}^n \e\bigl[
\xi_i I\bigl\{\llvert\xi_i
\rrvert\leq1/(1+x)\bigr\} \bigl\{ \bigl\llvert D_{1n}-D_{1n}^{(i)}
\bigr\rrvert+x \bigl\llvert D_{2n}-D_{2n}^{(i)}
\bigr\rrvert\bigr\} \bigr]
\nonumber
\end{eqnarray}
for $L_{n,1+x}$ as in (\ref{deix}).
\end{theorem}

The proof of Theorem~\ref{t23} is deferred to the \hyperref[appe]{Appendix}. In
particular, when $0\leq x\leq1$, the quantity $ L_{n,1+x}$ satisfies
\begin{eqnarray}
L_{n,1+x} & =& (1+x)^2 \sum_{i=1}^n
\e\xi_i^2 I\bigl\{ \llvert\xi_i \rrvert
>1/(1+x) \bigr\} + (1+x)^3 \sum_{i=1}^n
\e\llvert\xi_i \rrvert^3 I\bigl\{ \llvert
\xi_i \rrvert\leq1/(1+x) \bigr\}
\nonumber
\\
& \leq& (1+x)^2 \sum_{i=1}^n
\e\xi_i^2 I\bigl( \llvert\xi_i \rrvert
>1/2 \bigr) + (1+x)^3 \sum_{i=1}^n
\e\llvert\xi_i \rrvert^3 I\bigl( \llvert
\xi_i \rrvert\leq1\bigr)
\nonumber
\\
& \leq& (1+x)^2 \sum_{i=1}^n
\e\xi_i^2 I\bigl( \llvert\xi_i \rrvert>1
\bigr) + (1+x)^2 \sum_{i=1}^n
\e\xi_i^2 I\bigl( 1/2 < \llvert\xi_i
\rrvert\leq1 \bigr)
\nonumber
\\
&&{}+ (1+x)^3 \sum_{i=1}^n
\e\llvert\xi_i \rrvert^3 I\bigl( \llvert
\xi_i \rrvert\leq1 \bigr) ,
\nonumber
\end{eqnarray}
which can be further bounded, up to a constant, by
\[
\sum_{i=1}^n \e\xi_i^2
I\bigl( \llvert\xi_i \rrvert>1\bigr) + \sum
_{i=1}^n\e\llvert\xi_i \rrvert
^3 I\bigl( \llvert\xi_i \rrvert\leq1 \bigr).
\]

%
\begin{remark} \label{r20}
1.~When $D_{1n}=D_{2n}=0$, $T_n$ reduces to the
self-normalized sum of independent random variables, and thus
Theorems~\ref{t21} and~\ref{t23} together immediately imply the
main result in \cite{JingShaoWang2003}. The proof therein, however, is
lengthy and fairly complicated, especially the proof of Proposition
5.4, and can hardly be applied to prove the general result of Theorem
\ref{t21}. The proof of our Theorem~\ref{t21} is shorter and more
transparent.

2.~$D_{1n}$ and $D_{2n}$ in the definitions of $R_{n,x}$
and $\breve{R}_{n,x}$ can be replaced by any non-negative random
variables $D_{3n}$ and $D_{4n}$, respectively, provided that $\llvert
D_{1n} \rrvert
\leq D_{3n}$, $\llvert D_{2n} \rrvert \leq D_{4n}$.

3.~Condition (\ref{200}) implies that $\xi_i$ actually depends
on both $n$ and $i$; that is, $\xi_i$ denotes $\xi_{ni}$, which is an
array of independent random variables.
\end{remark}

\section{Studentized $U$-statistics}
\label{Self-Usec}

As a prototypical example of the self-normalized processes given in
(\ref{stu-stat}), we are particularly interested in Studentized
$U$-statistics. In this section, we apply Theorems~\ref{t21} and~\ref
{t23} to Studentized $U$-statistics and obtain a sharp Cram\'er
moderate deviation under optimal moment conditions.

Let $X_1, X_2, \ldots, X_n$ be a sequence of i.i.d. random variables
and let $h:\mathds{R}^m\rightarrow\mathds{R}$ be a symmetric Borel
measurable function of $m$ variables, where $2 \leq m < n/2$ is fixed.
The Hoeffding's $U$-statistic
with a kernel $h$ of degree $m$ is defined as (Hoeffding \cite{Hoeffding1948})
\begin{eqnarray*}
U_n=\frac{1}{{n \choose m}}\sum_{1\leq i_1<\cdots<i_m\leq
n}h(X_{i_1},
\ldots,X_{i_m}),
\end{eqnarray*}
which is an unbiased estimate of $\theta=\e h(X_1,\ldots,X_m)$. Let
\[
h_1(x)=\e\bigl\{ h(X_1, X_2,
\ldots,X_m) \mid X_1=x \bigr\} ,\qquad x \in\mathbb{R}
\]
and
\begin{equation}
\sigma^2 =\var\bigl\{ h_1(X_1) \bigr\},
\qquad\sigma_h^2 = \var\bigl\{h(X_1,
X_2,\ldots,X_m)\bigr\}. \label{vardef}
\end{equation}
Assume $0< \sigma^2 < \infty$, then the standardized non-degenerate
$U$-statistic is given by
\begin{eqnarray*}
Z_n = {\sqrt{n} \over m \sigma} (U_n-\theta).
\end{eqnarray*}

The $U$-statistic is a basic statistic and its asymptotic properties
have been extensively studied in the literature. We refer to \cite
{KoroljukBorovskich1994} for a systematic presentation of the theory of
$U$-statistics. For uniform Berry--Esseen bounds, see \cite
{Filippova1962,GramsSerfling1973,Bickel1974,ChanWierman1977,CallaertJanssen1978,Serfling1980,vanZwet1984,Friedrich1989,AlberinkBentkus2001,AlberinkBentkus2002,WangWeber2006} and \cite{ChenShao2007}. We refer to \cite
{EichelsbacherLowe1995,KeenerRobinsonWeber1998} and \cite
{BorovskikhWeber2003a,BorovskikhWeber2003b} for large and moderate
deviation asymptotics.

Because $\sigma$ is usually unknown, we are interested in the
following Studentized $U$-statistic (Arvensen \cite{Arvesen1969}),
which is widely used in practice:
\begin{eqnarray*}
T_n={ \sqrt{n} \over m s_1} (U_n-\theta),
\end{eqnarray*}
where $s_1^2$ denotes the leave-one-out Jackknife estimator of $\sigma
^2$ given by
%
\begin{eqnarray} \label{s1}
s_1^2 &=& \frac{(n-1)}{(n-m)^2}\sum
_{i=1}^n(q_i-U_n)^2\qquad\mbox{with}
\nonumber\\[-8pt]\\[-8pt]\nonumber
q_i & =& \frac{1}{{n-1 \choose m-1}} \mathop{\sum_{ 1\leq\ell_1< \cdots<
\ell_{m-1}\leq n}}_{\ell_j \neq i, j=1, \ldots,
m-1}h(X_i,
X_{\ell_1},\ldots,X_{\ell_{m-1}} ).
\nonumber
\end{eqnarray}
In contrast to the standardized $U$-statistics, few optimal limit
theorems are available for Studentized $U$-statistics in the
literature. A uniform Berry--Esseen bound for Studentized
$U$-statistics was proved in \cite{WangJingZhao2000} for $m=2$ and
$\e\llvert h(X_1, X_2)\rrvert ^3 < \infty$. However, a finite third
moment of
$h(X_1, X_2)$ may not be an optimal condition. Partial results on Cram\'
er type moderate deviation were obtained in \cite
{VandemaeleVeraverbeke1985,Wang1998} and \cite{LaiShaoWang2011}.

As a direct but non-trivial consequence of Theorems~\ref{t21} and
\ref{t23}, we establish the following sharp Cram\'er type moderate
deviation theorem for the Studentized $U$-statistic~$T_n$.

%
\begin{theorem} \label{t11}
Assume that $\sigma_p : = ( \e\llvert h_1(X_1)-\theta\rrvert ^p
)^{1/p} < \infty$
for some $2 < p \leq3$. Suppose that there are constants $c_0\geq1$
and $\tau\geq0$ such that
\begin{eqnarray}
\bigl\{ h(x_1,\ldots,x_m) - \theta\bigr
\}^2 \leq c_0 \Biggl[ \tau\sigma^2 + \sum
_{i=1}^m \bigl\{ h_1(x_i)-
\theta\bigr\}^2 \Biggr]. \label{k-c}
\end{eqnarray}
Then there exist positive constants $C_1$ and $c_1$ independent of $n$
such that
%
\begin{equation}
\frac{ \P(T_n\geq x)}{1-\Phi(x)}=1 + O(1) \biggl\{ ( \sigma_p /
\sigma)^p \frac{(1+x)^{p} }{n^{ p/2 -1 } } +( \sqrt{a_m}+
\sigma_h/\sigma) \frac{(1+x)^3}{\sqrt{n}} \biggr\} \label{t11a}
\end{equation}
holds uniformly for
\[
0 \leq x \leq c_1 \min\bigl\{ (\sigma/\sigma_p)
n^{1/2-1/p}, (n/a_m)^{1/6} \bigr\}, %
\]
where $\llvert O(1) \rrvert \leq C_1$ and $a_m=\max\{c_0 \tau, c_0+m\}
$. In particular,
\begin{equation}
{ \P(T_n\geq x) \over1-\Phi(x) } \to1 \label{t11b}
\end{equation}
holds uniformly in $ x\in[ 0, o(n^{1/2-1/p}) )$.
\end{theorem}

It is easy to verify that condition (\ref{k-c}) is satisfied for the
$t$-statistic $(h(x_1, x_2) = (x_1+x_2)/2$ with $c_0=2$ and $\tau=0$),
sample variance ($h(x_1, x_2) = (x_1-x_2)^2/2$, $c_0=10$, $\tau=\theta
^2/\sigma^2$), Gini's mean difference ($h(x_1, x_2) = \llvert
x_1-x_2\rrvert $,
$c_0=8$, $\tau=\theta^2/\sigma^2$) and one-sample Wilcoxon's
statistic ($h(x_1, x_2) = I( x_1+x_2 \leq0)$, $c_0=1$, $\tau=
1/\sigma^2$). Although it may be interesting to investigate whether
condition (\ref{k-c}) can be weakened, it seems that it is impossible
to remove condition (\ref{k-c}) completely. We also note that result
(\ref{t11b}) was earlier proved in \cite{LaiShaoWang2011} for $m=2$.
However, the approach used therein can hardly be extended to the case
$m\geq3$.

\section{A randomized concentration inequality}
\label{con-ineqsec}

To prove Theorem~\ref{t21}, we first develop a randomized
concentration inequality via Stein's method. Stein's method (Stein
\cite{Stein1986}) is a powerful tool in the normal and non-normal
approximation of both independent and dependent variables, and the
concentration inequality is a useful approach in Stein's method. We
refer to \cite{ChenGoldsteinShao2010} for systematic coverage of the
method and recent developments in both theory and applications and to
\cite{ChenShao2007} for uniform and non-uniform Berry--Esseen bounds
for nonlinear statistics using the concentration inequality approach.

Let $\xi_1, \dots,\xi_n$ be independent random variables such that
\begin{eqnarray*}
\e\xi_i =0\qquad\mbox{for } i=1,2,\ldots,n\quad\mbox{and}\quad\sum
_{i=1}^n\e\xi_i^2
=1.
\end{eqnarray*}
Let
\begin{equation}
W=\sum_{i=1}^n\xi_i,\qquad
V^2 = \sum_{i=1}^n
\xi_i^2 \label{WVdef}
\end{equation}
and let $\De_1=\De_1(\xi_1,\ldots,\xi_n)$ and $\De_2=\De_2(\xi
_1,\ldots,\xi_n)$ be two measurable functions of $\xi_1,\ldots,\xi
_n$. Moreover, set
\begin{eqnarray*}
\beta_2=\sum_{i=1}^n\e
\xi_i^2 I\bigl( \llvert\xi_i \rrvert>1
\bigr) , \qquad\beta_3= \sum_{i=1}^n
\e\llvert\xi_i \rrvert^3 I\bigl(\llvert
\xi_i \rrvert\leq1 \bigr).
\end{eqnarray*}

%
\begin{theorem} \label{mod-con-ineq}
For each $1\leq i\leq n$, let $\De_1^{(i)}$ and $\De_2^{(i)}$ be
random variables such that $\xi_i$ and $(\De_1^{(i)},\De
_2^{(i)},W-\xi_i)$ are independent. Then
%
\begin{eqnarray}
\P(\De_1 \leq W \leq\De_2) & \leq& 17(
\beta_2+ \beta_3 ) + 5 \e\llvert\De_2-
\De_1\rrvert+ 2 \sum_{i=1}^n
\sum_{j=1}^2 \e\bigl\llvert
\xi_i \bigl\{ \De_j-\De_j^{(i)}
\bigr\} \bigr\rrvert. \label{con-ineq-1}
\end{eqnarray}
\end{theorem}

We note that a similar result was obtained by \cite{ChenShao2007} with
$\e\llvert W ( \Delta_2 - \Delta_1)\rrvert $ instead of $\e\llvert
\Delta_2 - \Delta
_1\rrvert $ in (\ref{con-ineq-1}). However, using the term $\e\llvert
W ( \Delta_2
- \Delta_1)\rrvert $ will not yield the sharp bound in (\ref{t11a}) when
Theorem~\ref{t21} is applied to Studentized $U$-statistics. This
provides our main motivation for developing the new concentration
inequality (\ref{con-ineq-1}).

\begin{pf*}{Proof of Theorem~\ref{mod-con-ineq}}
Assume without
loss of generality that $\Delta_1 \leq\Delta_2$. The proof is based
on Stein's method. For every $x \in\mathds{R}$, let $f_x(w)$ be the
solution to Stein's equation
\begin{eqnarray}
f_x'(w)-w f_x(w)=I(w\leq x)-\Phi(x),
\label{stein-eqn}
\end{eqnarray}
which is given by
\begin{eqnarray}
\label{sol-form} f_x(w) = \cases{ \displaystyle\sqrt{2\pi}
e^{w^2/2} \Phi(w)\bigl\{ 1-\Phi(x) \bigr\}, &\quad$w\leq x$,
\cr
\displaystyle\sqrt{2\pi} e^{w^2/2} \Phi(x) \bigl\{ 1-\Phi(w) \bigr\}, &
\quad$w>x$.}
\end{eqnarray}
Set $f_{x,y}=f_x-f_y$ for any $x,y \in\mathds{R}$, $\delta= (\beta
_2 + \beta_3)/2$ and
\[
\De_{1,\delta} = \De_1-\delta,\qquad
\De_{2,\delta} = \De_2+\delta,\qquad
\De_{1,\delta}^{(i)} = \De_1^{(i)} -\delta,\qquad
\De_{2,\delta}^{(i)}= \De_2^{(i)} + \delta. %
\]

Noting that $\xi_i$ and $ ( \Delta_1^{(i)}, \Delta_2^{(i)},
W^{(i)}=W-\xi_i )$ are independent and $\e\xi_i=0$ for $i=1,\ldots,
n$, we have
%
\begin{eqnarray} \label{deco-1}
\e\bigl\{ W f_{\De_{2,\delta},\De_{1,\delta}}(W) \bigr\} & =& \sum_{i=1}^n
\e\bigl\{ \xi_i f_{\De_{2,\delta},\De_{1,\delta}}(W) \bigr\}
\nonumber
\\
&=& \sum_{i=1}^n \e\bigl[
\xi_i \bigl\{ f_{\De_{2,\delta},\De_{1,\delta}}(W) - f_{\De
_{2,\delta},\De_{1,\delta}}
\bigl(W^{(i)}\bigr) \bigr\} \bigr]
\nonumber\\[-8pt]\\[-8pt]\nonumber
&&{}+\sum_{i=1}^n\e\bigl[
\xi_i \bigl\{ f_{\De_{2,\delta},\De_{1,\delta
}}\bigl(W^{(i)}\bigr) -
f_{\De^{(i)}_{2,\delta},\De^{(i)}_{1,\delta
}}\bigl(W^{(i)}\bigr) \bigr\} \bigr]
\nonumber
\\
&:=& H_{1}+H_{2}.\nonumber
\end{eqnarray}
By (\ref{sol-form}),
\begin{eqnarray*}
\frac{\partial}{\partial x}f_x(w) = \cases{
-e^{(w^2-x^2)/2} \Phi(w), &\quad$w\leq x$,
\cr
e^{(w^2-x^2)/2} \bigl\{ 1-\Phi(w) \bigr\} , &\quad$w>x$.}
\end{eqnarray*}
Clearly, $ \sup_{x, w} \llvert \frac{\partial}{\partial
x}f_x(w)\rrvert \leq1$
and it follows that
\begin{equation}
\llvert H_{2}\rrvert\leq\sum_{i=1}^n
\sum_{j=1}^2 \e\bigl\llvert
\xi_i\bigl\{\De_j-\De_j^{(i)}
\bigr\} \bigr\rrvert. \label{estimate-1}
\end{equation}

As for $H_{1}$, let $\hat{k}_i(t) = \xi_i \{ I( -\xi_i \leq t \leq
0) -I( 0 < t \leq-\xi_i ) \}$ satisfying $\hat{k}_i(t) \geq0$ and
$\int_{\mathds{R}} \hat{k}_i(t) \,dt= \xi_i^2$. Observe by (\ref
{stein-eqn}) that
\begin{eqnarray}
&& \xi_i \bigl\{ f_{\De_{2,\delta},\De_{1,\delta}}(W) - f_{\De_{2,\delta
},\De_{1,\delta}}
\bigl(W^{(i)}\bigr) \bigr\}
\nonumber
\\
&&\quad = \xi_i\int_{-\xi_i}^0f'_{\De_{2,\delta},\De_{1,\delta}}(W+t)
\,dt
\nonumber
\\
&&\quad = \int_{\mathds{R}} f'_{\De_{2,\delta},\De_{1,\delta}}(W+t) \hat
{k}_i(t) \,dt
\nonumber
\\
&&\quad = \int_{\mathds{R}} (W+t)f_{\De_{2,\delta},\De_{1,\delta}}(W+t)
\hat{k}_i(t) \,dt
\nonumber
\\
&&\qquad{}+ \xi_i^2 \bigl\{\Phi(\De_{1,\delta})-
\Phi(\De_{2,\delta})\bigr\} + \int_{\mathds{R}}I(
\De_{1,\delta} \leq W +t \leq\De_{2,\delta} ) \hat{k}_i(t)
\,dt.
\nonumber
\end{eqnarray}
Adding up over $1\leq i\leq n$ gives
%
\begin{eqnarray}\label{decom-2}
H_{1} &=& \sum_{i=1}^n\e\int
_{\mathds{R}}(W+t)f_{\De_{2,\delta
},\De_{1,\delta}}(W +t) \hat{k}_i(t)
\,dt +\e\bigl[ V^2 \bigl\{ \Phi(\De_{1,\delta})-\Phi(
\De_{2,\delta})\bigr\} \bigr]\nonumber
\\
&&{}  + \sum_{i=1}^n
\mathbb{E}\int_{\mathds{R}}I(\De_{1,\delta} \leq W +t \leq
\De_{2,\delta} ) \hat{k}_i(t) \,dt
\\
&:=& H_{1 1} + H_{1 2} + H_{1 3}
\nonumber
\end{eqnarray}
for $V^2$ given in (\ref{WVdef}). Following the proof of
(10.59)--(10.61) in \cite{ChenGoldsteinShao2010} (or see (5.6)--(5.8)
in \cite{ChenShao2007}), we have
\begin{equation}
H_{1 3} \geq(1/2) \P( \Delta_1 \leq W \leq
\De_2) - \delta, \label{h13}
\end{equation}
where $\delta=(\beta_2+\beta_3)/2$. Assume that $\delta\leq1/8$.
Otherwise, (\ref{con-ineq-1}) is trivial. To finish the proof of~(\ref
{con-ineq-1}), in view of (\ref{deco-1}), (\ref{estimate-1}), (\ref
{decom-2}) and (\ref{h13}),
it suffices to show that
%
\begin{equation}
\llvert H_{1 2}\rrvert\leq0.6 \e\llvert\De_2 -
\De_1\rrvert+ \beta_2 + 0.5 \beta_3
\label{h12}
\end{equation}
and
%
\begin{equation}
\e\bigl\{W f_{\De_{2,\delta},\De_{1,\delta}}(W) \bigr\} - H_{1 1} \leq
1.75 \e\llvert
\De_2 - \De_1\rrvert+ 7 \beta_2 + 6
\beta_3. \label{new-est}
\end{equation}
Next we prove (\ref{h12}) and (\ref{new-est}), starting with (\ref{h12}).

\begin{pf*}{Proof of (\ref{h12})}
Recall that $\De_1 \leq\De_2$ and $\sum
_{i=1}^n\e\xi_i^2 =1$. Let $\barxi_i = \xi_i I( \llvert \xi_i \rrvert
\leq1)$,
we have
\begin{eqnarray*}
\llvert H_{1 2}\rrvert& =& \e\bigl[ V^2 \bigl\{ \Phi(
\De_2)- \Phi(\De_1) \bigr\}\bigr]
\nonumber
\\
& \leq&\sum_{i=1}^n \e
\xi_i^2 I\bigl( \llvert\xi_i \rrvert> 1
\bigr) + \e\Biggl[ \bigl\{ \Phi(\De_2)- \Phi(\De_1)
\bigr\} \sum_{i=1}^n\xi_i^2
I\bigl( \llvert\xi_i \rrvert\leq1\bigr) \Biggr]
\nonumber
\\
& =& \beta_2 + \e\bigl[ \bigl\{ \Phi(\De_2)- \Phi(
\De_1) \bigr\} \bigr] \sum_{i=1}^n
\e\barxi_i^{ 2} + \e\Biggl[ \bigl\{ \Phi(
\De_2)- \Phi(\De_1) \bigr\} \sum
_{i=1}^n \bigl(\barxi_i^{ 2}
- \e\barxi_i^{ 2} \bigr) \Biggr]
\nonumber
\\
& \leq&\beta_2 + { 1 \over\sqrt{2 \pi}} \e( \De_2 -
\De_1 ) + \e\Biggl\{ \min\biggl( 1, { \De_2 - \De_1 \over\sqrt{2 \pi}}
\biggr) \Biggl\llvert\sum_{i=1}^n\bigl(
\barxi_i^{ 2} - \e\barxi_i^{ 2}
\bigr) \Biggr\rrvert\Biggr\}
\nonumber
\\
& \leq&\beta_2 + { 1 \over\sqrt{2 \pi}} \e( \De_2 -
\De_1 ) + { 1 \over2} \e\min\biggl( 1,
{ \De_2 - \De_1 \over\sqrt{2 \pi}} \biggr)^2 + { 1 \over2} \e\Biggl
\{ \sum_{i=1}^n \bigl(
\barxi_i^{ 2} - \e\barxi_i^{ 2}
\bigr) \Biggr\}^2
\nonumber
\\
& \leq&\beta_2 + { 1 \over\sqrt{2 \pi}} \e( \De_2 -
\De_1 ) + { 1 \over2 \sqrt{ 2 \pi}} \e( \De_2 -
\De_1 ) + { 1 \over2} \beta_3
\nonumber
\\
& \leq&0.6 \e( \De_2 - \De_1 ) + \beta_2 +
0.5 \beta_3,
\end{eqnarray*}
as desired.\vadjust{\goodbreak}
\end{pf*}

\begin{pf*}{Proof of (\ref{new-est})} Observe that
%
\begin{eqnarray}\label{h3}
&& \e\bigl\{ W f_{\De_{2,\delta},\De_{1,\delta}}(W) \bigr\} - H_{1 1}
\nonumber
\\
&&\quad  = \e\bigl\{ W f_{\De_{2,\delta},\De_{1,\delta}}(W) \bigl( 1- V^2\bigr
) \bigr\}
\nonumber\\[-8pt]\\[-8pt]\nonumber
&&\qquad{}+ \sum_{i=1}^n\e\int\bigl\{ W
f_{\De_{2,\delta},\De
_{1,\delta}}(W) - (W+t) f_{\De_{2,\delta},\De_{1,\delta}}(W+t)\bigr\}
\hatk_i(t)
\,dt
\nonumber
\\
&&\quad := H_{3 1} + H_{3 2}. \nonumber
\end{eqnarray}
Recall that $\sup_{x, w} \llvert \frac{\partial}{\partial
x}f_x(w)\rrvert \leq
1$. This, together with the following basic properties of $f_x(w)$
(see, e.g., Lemma 2.3 in \cite{ChenGoldsteinShao2010})
%
\begin{eqnarray}
\bigl\llvert w f_x(w)\bigr\rrvert&\leq&1,\qquad \bigl\llvert
f_x (w)\bigr\rrvert\leq1, \label{property-1}
\\
\bigl\llvert w f_x(w) - (w+t) f_x(w+t)\bigr\rrvert
&\leq&\min\bigl\{ 1, \bigl( \llvert w\rrvert+ \sqrt{2\pi}/4 \bigr)
\llvert t\rrvert
\bigr\} \label{property-2}
\end{eqnarray}
and $\llvert f_{x,y}(w)\rrvert \leq\llvert x-y\rrvert $, yields
%
\begin{eqnarray}\label{h31}
H_{3 1} & =& \e\Biggl[ Wf_{\De_{2,\delta},\De_{1,\delta}}(W) \sum
_{i=1}^n\bigl\{ \e\xi_i^2
I\bigl( \llvert\xi_i \rrvert> 1 \bigr) - \xi_i^2
I\bigl( \llvert\xi_i \rrvert>1\bigr) \bigr\} \Biggr]
\nonumber
\\
&&{} + \e\Biggl\{ W f_{\De_{2,\delta},\De_{1,\delta}}(W) \sum_{i=1}^n
\bigl( \e\barxi_i^{ 2} -\barxi_i^{ 2}
\bigr) \Biggr\}
\nonumber
\\
& \leq&2 \beta_2 + 2 \e\Biggl\{ I(\De_2 -
\De_1 >1 ) \Biggl\llvert\sum_{i=1}^n
\bigl( \e\bar{\xi}_i^{ 2} - \bar{\xi}_i^{ 2}
\bigr) \Biggr\rrvert\Biggr\}
\nonumber
\\
&&{} + \e\Biggl\{ Wf_{\De_{2,\delta},\De_{1,\delta}}(W) I( \De_2 -
\De_1 \leq1 ) \sum_{i=1}^n
\bigl( \e\barxi_i^{ 2} -\barxi_i^{ 2}
\bigr) \Biggr\}
\nonumber
\\
& \leq&2 \beta_2 + \e(\De_2 - \De_1) +
\beta_3
\nonumber\\[-8pt]\\[-8pt]\nonumber
&&{} + \e\Biggl\{ \llvert W\rrvert( 2\delta+ \De_2 -
\De_1 ) I( \De_2 - \De_1 \leq1 ) \sum
_{i=1}^n\bigl( \e\barxi_i^{ 2}
-\barxi_i^{ 2} \bigr) \Biggr\}
\nonumber
\\
& \leq&2 \beta_2 + \e(\De_2 - \De_1) +
\beta_3 + 0.5 \e\bigl\{ ( 2 \delta+\De_2 -
\De_1 )^2 I( \De_2 - \De_1
\leq1 ) \bigr\}
\nonumber
\\
&&{} + 0.5 \e\Biggl[ W^2 \Biggl\{ \sum
_{i=1}^n\bigl( \e\barxi_i^{
2}
-\barxi_i^{ 2} \bigr) \Biggr\}^2 \Biggr]
\nonumber
\\
& \leq&2 \beta_2 + \e(\De_2 - \De_1) +
\beta_3 + 2 \delta^2 + 0.75 \e(\De_2 -
\De_1 )+ 2 \beta_3
\nonumber
\\
& \leq&2.125 \beta_2 + 3.125 \beta_3 + 1.75 \e(
\De_2 - \De_1 ) , \nonumber
\end{eqnarray}
where we used the facts that $\delta\leq1/8$,
\[
\e\Biggl\{ \sum_{i=1}^n \bigl(
\barxi_i^{ 2} - \e\barxi_i^{ 2}
\bigr) \Biggr\} ^2 \leq\beta_3 \quad\mbox{and}\quad \e
\Biggl\{ W \sum_{i=1}^n \bigl( \e
\barxi_i^{ 2} -\barxi_i^{ 2}
\bigr) \Biggr\}^2 \leq4 \beta_3. %
\]
To see this, set $U=\sum_{i=1}^n\eta_i$ with $\eta_i=\barxi_i^{ 2}
-\e\barxi_i^{ 2}$, then by standard calculations,
\begin{eqnarray*}
\e U^2 & =& \sum_{i=1}^n\e
\eta_i^2 \leq\sum_{i=1}^n
\e\barxi_i^{
4} \leq\sum_{i=1}^n
\e\llvert\barxi_i\rrvert^{ 3} = \beta_3
\end{eqnarray*}
and
\begin{eqnarray*}
\e\bigl( W^2 U^2 \bigr) & =& \sum
_{i,j,k,\ell} \e( \xi_i \xi_j
\eta_k \eta_{\ell} ) = \sum_{i=1}^n
\e\bigl( \xi_i^2 \eta_i^2
\bigr) + \sum_{i\neq j } \e\xi_i^2
\e\eta_j^2 + 2 \sum_{i \neq j}
\e\xi_i \eta_i \e\xi_j
\eta_j \leq4 \beta_3.
\end{eqnarray*}

As for $H_{3 2}$, by (\ref{property-2})
%
%
\begin{eqnarray} \label{H32}
H_{3 2} & \leq&\sum_{i=1}^n\e
\int_{\mathds{R}} 2 \min\bigl\{ 1, \bigl( \llvert W\rrvert+ \sqrt{2
\pi}/4 \bigr) \llvert t\rrvert\bigr\} \hatk_i(t) \,dt
\nonumber
\\
& \leq&2 \sum_{i=1}^n\e\int
_{\llvert t\rrvert >1} \hatk_i(t) \,dt + 2 \sum
_{i=1}^n\e\int_{\llvert t\rrvert \leq1} \bigl(
\llvert W\rrvert+\sqrt{2\pi}/4 \bigr) \llvert t\rrvert\hatk_i(t)
\,dt
\nonumber
\\
& \leq&2 \beta_2 + \e\Biggl\{ \bigl( \llvert W\rrvert+\sqrt{2\pi}/4
\bigr) \sum_{i=1}^n\llvert
\xi_i \rrvert\min\bigl( 1, \xi_i^2\bigr)
\Biggr\}
\nonumber\\[-8pt]\\[-8pt]\nonumber
& \leq&2 \beta_2 + \e\Biggl[ \bigl( \llvert W\rrvert+\sqrt{2\pi}/4
\bigr) \Biggl\{ \sum_{i=1}^n\llvert
\xi_i \rrvert I\bigl( \llvert\xi_i \rrvert>1\bigr) +
\sum_{i=1}^n\llvert\bar{
\xi}_i \rrvert^3 \Biggr\} \Biggr]
\nonumber
\\
& \leq&2 \beta_2 + (2+\sqrt{2\pi}/4) (\beta_2 +
\beta_3)
\nonumber
\\
& \leq&4.7 \beta_2 + 2.7 \beta_3,\nonumber
\end{eqnarray}
where we used the inequalities
\begin{eqnarray*}
\e\bigl\{ \llvert W\rrvert\cdot\llvert\xi_i \rrvert I\bigl(
\llvert\xi_i \rrvert>1\bigr) \bigr\} \leq\e\bigl\llvert
W^{(i)}\bigr\rrvert\cdot\e\llvert\xi_i \rrvert I\bigl(
\llvert\xi_i \rrvert>1\bigr) + \e\xi_i^2I
\bigl(\llvert\xi_i \rrvert>1\bigr) \leq2 \e\xi
_i^2I\bigl(\llvert\xi_i \rrvert>1\bigr)
\end{eqnarray*}
and $\e( \llvert W\rrvert \cdot\llvert \bar{\xi}_i\rrvert ^3) \leq\e
\llvert W^{(i)}\rrvert \cdot\e\llvert \bar
{\xi}_i\rrvert ^3 + \e\bar{\xi}_i^{ 4} \leq2 \e\llvert \bar{\xi
}_i\rrvert ^3$.
Combining (\ref{h3}), (\ref{h31}) and (\ref{H32}) yields~(\ref
{new-est}).
\end{pf*}
\end{pf*}

\section{Proof of Theorem \texorpdfstring{\protect\ref{t21}}{2.1}}
\label{proof1sec}

\subsection{Main idea of the proof}

Observe that $V_n$ is close to $1$ and $1+D_{2n}>0$. Remember that we
are interested in a particular type of nonlinear process that can be
written as a linear process plus a negligible remainder. Intuitively,
the leading term of the normalizing factor should be a quadratic
process, say $V_n^2$. The key idea of the proof is to first transform
$V_n ( 1+D_{2n})^{1/2}$ to
$(V_n^2 +1 )/2 + D_{2n}$ plus a small term and then apply the method of
conjugated distributions and the randomized concentration inequality
(\ref{con-ineq-1}).
It follows from the elementary inequalities
\[
1+s/2-s^2/2 \leq(1+s)^{1/2} \leq1+s/2,\qquad s\geq-1
\nonumber
\]
that $(1+D_{2n})^{1/2} \geq1+ \min(D_{2n}, 0)$, which leads to
%
\begin{eqnarray} \label{t21-1}
V_n(1+D_{2n})^{1/2} & \geq& V_n +
V_n \min(D_{2n}, 0)
\nonumber
\\
&\geq& 1+\bigl(V_n^2-1\bigr)/2 -\bigl(V_n^2-1
\bigr)^2 /2 + V_n \min(D_{2n}, 0)
\nonumber\\[-8pt]\\[-8pt]\nonumber
& \geq& V_n^2/2 + 1/2 -\bigl(V_n^2-1
\bigr)^2/2 + \bigl\{ 1+ \bigl(V_n^2-1\bigr)/2
\bigr\} \min(D_{2n}, 0)
\nonumber
\\
& \geq& V_n^2/2 + 1/2 -\bigl(V_n^2-1
\bigr)^2 + \min(D_{2n}, 0).\nonumber
\end{eqnarray}
Using the inequality $2 a b \leq a^2 + b^2$ yields the reverse inequality
\begin{eqnarray*}
V_n(1+D_{2n})^{1/2} \leq(1+D_{2n})/2
+ V_n^2/2 = V_n^2/2 +1/2 +
D_{2n}/2. 
\end{eqnarray*}

Consequently, for any $x > 0$,
%
\begin{eqnarray}\label{up}
\{ T_n \geq x \} & \subseteq& \bigl\{ W_n +
D_{1n} \geq x \bigl( V_n^2/2 + 1/2 -
\bigl(V_n^2-1\bigr)^2 + D_{2n}
\wedge0 \bigr) \bigr\}
\nonumber\\[-8pt]\\[-8pt]\nonumber
& =& \bigl[ xW_n-x^2V_n^2/2 \geq
x^2/2- x \bigl\{ x\bigl(V_n^2-1
\bigr)^2+D_{1n}+ x D_{2n}\wedge0 \bigr\} \bigr]
\end{eqnarray}
and
\begin{equation}
\{ T_n \geq x \} \supseteq\bigl\{ xW_n-x^2V_n^2/2
\geq x^2/2 + x ( x D_{2n} /2-D_{1n} ) \bigr\} .
\label{low}
\end{equation}

\begin{pf*}{Proof of (\ref{cmd-ubd})}
By (\ref{up}), we have for $x
\geq1$,
%
\begin{eqnarray}\label{p21}
&& \P( T_n \geq x )
\nonumber
\\
&&\quad  \leq\P\bigl\{ W_n \geq x V_n (1+ D_{2n}
\wedge0) - D_{1n}, \llvert D_{1n} \rrvert\leq
V_n/4x, \llvert D_{2n} \rrvert\leq1/4x^2
\bigr\}
\nonumber
\\
&&\qquad{} + \P\bigl( \llvert D_{1n} \rrvert/V_n> 1/4x
\bigr) + \P\bigl( \llvert D_{2n} \rrvert>1/ 4x^2 \bigr)
\\
&&\quad \leq\P\bigl( x W_n - x^2 V_n^2/2
\geq x^2/2- x \De_{1n} \bigr) + \P\bigl\{ W_n
\geq(x-1/ 2x )V_n, \bigl\llvert V_n^2-1\bigr
\rrvert>1/2x \bigr\}
\nonumber
\\
&&\qquad{}+ \P\bigl( \llvert D_{1n} \rrvert/V_n> 1/4x
\bigr) + \P\bigl( \llvert D_{2n} \rrvert>1/ 4x^2 \bigr) ,\nonumber
\end{eqnarray}
where
\begin{equation}
\De_{1n}= \min\bigl\{x\bigl(V_n^2-1
\bigr)^2+\llvert D_{1n} \rrvert+ x D_{2n}
\wedge0 , 1/x \bigr\}. \label{Den}
\end{equation}
Consequently, (\ref{cmd-ubd}) follows from the next two propositions. We
postpone the proofs to Section~\ref{pfprop}.
\end{pf*}

%
\begin{proposition} \label{ci}
There exist positive absolute constants $C_1, C_2$ such that
\begin{eqnarray}
\P\bigl(xW_n-x^2V_n^2/2 \geq
x^2/2- x \De_{1n} \bigr) \leq\bigl\{ 1-\Phi(x) \bigr\}
\exp( C_1 L_{n,x} ) (1+C_2 R_{n,x} )
\label{con-u}
\end{eqnarray}
holds for $x\geq1$ satisfying (\ref{c1}) and (\ref{rc}).
\end{proposition}

%
\begin{proposition} \label{tr-ineq}
There exist positive absolute constants $C_3, C_4$ such that
\begin{equation}
\P\bigl( W_n/V_n \geq x-1/2x , \bigl\llvert
V_n^2-1\bigr\rrvert>1/ 2x \bigr) \leq C _3
\bigl\{ 1-\Phi(x) \bigr\} \exp( C_4 L_{n,x} )
L_{n,x} \label{part-1}
\end{equation}
holds for all $x\geq1$.
\end{proposition}

\begin{pf*}{Proof of (\ref{cmd-lbd})} By (\ref{low}),
\begin{equation}
\P( T_n \geq x) \geq\P\bigl( x W_n - x^2
V_n^2/2 \geq x^2/2 + x \De_{2n}
\bigr), \label{p21prime}
\end{equation}
where $\De_{2n}= xD_{2n}/2-D_{1n}$. Then (\ref{cmd-lbd}) follows
directly from the following proposition.

%
\begin{proposition} \label{ci-lbd}
There exist positive absolute constants $C_5, C_6$ such that
\begin{equation}
\P\bigl( x W_n - x^2 V_n^2/2
\geq x^2/2+ x \De_{2n} \bigr) \geq\bigl\{ 1-\Phi(x) \bigr\}
\exp( -C_5 L_{n,x} ) (1- C_6
R_{n,x} ) \label{con-l}
\end{equation}
for $ x\geq1$ satisfying (\ref{c1}) and (\ref{rc}).
\end{proposition}

The proof of Theorem~\ref{t21} is then complete.
\end{pf*}

\subsection{Proof of Propositions \texorpdfstring{\protect\ref{ci}}{5.1}, 
\texorpdfstring{\protect\ref{tr-ineq}}{5.2} and \texorpdfstring{\protect\ref{ci-lbd}}{5.3}}
\label{pfprop}

For two sequences of real numbers $a_n$ and $b_n$, we write $a_n
\lesssim b_n$ if there is a universal constant $C$ such that $a_n \leq
C b_n$ holds for all $n$. Throughout this section, $C, C_1, C_2, \ldots
$ denote positive constants that are independent of $n$. We start with
some preliminary lemmas. The first two lemmas are Lemmas 5.1 and 5.2 in
\cite{JingShaoWang2003}. Let $X$ be a random variable such that $\e X
=0$ and $\e X^2 < \infty$, and set
\[
\delta_1=\e X^2 I\bigl( \llvert X\rrvert>1 \bigr) +\e
\llvert X\rrvert^3 I\bigl( \llvert X\rrvert\leq1 \bigr) .
\]

%
\begin{lemma} \label{l0}
For $0 \leq\lambda\leq4$ and
$0.25 \leq\theta\leq4$, we have
\begin{equation}
\e e^{\lambda X-\theta X^2} = 1+\bigl(\lambda^2/2-\theta\bigr)\e
X^2 +O(1) \delta_1, \label{f-1}
\end{equation}
where $O(1)$ is bounded by an absolute constant.
\end{lemma}

%
\begin{lemma} \label{l2}
Let $Y=X-X^2/2$. Then for $0.25 \leq\lambda\leq4$, we have
\begin{eqnarray}
\e e^{\lambda Y} & =& 1+ \bigl(\lambda^2/2 -\lambda/2 \bigr) \e
X^2 +O(1) \delta_1,
\nonumber
\\
\e\bigl( Y e^{\lambda Y} \bigr) &=& ( \lambda-1/2 ) \e X^2 +O(1)
\delta_1,
\nonumber
\\
\e\bigl( Y^2 e^{\lambda Y} \bigr) &=& \e X^2+O(1)
\delta_1,
\nonumber
\\
\e\bigl( \llvert Y\rrvert^3e^{\lambda Y} \bigr) &=& O(1)
\delta_1 \quad\mbox{and} \quad\bigl\{ \e\bigl( Y
e^{\lambda Y } \bigr) \bigr\}^2 = O(1) \delta_1,
\nonumber
\end{eqnarray}
where the $O(1)$'s are bounded by an absolute constant. In particular,
when $\lambda= 1 $, we have
\begin{equation}
e^{-5.5\delta_1} \leq\e e^{Y} \leq e^{2.65\delta_1}. \label{mgf-est}
\end{equation}
\end{lemma}

%
\begin{lemma} \label{l3}
Let $Y = X - X^2/2$, $Z= X^2 - \e X^2$ and write
\[
\delta_{1 1} = \e X^2 I\bigl(\llvert X\rrvert> 1\bigr) ,\qquad\delta_{1 2} = \e\llvert X\rrvert^3 I\bigl(\llvert X
\rrvert\leq1\bigr) . %
\]
Then
%
\begin{eqnarray}
\bigl\llvert\e\bigl( Z e^{ Y}\bigr) \bigr\rrvert& \leq&4.2
\delta_{1 1} +1.5 \delta_{1 2}, \label
{l3-a}
\\
\e\bigl(Z^2 e^{Y}\bigr) &\leq&4 \delta_{1 1} +2
\delta_{1 2} + 2 \delta_{1
1}^2 , \label{l3-b}
\\
\e\bigl(\llvert Y Z \rrvert e^{Y}\bigr) & \leq&2 \delta_{1 1}
+ \delta_{1 2}, \label
{l3-c}
\\
\e\bigl(\llvert Y\rrvert Z^2 e^{Y}\bigr) &\leq&3.1
\delta_{1 1} + \delta_{1 2} + \delta_{1 1}^2.
\label{l3-d}
\end{eqnarray}
\end{lemma}

\begin{pf}
See the \hyperref[appe]{Appendix}.
\end{pf}

The next lemma provides an estimate of $I_{n,x} $ given in (\ref{deix}).

%
\begin{lemma} \label{In0}
Let $\xi_i$ be independent random variables satisfying (\ref{200}) and
let $L_{n,x}$ be defined as
in (\ref{deix}). Then there exists an absolute positive constant $C$
such that
\begin{eqnarray}
I_{n,x}= \exp\bigl\{ O(1) L_{n,x} \bigr\} \label{In-1}
\end{eqnarray}
for all $x \geq1$, where $\llvert O(1) \rrvert \leq C$.
\end{lemma}

\begin{pf}
Applying (\ref{mgf-est}) in Lemma~\ref{l0} to $X=x \xi_i$ and $Y=X -
X^2/2$ yields (\ref{In-1}) with \mbox{$\llvert O(1) \rrvert \leq5.5$}.
\end{pf}

Our proof is based on the following method of conjugated distributions
or the change of measure technique (Petrov \cite{Petrov1965}), which
can be traced back to Harald Cram\'er in 1938. Let $\xi_i$ be
independent random variables and $g(x)$ be a measurable function
satisfying $\e e^{g(\xi_i)} < \infty$. Let $\hat{\xi}_i$ be
independent random variables with the distribution functions given by
\[
\P( \hat{\xi}_i \leq y ) = { 1 \over\e e^{g(\xi_i)}} \e\bigl\{
e^{g(\xi_i)} I( \xi_i \leq y) \bigr\}. %
\]
Then, for any measurable function $f:\mathbb{R}^n \to\mathbb{R}$ and
any Borel measurable set $C$,
\[
\P\bigl\{ f(\xi_1, \ldots, \xi_n) \in C \bigr\} =
\prod_{i=1}^n \e e^{g(\xi_i)} \times\e
\bigl[ e^{- \sum_{i=1}^n
g(\hat{\xi}_i)} I\bigl\{ f(\hat{\xi}_1, \ldots, \hat{
\xi}_n) \in C \bigr\} \bigr]. %
\]
See, for example, \cite{JingShaoWang2003} and \cite{ShaoZhou2014} for
the applications of the change of measure method in deriving moderate
deviations.

\begin{pf*}{Proof of Proposition~\ref{ci}}
Let $Y_i = g(\xi_{i})
= \xi_{i,x} - \xi_{i,x}^2/2$ with $\xi_{i,x}=x\xi_i$, and let $\hat
\xi_1,\ldots, \hat\xi_n$ be independent random variables with $\hat
\xi_i$ having the distribution function
\begin{eqnarray}
V_i(y)=\e\bigl\{ e^{Y_i} I( \xi_i \leq y)
\bigr\} /\e e^{Y_i},\qquad y\in\mathds{R}.
\nonumber
\end{eqnarray}
Put $\widehat{Y_i}=g(\hat\xi_i) = x\hat\xi_i-x^2\hat\xi_i^{
2}/2$ and recall that $xW_n-x^2V_n^2/2= \sum_{i=1}^nY_i := S_Y $. Then
using the method of conjugated distributions gives
%
\begin{eqnarray} \label{conju-1}
&& \P\bigl( xW_n-x^2V_n^2/2 \geq
x^2/2-x \De_{1n} \bigr)
\nonumber
\\
&&\quad  = \P\Biggl\{ \sum_{i=1}^ng(
\xi_i) \geq x^2 - x \De_{1n}( \xi
_1,\ldots, \xi_n) \Biggr\}
\nonumber\\[-8pt]\\[-8pt]\nonumber
&&\quad = \prod
_{i=1}^n \e e^{Y_i} \times\e\bigl\{
e^{-\widehat S_Y } I\bigl( \widehat S_Y \geq x^2 /2 -x
\widehat{\De}_{1n} \bigr) \bigr\}
\nonumber
\\
&&\quad := I_{n,x} \times H_n,\nonumber
\end{eqnarray}
where $\widehat{S}_Y = \sum_{i=1}^n\widehat Y_i$, $ H_n = \e\{
e^{-\widehat S_Y} I( \widehat S_Y \geq x^2/2-x \widehat{\De}_{1n} ) \}
$ and $\widehat{\De}_{1n} =\De_{1n}(\hat\xi_1,\ldots,\hat\xi_n)$.

Set
\begin{eqnarray}
m_n = \sum_{i=1}^n\e
\widehat{Y}_i ,\qquad\sigma_n^2 = \sum
_{i=1}^n\var( \widehat{Y_i}
)\quad\mbox{and}\quad v_n = \sum_{i=1}^n
\e\llvert\widehat{Y_i}\rrvert^3 .
\nonumber
\end{eqnarray}
Then it follows from the definition of $\hat\xi_i$ that
\begin{eqnarray}
\e\widehat{Y_i} &=& \e\bigl(Y_i e^{ Y_i}
\bigr)/ \e e^{Y_i} ,
\nonumber
\\
\var(\widehat{Y_i}) &=& \e\bigl( Y_i^2
e^{Y_i} \bigr) /\e e^{Y_i} -(\e\widehat{Y_i})^2,
\nonumber
\\
\e\llvert\widehat{Y_i}\rrvert^3 &=& \e\bigl(\llvert
Y_i\rrvert^3 e^{Y_i}\bigr)/\e
e^{Y_i} .
\nonumber
\end{eqnarray}
Applying Lemma~\ref{l3} with $X=x\xi_i$ and $\lambda= 1$ yields
%
\begin{eqnarray} \label{mgf-expan}
\e e^{Y_i} &=& e^{O(1) \delta_{i,x} }, \qquad \e\bigl( Y_i
e^{Y_i}\bigr) = \bigl( x^2/2\bigr) \e\xi_i^2
+O(1) \delta_{i,x},
\nonumber\\[-8pt]\\[-8pt]\nonumber
\e\bigl( Y_i^2 e^{Y_i} \bigr) &=&
x^2 \e\xi_i^2 +O(1) \delta_{i,x},
\qquad\e\bigl( \llvert Y_i\rrvert^3 e^{Y_i}
\bigr) = O(1) \delta_{i,x}\nonumber
\end{eqnarray}
and $\{ \e( Y_i e^{Y_i}) \}^2 = O(1) \delta_{i,x}$. In view of (\ref
{mgf-est}) and (\ref{c1}), using a similar argument as in the proof of
(7.11)--(7.13) in \cite{JingShaoWang2003} gives
%
\begin{eqnarray}
m_n &=& \sum_{i=1}^n\e
\bigl(Y_i e^{Y_i}\bigr)/\e e^{Y_i}
=x^2/2 +O(1)L_{n,x}, \label{mn}
\\
\sigma_n^2 &=& \sum_{i=1}^n
\bigl\{ \e\bigl( Y_i^2e^{Y_i} \bigr) / \e
e^{Y_i} -( \e\widehat{Y}_i)^2 \bigr\} =
x^2+ O(1)L_{n,x}, \label{sigman}
\\
v_n &=& \sum_{i=1}^n\e\bigl(
\llvert Y_i\rrvert^3 e^{Y_i} \bigr)/\e
e^{Y_i} = O(1) L_{n,x}, \label{vn}
\end{eqnarray}
where all of the $O(1)$'s appeared above are bounded by an absolute
constant, say $C_1$. Taking into account the condition (\ref{rc}), we
have $ \sigma_n^2 \geq x^2/2 $, provided the constant $c_1$ in (\ref
{rc}) is sufficiently large, say, no larger than $( 4 C_1)^{-1}$.

Define the standardized sum $\widehat{W} := \widehat{W}_n = (
\widehat S_Y -m_n )/\sigma_n $, and let
\begin{eqnarray*}
\varepsilon_n = \sigma_n^{-1}\bigl(
x^2/2- m_n \bigr), \qquad r_n = \varepsilon
_n + \sigma_n .
\end{eqnarray*}
By (\ref{mn})--(\ref{vn}) and (\ref{rc}) with $c_1 \leq( 4C_1 )^{-1}$,
\begin{eqnarray}
\llvert\varepsilon_n\rrvert&\leq&\sqrt{2}C_1
x^{-1}L_{n,x},\qquad v_n\sigma
_n^{-3} \leq\sqrt{8}C_1 x^{-3}L_{n,x},
\label{error-2}
\\
\llvert r_n-x\rrvert&\leq&\llvert\varepsilon_n\rrvert+
\bigl\llvert\sigma_n^2 - x^2\bigr\rrvert/(
\sigma_n+x) \leq2C_1 x^{-1}L_{n,x}
\leq x/2, \label{error-3}
\end{eqnarray}
which leads to
\begin{eqnarray}
H_n \leq\e\bigl\{ \exp(-\sigma_n \widehat{W}
-m_n ) I( \widehat{W}-\varepsilon_n \geq- x \widehat{
\De}_{1n} /\sigma_n ) \bigr\} \leq H_{1n}+
H_{2n} \label{step-1}
\end{eqnarray}
with $H_{1n} = \e\{ \exp(-\sigma_n \widehat{W} - m_n )I( \widehat
{W}\geq\varepsilon_n ) \}$ and
\[
H_{2n} = \e\bigl\{ \exp(-\sigma_n\widehat{W}
-m_n )I( -x \widehat{\De}_{1n}/\sigma_n \leq
\widehat{W}-\varepsilon_n< 0 ) \bigr\}. %
\]

Denote by $G_n$ the distribution function of $\widehat{W}$, then
$H_{1n}$ reads as
%
\begin{eqnarray} \label{decom-I2}
H_{1n } & =& \int_{\varepsilon_n}^{\infty}
e^{-\sigma_n t-m_n} \,d G_n(t)
\nonumber
\\
&=& e^{-x^2/2}\int_0^{\infty}e^{-\sigma_n s} \,d
G_n(s +\varepsilon_n)
\nonumber\\[-8pt]\\[-8pt]\nonumber
&=& e^{-x^2/2} \biggl\{ \int_0^{\infty}e^{-\sigma_n s }
\,d \bigl\{ G_n( s+\varepsilon_n)-\Phi( s +
\varepsilon_n)\bigr\}
\nonumber
+ \int_0^{\infty}e^{-\sigma_n s } \,d\Phi( s +
\varepsilon_n) \biggr\}
\nonumber
\\
&:=& e^{-x^2/2}(J_{1n}+J_{2n}).\nonumber
\end{eqnarray}
Using integration by parts for the Lebesgue--Stieltjes integral, the
Berry--Esseen inequality, (\ref{error-2}) and the following upper and
lower tail inequalities for the standard normal distribution
\begin{equation}
\frac{t}{1+t^2} e^{-t^2/2} \leq\int_t^\infty
e^{-u^2/2} \,du \leq\frac{1}{t} e^{-t^2/2}\qquad\mbox{for } t> 0,
\label{gaussiantails}
\end{equation}
we have
\begin{eqnarray*}
\llvert J_{1n}\rrvert&\leq&2\sup_{t\in\mathds{R}}\bigl\llvert
G_n(t)-\Phi(t)\bigr\rrvert\leq4 v_n
\sigma_n^{-3} \leq C_2 e^{x^2/2}\bigl\{
1- \Phi(x)\bigr\} x^{-2} L_{n,x} .
\end{eqnarray*}

For $J_{2n}$, by the change of variables we have
\[
\qquad J_{2n} = \frac{e^{-\varepsilon_n^2/2}}{\sqrt{2\pi}}\int_0^{\infty}
\exp\bigl\{ -( \sigma_n+ \varepsilon_n)t-t^2/2
\bigr\} \,dt = \frac
{e^{-\varepsilon_n^2/2}}{\sqrt{2\pi}}\Psi(r_n),
\]
where
\[
\Psi(x) =\frac{ 1-\Phi(x)}{\Phi'(x)}=e^{x^2/2}\int_x^{\infty
}e^{-t^2/2}
\,dt. %
\]
By (\ref{gaussiantails}),
\[
\Psi(s) \geq\frac{s}{1+s^2}\quad\mbox{and}\quad 0< -\Psi'(s) = 1- s
e^{s^2/2}\int_s^\infty e^{-t^2/2} \,dt
\leq\frac{1}{1+s^2}\qquad\mbox{for } s\geq0. %
\]
In view of (\ref{error-3}), $x/2 \leq r_n \leq3x/2$. Consequently,
$\llvert \Psi(r_n) - \Psi(x) \rrvert \leq4\llvert r_n - x\rrvert
/(4+x^2)$, which further
implies that
\begin{eqnarray}
J_{2n} \leq\frac{1}{\sqrt{2\pi}}\biggl\{ \Psi(x) + \frac
{4}{4+x^2}
\llvert r_n -x \rrvert\biggr\} \leq e^{x^2/2}\bigl\{ 1-\Phi(x)
\bigr\} \bigl(1+ C_3 x^{-2}L_{n,x} \bigr).
\nonumber
\end{eqnarray}
By (\ref{decom-I2}) and the above upper bounds for $J_{1n}$ and $J_{2n}$,
\begin{equation}
H_{1n} \leq\bigl\{1-\Phi(x)\bigr\} \bigl( 1+C_4
x^{-2}L_{n,x} \bigr) . \label{I1}
\end{equation}

As for $H_{2n}$, note that $x \widehat{\De}_{1n} \leq1$ by (\ref
{Den}). Therefore,
\begin{equation}
H_{2n} \leq e^{1-x^2/2} \times\P( \varepsilon_n- x
\widehat{\De}_{1n}/\sigma_n \leq\widehat{W}<
\varepsilon_n ). \label{ubd-I2}
\end{equation}
Applying inequality (\ref{con-ineq-1}) to the standardized sum
$\widehat
{W}$ gives
%
\begin{eqnarray} \label{con2}
&& \P(\varepsilon_n- x\widehat{\De}_{1n}/
\sigma_n \leq\widehat{W} \leq\varepsilon_n )
\nonumber\\[-8pt]\\[-8pt]\nonumber
&&\quad \leq17 v_n \sigma_n^{-3} + 5 x
\sigma_n^{-1} \e\llvert\widehat{\De}_{1n}
\rrvert+ 2 x \sigma_n^{-2} \sum
_{i=1}^n \e\bigl\llvert\widehat{Y_i}
\bigl\{ \widehat{\De}_{1n}- \widehat{\De}_{1n}^{(i)}
\bigr\} \bigr\rrvert,\nonumber
\end{eqnarray}
where $\widehat{\De}_{1n}^{(i)} $ can be any random variable that is
independent of $\hat\xi_i$. By (\ref{error-2}), it is readily known
that $v_n \sigma_n^{-3} \leq\sqrt{8} C_1 x^{-3}L_{n,x}$. For the
other two terms, recall that the distribution function of $\hat\xi_i$\vadjust{\goodbreak}
is given by $V_i(y)= \e\{e^{Y_i} I(\xi_i \leq y)\}/ \e e^{Y_i} $
with $Y_i=g(\xi_i)$. Then
%
\begin{eqnarray} \label{t-1}
\e\llvert\widehat{\De}_{1n}\rrvert &=& \int\cdots\int
\De_{1n}(x_1,\ldots, x_n) \,d
V_1(x_1) \cdots d V_n(x_n)
\nonumber
\\
&=& I_{n,x}^{-1} \int\cdots\int\De_{1n}(x_1,
\ldots, x_n) \prod_{i=1}^n
\bigl\{ e^{g(x_i)} \,d F_{\xi_i}(x_i ) \bigr\}
\\
&=& I_{n,x}^{-1} \times\e\bigl( \llvert\De_{1n}
\rrvert e^{\sum_{i=1}^nY_i} \bigr).\nonumber
\end{eqnarray}
It can be similarly obtained that for each $i=1,\ldots,n$,
%
\begin{eqnarray}
\e\bigl\llvert\widehat{Y}_i \bigl\{ \widehat{
\De}_{1n}- \widehat{\De}_{1n}^{(i)} \bigr\} \bigr
\rrvert= I_{n,x}^{-1} \times\e\bigl[ \bigl\llvert
Y_i \bigl\{\De_{1n}-\De_{1n}^{(i)}
\bigr\}\bigr\rrvert e^{ \sum_{j=1}^n Y_j } \bigr]. \label{t-2}
\end{eqnarray}

Assembling (\ref{ubd-I2})--(\ref{t-2}), we obtain from (\ref
{gaussiantails}) that
\begin{eqnarray*}
H_{2n} & \leq& C_5 \bigl\{1-\Phi(x) \bigr\} \Biggl(
x^{-2}L_{n,x} + I_{n,x}^{-1} \times x \e
\bigl( \llvert\De_{1n}\rrvert e^{ \sum_{j=1}^n Y_j } \bigr)
\\
&&{} + I_{n,x}^{-1} \sum
_{i=1}^n \e\bigl[ \bigl\llvert Y_i
\bigl\{ \De_{1n}-\De_{1n}^{(i)}\bigr\} \bigr\rrvert
e^{ \sum_{j=1}^n Y_j } \bigr] \Biggr)
\\
& \leq& C_5 \bigl\{1-\Phi(x) \bigr\} \Biggl[ x^{-2}L_{n,x}
+ I_{n,x}^{-1} \times x \e\bigl( \llvert\De_{1n}
\rrvert e^{ \sum_{j=1}^n Y_j } \bigr)
\\
&&{} + 2 I_{n,x}^{-1} \sum
_{i=1}^n \e\bigl\{\min\bigl(\llvert
\xi_{i,x} \rrvert, 1 \bigr)\bigl\llvert\De_{1n}-
\De_{1n}^{(i)}\bigr\rrvert e^{ \sum_{j\neq i}^n Y_j
} \bigr\} \Biggr],
\end{eqnarray*}
where the last step follows from the inequality $\llvert t -
t^2/2\rrvert
e^{t-t^2/2} \leq2 \min(1, \llvert t\rrvert )$ for $t\in\mathbb{R}$.

Recall that $\De_{1n} \leq x(V_n^2-1)^2+\llvert D_{1n} \rrvert +
x\llvert D_{2n} \rrvert $. To
finish the proof of (\ref{con-u}), we only need to consider the
contribution from $x (V_n^2 -1)^2$. For notational convenience, let
$Z_i= \xi_i^2- \e\xi_i^2$ for $1\leq i\leq n $, such that $V_n^2-1=
\sum_{i=1}^nZ_i $ and
\begin{eqnarray*}
\bigl(V_n^2-1\bigr)^2 -\bigl\{
\bigl(V_n^2-1\bigr)^2 \bigr\}^{(i)} =
Z_i^2+2 Z_i \cdot\sum
_{j\neq i} Z_j.
\end{eqnarray*}
By Lemma~\ref{De1}, (\ref{ubd-I2}) and (\ref{con2}),
\begin{equation} \label{I2}
H_{2n} \leq C_6 \bigl\{ 1-\Phi(x) \bigr\} \bigl\{
R_{n,x} + x^{-2} L_{n,x}(1+ L_{n,x})
e^{C_7 \max_{i} \delta_{i,x}} \bigr\}.
\end{equation}

Together, (\ref{conju-1}), (\ref{step-1}), (\ref{I1}), (\ref{I2})
and Lemma
\ref{In0} prove (\ref{con-u}).
\end{pf*}

%
\begin{lemma} \label{De1}
For $x\geq1$, we have
\begin{equation}
\e\bigl\{ \bigl(V_n^2-1\bigr)^2
e^{ \sum_{j=1}^n Y_j } \bigr\} \lesssim I_{n,x} \times x^{-4}
L_{n,x}( 1+ L_{n,x}) \label{est-1}
\end{equation}
and
\begin{eqnarray}
\sum_{i=1}^n\e\biggl\{ \biggl\llvert
Y_i \biggl( Z_i^2+2 Z_i
\sum_{j\neq i} Z_j \biggr) \biggr\rrvert
e^{\sum_{j=1}^n Y_j} \biggr\} \lesssim I_{n,x} \times x^{-4}L_{n,x}(
1+L_{n,x}) . \label{est-2}
\end{eqnarray}
\end{lemma}

\begin{pf} 
Recall that $V_n^2-1=\sum_{i=1}^nZ_i$. By independence,
%
\begin{eqnarray} \label{l31}
&& \e\Biggl\{ \Biggl( \sum_{i=1}^nZ_i
\Biggr)^2 e^{\sum_{j=1}^n Y_j} \Biggr\}
\nonumber
\\
&&\quad = \sum_{i=1}^n\e\bigl(
Z_i^2 e^{Y_i}\bigr) \e e^{\sum_{j\neq i} Y_j} + \sum
_{i\neq j} \e\bigl( Z_i e^{Y_i}
\bigr) \cdot\e\bigl( Z_j e^{Y_j} \bigr) \cdot\e
e^{\sum_{k=1, k\neq i,j}^n Y_k}
\\
&&\quad = I_{n,x} \Biggl\{ \sum_{i=1}^n
\e\bigl( Z_i^2 e^{Y_i}\bigr)/\e
e^{Y_i} + \sum_{i\neq j} \e\bigl(
Z_i e^{Y_i}\bigr) \cdot\e\bigl( Z_j
e^{Y_j}\bigr) / \bigl( \e e^{Y_i} \e e^{Y_j} \bigr)
\Biggr\}.\nonumber
\end{eqnarray}
It follows from Lemma~\ref{l3} that $\llvert \e( Z_i e^{Y_i} )
\rrvert \lesssim
x^{-2}\delta_{i,x}$ and $\e( Z_i^2 e^{Y_i}) \lesssim x^{-4} ( \delta
_{i,x}+\delta_{i,x}^2 )$. Substituting these into (\ref{l31}) proves
(\ref{est-1}) in view of (\ref{mgf-est}).

Again, applying Lemma~\ref{l3} gives us
\begin{eqnarray*}
\e\bigl( \llvert Z_i Y_i\rrvert e^{Y_i}
\bigr) \lesssim x^{-2} \delta_{i,x}\quad\mbox{and}\quad \e\bigl(
Z_i^2 \llvert Y_i\rrvert e^{Y_i}
\bigr) \lesssim x^{-4}\bigl( \delta_{i,x}+ \delta
_{i,x}^2 \bigr),
\end{eqnarray*}
which together with H\"older's inequality imply
\begin{eqnarray}
&& \sum_{i=1}^n\e\biggl\{ \biggl\llvert
Y_i \biggl( Z_i^2+2 Z_i
\sum_{j\neq
i} Z_j \biggr) \biggr\rrvert
e^{\sum_{j=1}^n Y_j} \biggr\}
\nonumber
\\
&&\quad  \lesssim I_{n,x} \times x^{-4 }L_{n,x}(1+L_{n,x})
\nonumber
\\
&&\qquad{} + 2 \sum_{i=1}^n\e\bigl( \llvert
Z_i Y_i\rrvert e^{Y_i} \bigr) \biggl\{ \e
\biggl(\sum_{j \neq i} Z_j
\biggr)^2 e^{\sum_{j\neq i}Y_j} \biggr\}^{1/2} \cdot\bigl( \e
e^{\sum_{j\neq i}Y_j} \bigr)^{1/2}
\nonumber
\\
&&\quad \lesssim I_{n,x} \times x^{-4} L_{n,x}(1+L_{n,x}),
\nonumber
\end{eqnarray}
where we use (\ref{est-1}) in the last step. This completes the proof
of (\ref{est-2}).
\end{pf}

\begin{pf*}{Proof of Proposition~\ref{tr-ineq}} This proof is
similar to the argument used in \cite{Shao1999}. First, consider the
following decomposition:
%
\begin{eqnarray}\label{3-p}
&& \P\bigl( W_n/V_n \geq x-1/2x , \bigl\llvert
V_n^2-1\bigr\rrvert>1/ 2x \bigr)
\nonumber
\\
&&\quad \leq\P\bigl\{ W_n/V_n \geq x-1/2x, (1+1/2x
)^{1/2} < V_n\leq4 \bigr\}
\nonumber
\\
&&\qquad{} + \P\bigl\{ W_n/V_n \geq x-1/2x, V_n
< (1-1/2x )^{1/2} \bigr\}
\\
&&\qquad{} + \P( W_n/V_n \geq
x-1/2x, V_n>4 )
\nonumber
\\
&&\quad := \sum_{\nu=1}^3 \P\bigl\{
(W_n, V_n ) \in\mathcal{E}_{\nu} \bigr\},\nonumber
\end{eqnarray}
where $\mathcal{E}_{\nu} \subseteq\mathbb{R}\times\mathbb{R}^+$,
$1\leq{\nu}\leq3$ are given by
\begin{eqnarray*}
\mathcal{E}_1 &=& \bigl\{ (u,v) \in\mathbb{R}\times\mathbb{R}^+:
u/v \geq x-1/2x , \sqrt{1+1/2x} <v \leq4 \bigr\},
\\
\mathcal{E}_2 &=& \bigl\{ (u,v) \in\mathbb{R}\times\mathbb{R}^+: u/v
\geq x-1/2x , v <\sqrt{1-1/2x} \bigr\},
\nonumber
\\
\mathcal{E}_3 &=& \bigl\{ (u,v) \in\mathbb{R}\times\mathbb{R}^+: u/v
\geq x-1/2x, v >4 \bigr\}.
\end{eqnarray*}

To bound the probability $\P\{ (W_n, V_n) \in\mathcal{E}_1\}$, put
$t_1 =x\sqrt{1+1/2x}$ and $\lambda_1 =t_1 (x-1/2x)/8$. By Markov's inequality,
\begin{eqnarray*}
\P\bigl\{ (W_n,V_n)\in\mathcal{E}_1 \bigr
\} \leq x^2 e^{ -\inf_{(u,v)\in
\mathcal{E}_1}(t_1 u-\lambda_1 v^2) } \e\bigl\{ \bigl(V_n^2-1
\bigr)^2 e^{t_1 W_n-\lambda_1 V_n^2} \bigr\},
\end{eqnarray*}
where it can be easily verified that
\begin{eqnarray*}
\inf_{(u,v)\in\mathcal{E}_1}\bigl(t_1 u-\lambda_1
v^2\bigr) = x^2+x/2-\lambda_1(1+1/x)-1/2-1/
4 x .
\end{eqnarray*}

However, recall that $V_n^2-1=\sum_{i=1}^nZ_i$ with $Z_i= \xi_i^2- \e
\xi_i^2$, it follows from the independence and (\ref{f-1}) that
%
\begin{eqnarray} \label{t-la}
&& \e\bigl\{ \bigl(V_n^2-1\bigr)^2
e^{t_1 W_n-\lambda_1 V_n^2}\bigr\}
\nonumber
\\
&&\quad  = \sum_{i=1}^n\e\bigl(
Z_i^2 e^{t_1\xi_i-\lambda_1 \xi_i^2} \bigr) \times\prod
_{j\neq i} \e\bigl(e^{t_1 \xi_j -\lambda_1 \xi_j^2}\bigr)
\nonumber\\[-8pt]\\[-8pt]\nonumber
&&\qquad{} + \sum_{i \neq j} \e\bigl( Z_i
e^{t_1\xi_i -\lambda_1 \xi_i^2} \bigr) \e\bigl( Z_j e^{t_1\xi_j-\lambda
_1 \xi_j^2} \bigr) \times
\prod_{k \neq i,
j} \e\bigl( e^{t_1 \xi_k -\lambda_1 \xi_k^2} \bigr)
\nonumber
\\
&&\quad \lesssim x^{-4} L_{n,x} ( 1+L_{n,x} ) \exp\bigl(
t_1^2/2-\lambda_1+ C L_{n,x}
\bigr),\nonumber
\end{eqnarray}
where we use the fact $t_1^2/2-\lambda_1>0$. Consequently,
%
\begin{eqnarray} \label{p1}
&& \P\bigl\{ (W_n,V_n)\in\mathcal{E}_1
\bigr\}/\bigl\{ 1-\Phi(x)\bigr\}
\nonumber\\[-8pt]\\[-8pt]\nonumber
&&\quad \lesssim x^{-2} L_{n,x}( 1+L_{n,x}) \exp( -3x/8+C
L_{n,x}) \lesssim L_{n,x} \exp( -3 x/8+ C L_{n,x}
).
\end{eqnarray}

Likewise, we can bound the probability $\P\{(W_n,V_n) \in\mathcal
{E}_2\}$ by using $(t_2 , \lambda_2)$ instead of $(t_1, \lambda_1)$,
given by
\[
t_2 = x\sqrt{1-1/2 x }, \qquad\lambda_2
=2x^2-1. %
\]
Note that $
\inf_{(u,v)\in\mathcal{E}_2}( t_2 u- \lambda_2 v^2)=x^2-x/2-1/2+1/
4x -\lambda_2 (1-1/2x )$. Together with (\ref{t-la}), this yields
%
\begin{eqnarray} \label{p2}
&& \P\bigl\{ (W_n,V_n)\in\mathcal{E}_2 \bigr
\} /\bigl\{1-\Phi(x)\bigr\}
\nonumber\\[-8pt]\\[-8pt]\nonumber
&&\quad \lesssim x^{-2} L_{n,x}( 1+L_{n,x}) \exp( -3
x/4+C L_{n,x} ) \lesssim L_{n,x}\exp( -3 x/ 4 +
CL_{n,x} ) .
\end{eqnarray}

For the last term $\P\{ (W_n, V_n) \in\mathcal{E}_3 \}$, we use a
truncation technique and the probability estimation of binomial
distribution. Let $\widehat{W}_n = \sum_{i=1}^n\xi_i I( x\xi_i \leq
a_0 ) $, where $a_0$ is an absolute constant to be determined (see
(\ref{j3})). Observe that
\begin{eqnarray*}
\P\bigl\{ (W_n,V_n)\in\mathcal{E}_3 \bigr\}
&\leq&\P\Biggl( \widehat{W}_n \geq2x-1/x , \sum
_{i=1}^n\xi_i^2 I\bigl( x
\llvert\xi_i \rrvert\leq1\bigr) \geq3 \Biggr)
\\
&&{} + \P\Biggl( \widehat{W}_n \geq2x-1/x , \sum
_{i=1}^n\xi_i^2 I\bigl(x
\llvert\xi_i \rrvert> 1 \bigr) \geq13 \Biggr)
\\
&&{} + \P\Biggl( \sum_{i=1}^n
\xi_i I\{x\xi_i >a_0 \} \geq(x-1/2x
)V_n/2 \Biggr)
\\
&:=& J_{3n}+J_{4n}+J_{5n}.
\end{eqnarray*}
Let
\[
\bar{V}^2_n=\sum_{i=1}^n
\bar{\xi}_i^{ 2}\qquad\mbox{with } \bar{\xi}_i=
\xi_i I\bigl(x \llvert\xi_i \rrvert\leq1 \bigr), 1
\leq i\leq n, %
\]
such that
\begin{eqnarray*}
J_{3n} &=& \P\bigl( \widehat{W}_n \geq2x-1/x ,
\bar{V}_n^2 \geq3 \bigr)
\leq (\sqrt{e}/4) e^{-x^2} \e\bigl\{ \bigl(\bar{V}_n^2-1
\bigr)^2 e^{x \widehat
{W}_n/2} \bigr\}
\\
&\leq& e^{-x^2} \Biggl( \e\Biggl[ \Biggl\{ \sum
_{i=1}^n\bigl( \bar{\xi}_i^{ 2}
- \e\bar{\xi}_i^{ 2} \bigr) \Biggr\}^2
e^{x \widehat{W}_n/2} \Biggr] + x^{-4}L_{n,x}^2 \e
e^{x\widehat{W}_n/2} \Biggr).
\end{eqnarray*}
Noting that $\e\{ \xi_i I( x\xi_i \geq a_0 )\} =- \e\{ \xi_i I (
x\xi_i >a_0 )\} \leq0$ for every $i$, and
\[
e^s \leq1+s+s^2/2+ \llvert s\rrvert^3
e^{\max(s,0)}/6\qquad\mbox{for all } s, %
\]
we obtain
%
\begin{eqnarray}\label{j3-0}
\e e^{x \widehat{W}_n/2} &\leq&\prod_{i=1}^n
\biggl[ 1+\frac{x^2}{8}\e\xi_i^2 +\frac{e^{a_0/2}x^3}{48}
\e\bigl\{ \llvert\xi_i \rrvert^3 I\bigl(\llvert x
\xi_i\rrvert\leq a_0\bigr) \bigr\} \biggr]
\nonumber
\\
&\leq&\prod_{i=1}^n \biggl\{ 1+
\frac{x^2}{8}\e\xi_i^2 +\frac{e^{a_0/2}x^3}{48} \e\llvert
\xi_i \rrvert^3 I\bigl( x \llvert\xi_i
\rrvert\leq1 \bigr)
\nonumber
\\[-8pt]
\\[-8pt]
\nonumber
&&{} + \frac{a_0 e^{a_0/2}x^2}{48}\e\xi_i^2 I
\bigl(x \llvert\xi_i \rrvert> 1 \bigr) \biggr\}\quad\qquad
\\
&\leq&\exp\bigl\{ x^2/8 + O(1)L_{n,x} \bigr\}.\nonumber
\end{eqnarray}
Similar to the proof of (\ref{t-la}), it follows that
\begin{equation}
J_{3n} \lesssim x^{-4}L_{n,x}(
1+L_{n,x}) \exp\bigl\{ -7x^2/8 +O(1) L_{n,x}
\bigr\}. \label{j1}
\end{equation}

To bound $J_{4n}$, let $\widehat{W}_n^{(i)}=\widehat{W}_n-\xi_iI (
x\xi_i\leq a_0 ) $, then applying (\ref{j3-0}) gives, for any $i$,
\[
\e e^{x \widehat{W}_n^{(i)}/2} \leq\exp\bigl\{ x^2/8+ O(1)L_{n,x}
\bigr\}.
\]
Subsequently,
%
\begin{eqnarray}\label{j2}
J_{4n} &\leq& (\sqrt{e}/13) e^{-x^2} \sum
_{i=1}^n\e\bigl\{ \xi_i^2
e^{(x/2) \xi_i I( x\xi_i \leq a_0)} I\bigl( x\llvert\xi_i \rrvert>1
\bigr) \bigr\}
\times\e e^{x \widehat{W}_n^{(i)}/2}
\nonumber\\[-8pt]\\[-8pt]\nonumber
&\leq& \bigl( \sqrt{e^{1+a_0}}/13\bigr) x^{-2}L_{n,x}
\exp\bigl\{ -7x^2/8 +O(1)L_{n,x} \bigr\}.
\end{eqnarray}

Finally, we study $J_{5n}$. By Cauchy's inequality,
%
\begin{eqnarray} \label{j3}
J_{5n} &\leq& \P\Biggl\{ \sum_{i=1}^nI
\bigl( \llvert x\xi_i\rrvert>a_0 \bigr) \geq( x-1/2x
)^2/4 \Biggr\}
\nonumber
\\
&\leq& \frac{4e^{-(x-1/2x )^2}}{(x-1/2x )^2} \sum_{i=1}^n\e
\bigl\{ e^{4 I( \llvert x\xi_i\rrvert >a_0)} I\bigl( \llvert x \xi
_i\rrvert
>a_0 \bigr) \bigr\} \times\prod_{j\neq i} \e
e^{4 I( \llvert x\xi_j\rrvert >a_0 )}
\nonumber
\\
& \lesssim & x^{-2} e^{-x^2} \sum_{i=1}^ne^4
\P\bigl( \llvert x \xi_i\rrvert> a_0 \bigr) \times
\prod_{j \neq i} \bigl\{ 1+ e^4 \P\bigl(
\llvert x \xi_j\rrvert> a_0 \bigr) \bigr\}
\\
& \lesssim & a_0^{-2} \exp\bigl\{ \bigl( e^4
a_0^{-2} - 1\bigr) x^2 \bigr\} \sum
_{i=1}^n\e\xi_i^2 I\bigl(
x\llvert\xi_i \rrvert>1 \bigr)
\nonumber
\\
& \lesssim & x^{-2}L_{n,x}\exp\bigl( -x^2/2-x^2/22
\bigr)\nonumber
\end{eqnarray}
by letting $a_0=11$.

Adding up (\ref{j1})--(\ref{j3}), we get
\begin{eqnarray*}
\P\bigl\{ (W_n,V_n)\in\mathcal{E}_3 \bigr\}
\lesssim\bigl\{ 1-\Phi(x) \bigr\} L_{n,x}\exp( C L_{n,x}).
\end{eqnarray*}
This, together with (\ref{p1}) and (\ref{p2}) yields (\ref{part-1}).
\end{pf*}

\begin{pf*}{Proof of Proposition~\ref{ci-lbd}}
Retain the notation in the proof of Proposition~\ref{ci}, and recall
that $\De_{2n} = x D_{2n}/2 - D_{1n}, \widehat W= \sum_{i=1}^n\widehat
Y_i$. Analogous to (\ref{conju-1}) and (\ref
{step-1}), we
see that
%
\begin{eqnarray} \label{c-lbd}
&& \P\bigl( xW_n-x^2V_n^2/2 \geq
x^2/2 + x \De_{2n} \bigr)
\nonumber
\\
&&\quad  = I_{n,x} \e\bigl\{ e^{- \widehat W } I\bigl( \widehat W \geq
x^2/2 +x \widehat{\De}_{2n} \bigr) \bigr\}
\\
&&\quad \geq I_{n,x} \bigl[ \e\bigl\{ \exp(-\sigma_n \widehat{W}
- m_n ) I( \widehat{W} \geq\varepsilon_n ) \bigr\}\nonumber
\\
&&\qquad{} -\e\bigl\{ \exp(-\sigma_n\widehat{W} - m_n
) I( \varepsilon_n \leq\widehat{W} <\varepsilon_n+ x
\widehat{\De}_{2n}/\sigma_n ) \bigr\} \bigr]
\nonumber
\\
&&\quad \geq I_{n,x} \biggl\{ \int_{\varepsilon_n}^{\infty}e^{-\sigma_n t
-m_n }
\,dG_n(t) - e^{-x^2/2} \P(\varepsilon_n \leq
\widehat{W} < \varepsilon_n+ x \widehat{\De}_{2n}/
\sigma_n ) \biggr\}
\nonumber
\\
&&\quad := I_{n,x} \bigl( H_{1n}-H _{2n}'
\bigr),\nonumber
\end{eqnarray}
for $H_{1n}$ given in (\ref{step-1}), and where $\varepsilon_n =
\sigma
_n^{-1}( x^2/2- m_n )$,
\[
\widehat{\De}_{2n} = \De_{2n}(\hat\xi_1,
\ldots,\hat\xi_n),\qquad H_{2n}'=e^{-x^2/2 }
\P(\varepsilon_n \leq\widehat{W} < \varepsilon_n+ x
\widehat{\De}_{2n}/\sigma_n ). %
\]

Following the proof of (\ref{I1}), it can be similarly obtained that
\begin{equation}
H_{1n} \geq\bigl\{ 1- \Phi(x) \bigr\} \bigl( 1- C x^{-2}
L_{n,x} \bigr). \label{kn1-01}
\end{equation}
Replacing $\widehat\De_{1n}$ with $\widehat\De_{2n}$ in (\ref
{ubd-I2}) and using the same argument that leads to (\ref{I2}) implies
\begin{equation}
H_{2n}' \leq C \bigl\{ 1-\Phi(x) \bigr\}
R_{n,x}. \label{kn2-02}
\end{equation}

Substituting (\ref{In-1}), (\ref{kn1-01}) and (\ref{kn2-02}) into
(\ref{c-lbd})
proves (\ref{con-l}).
\end{pf*}

\section{Proof of Theorem \texorpdfstring{\protect\ref{t11}}{3.1}}
\label{proof2sec}

Throughout this section, we use $C, C_1, C_2, \ldots$ and $c, c_1,
c_2, \ldots$ to denote positive constants that are independent of $n$.

\subsection{Outline of the proof}

Put $\tilde{h}=(h-\theta)/\sigma$ and $\tilde{h}_1 =(h_1 -\theta
)/\sigma$, such that $\tilde{h}_1(x)=\e\{ \tilde{h}(X_1, X_2,
\ldots, X_m) \mid X_1=x \} $ and $\tilde{h}_1(X_1), \ldots, \tilde
{h}_1(X_n)$ are i.i.d. random variables with zero means and unit
variances. Using this notation, condition (\ref{k-c}) can be written as
\begin{equation}
\tilde{h}^2(x_1,\ldots,x_m) \leq
c_0 \Biggl\{ \tau+\sum_{i=1}^m
\tilde{h}_1^2(x_i) \Biggr\}.
\label{k-cprime}
\end{equation}

By the scale-invariance property of Studentized $U$-statistics, we can
replace, respectively, $h$ and $h_1$ with $\tilde{h}$ and $\tilde
{h}_1$, which does not change the definition of $T_n$. For ease of
exposition, we still use $h$ and $h_1$ but assume without loss of
generality that $ \e h_{1 i} =0$ and $\e h_{1 i}^2 =1$, where $h_{1 i}
:= h_1(X_i)$ for $i=1,\ldots,n $.

For $s_1^2$ given in (\ref{s1}), observe that
\[
\frac{(n-m)^2 }{ (n-1)} s_1^2 = \sum
_{i=1}^n(q_i-U_n)^2=
\sum_{i=1}^nq_i^2-nU_n^2
. %
\]
Define
\begin{equation}
T_n^{\ast}=\frac{\sqrt{n}}{m s_1^{\ast}}U_n,\qquad
s_1^{\ast
2}=\frac{(n-1)}{(n-m)^2}\sum
_{i=1}^nq_i^2,
\end{equation}
then by the definition of $T_n$,
\begin{eqnarray*}
T_n = T_n^{\ast} \Big/ \biggl(1-\frac{m^2(n-1)}{(n-m)^2}
T_n^{\ast
2} \biggr)^{1/2}, 
\end{eqnarray*}
such that for any $x \geq0$,
\begin{eqnarray}
\{ T_n \geq x \} = \bigl\{ T_n^{\ast} \geq x/
\bigl( 1+x^2m^2(n-1)/(n-m)^2
\bigr)^{1/2} \bigr\}. \label{tt}
\end{eqnarray}
Therefore, we only need to focus on $T_n^{\ast}$, instead of $T_n$.

To reformulate $T_n^*=\sqrt{n}U_n/(m s_1^*)$ in the form of (\ref
{stu-stat}), set
\begin{equation}
W_n = \sum_{i=1}^n
\xi_i,\qquad V_n^2= \sum
_{i=1}^n\xi_i^2,
\label{WVxi}
\end{equation}
where $\xi_i = n^{-1/2} h_{1 i}$ for $1\leq i \leq n$. Moreover, put
\begin{equation}
r(x_1, \ldots, x_m) = h(x_1, \ldots,
x_m) - \sum_{i=1}^m
h_1(x_i). \label{rm}
\end{equation}
For $U_n$, using Hoeffding's decomposition gives $\sqrt{n} U_n /m =
W_n + D_{1n}$, where
\begin{equation}
D_{1n} = { \sqrt{n} \over m {n \choose m}} \sum_{1 \leq i_1 < i_2<
\cdots< i_m \leq n}
r(X_{i_1}, \ldots, X_{i_m}). \label{D1}
\end{equation}
However, a direct calculation shows that $s_1^2 = V_n^2 (1+D_{2n}) $, where
%
\begin{eqnarray}
(n-1) D_{2n} & =& 1+ V_n^{-2} \Biggl\{
\frac{1}{ {n-2 \choose
m-1}^2}\La_n^2 + \frac{(m-1)\{ (m+1)n- 2 m\} n }{(n-m)^2 }W_n^2
\nonumber\\[-8pt]\label{d2-01} \\[-8pt]\nonumber
&&{}+ \frac{2\sqrt{n}}{ {n-2 \choose m-1}
} \sum
_{i=1}^n\xi_i \psi_i +
\frac{2m(m-1)n}{(n-m)^2} W_n D_{1n} \Biggr\},
\\
\La_n^2 & =& \sum_{i=1}^n
\psi_i^2,\qquad\psi_i = \mathop{\sum
_{ 1\leq\ell_1< \cdots< \ell_{m-1}\leq n}}_{\ell
_j \neq i, j=1, \ldots, m-1 } r(X_i, X_{\ell_1}, \ldots,
X_{\ell_{m-1}}). \label{PsiGa}
\end{eqnarray}
In particular, (\ref{d2-01}) generalizes (2.5) in \cite
{LaiShaoWang2011} for $m=2$. Combining the above decompositions of
$U_n$ and $s_1^2$, we obtain
\begin{eqnarray}
T_n^{\ast}=\frac{W_n+D_{1n}}{ V_n(1+D_{2n})^{1/2}}. \label{stu-u2}
\end{eqnarray}

To prove (\ref{t11a}), by (\ref{tt}), it is sufficient to show that there
exists a constant $C >1$ independent of $n$ such that
\begin{equation}
\qquad\P\bigl(T_n^{\ast}\geq x \bigr) \leq\bigl\{1-\Phi(x)
\bigr\} e^{C L_{n,1+x}} \biggl\{ 1+ C ( \sqrt{ a_m} +
\sigma_h )\frac{(1+x)^3}{\sqrt{n}} \biggr\} \label{cmd-ubd-u}
\end{equation}
and
\begin{equation}
\qquad\P\bigl(T_n^{\ast}\geq x \bigr) \geq\bigl\{1-\Phi(x)
\bigr\} e^{-C L_{n, 1+ x}} \biggl\{ 1- C ( \sqrt{a_m } +
\sigma_h ) \frac{(1+x)^3}{n^{1/2}} \biggr\} \label{cmd-lbd-u}
\end{equation}
hold uniformly for
\begin{equation}
0\leq x \leq C^{-1} \min\bigl\{ (\sigma/\sigma_p)
n^{1/2-1/p}, (n/a_m)^{1/6} \bigr\} , \label{hypo1}
\end{equation}
where $L_{n,x} = n \e\xi_{1,x}^2 I( \llvert \xi_{1,x}\rrvert >1 ) + n
\e\llvert \xi
_{1,x}\rrvert ^3 I(\llvert \xi_{1,x}\rrvert \leq1 )$ with $\xi
_{i,x}=x\xi_i$ for $x\geq1$.

The main strategy of proving (\ref{cmd-ubd-u}) and (\ref{cmd-lbd-u})
is to
first partition the probability space into two parts, say $\mathcal
{G}_{n,x}$ and its complement $\mathcal{G}_{n,x}^c$ such that $\P
(\mathcal{G}_{n,x}^c)$ is sufficiently small, then find a tight upper
bound for the tail probability of $\llvert D_{2n} \rrvert $ on
$\mathcal{G}_{n,x}$,
and finally apply Theorem~\ref{t21}.

First, by Lemma 3.3 of \cite{LaiShaoWang2011}, $\P( V_n^2 \leq
\sigma^2/2 ) \leq\exp\{ -n/(32 a^2) \}$ for all $n\geq1$, where $a
>0$ is such that $\e h_{1 i}^2 I( \llvert h_{1 i}\rrvert \geq a \sigma
) \leq
\sigma^2/4$. In particular, we take
\[
a =4^{1/(p-2)}(\sigma_p/\sigma)^{p/(p-2)} \leq(2
\sigma_p/\sigma)^{p/(p-2)}. %
\]
Then it follows from the inequality that $\sup_{2<p\leq3} \sup_{s\geq
0}( s^{ p/2-1} e^{-s} ) \leq1$ and (\ref{gaussiantails})
that (recall that $\sigma^2=1$)
\begin{equation}
\P\bigl( V_n^2 \leq1/2 \bigr) \leq C_1
\bigl\{1-\Phi(x)\bigr\} (\sigma_p/\sigma)^{p} (1+x)
n^{1-p/2 } \label{pbest1}
\end{equation}
for all $0\leq x \leq c_1 (\sigma/\sigma_1) n^{ p/2 -1 }$. We can
therefore regard $\{ V_n^2 \}_{n\geq1}$ as a sequence of positive
random variables that are uniformly bounded away from zero. For
$W_n/V_n$, applying Lemma 6.4 in \cite{JingShaoWang2003} implies that
for every $t>0$,
\begin{equation}
\P\bigl\{ \llvert W_n\rrvert\geq t (4 + V_n) \bigr\}
\leq4\exp\bigl(-t^2/2\bigr) . \label{pbest2}
\end{equation}
In view of (\ref{pbest1}) and (\ref{pbest2}), define the subset
\begin{equation}
\mathcal{G}_{n,x}=\bigl\{ \llvert W_n\rrvert\leq\sqrt{x}
n^{1/4}(4+V_n), V_n^2 \geq1/2
\bigr\}, \label{setGnx}
\end{equation}
such that
\begin{equation}
\P\bigl( \mathcal{G}_{n,x}^c \bigr) \leq C_2
\bigl\{1-\Phi(x)\bigr\} ( \sigma_p/\sigma)^{p} (1+x )
n^{1-p/2} \label{pbsetest}
\end{equation}
holds uniformly for
\begin{equation}
0 \leq x \leq c_2 \min\bigl\{ (\sigma/\sigma_1)
n^{ p/2- 1} , \sqrt{n} \bigr\} . \label{xcond1}
\end{equation}

Next, we restrict our attention to the subset $\mathcal{G}_{n,x}$.
Recall the definition of $D_{2n}$ in (\ref{d2-01}). For any
$\varepsilon
>0$, we have
\begin{equation}
\Biggl\llvert\sum_{i=1}^n
\xi_i \psi_i \Biggr\rrvert\leq( 4
\varepsilon)^{-1} V_n^2 + \varepsilon
\La_n^2 . \label{roughbd}
\end{equation}
In particular, taking $\varepsilon=\sigma/ ( x n^{m-1}\sigma_h )$
for $\sigma_h^2 $ as in (\ref{roughbd}) yields
%
\begin{eqnarray}\label{de-ubd}
\llvert D_{2n} \rrvert& \leq& C_3 \bigl\{
\sigma_h x n^{-1/2 } + (\sigma_h
x)^{-1} n^{3/2-2m} V_n^{-2}
\La_n^2
\nonumber\\[-8pt]\\[-8pt]\nonumber
&&{} + n^{-1} (W_n/V_n)^2
+ n^{-1} V_n^{-2} \llvert W_n \rrvert
\llvert D_{1n} \rrvert\bigr\}.
\end{eqnarray}
In addition to the subset $\mathcal{G}_{n,x}$ given in (\ref
{setGnx}), put
\begin{equation}
\mathcal{E}_{n,x} = \mathcal{G}_{n,x} \cap\bigl\{ \llvert
D_{1n} \rrvert/V_n \leq1/ 4x \bigr\}. \label{setSnx}
\end{equation}
Together, (\ref{de-ubd}) and (\ref{setSnx}) imply that
\begin{equation}
\llvert D_{2n} \rrvert\leq C_4 \bigl\{
\sigma_h x n^{-1/2} + (\sigma_h
x)^{-1} n^{3/2-2m} \La_n^2 \bigr\} :=
D_{3n} \label{D3}
\end{equation}
holds on $\mathcal E_{n,x}$ for all $1\leq x \leq\sqrt{n}$.

\begin{pf*}{Proof of (\ref{cmd-ubd-u})}
By (\ref{cmd-ubd}), Remark~\ref
{r20}, (\ref{stu-u2}), (\ref{de-ubd}) and condition (\ref
{xcond1}), we have
%
\begin{eqnarray} \label{ubddec}
\P\bigl(T_n^{\ast} \geq x \bigr) & \leq&\bigl\{ 1-\Phi(x)
\bigr\} e^{C_5L_{n,x}} (1+C_6 R_{n,x} )
\nonumber\\[-8pt]\\[-8pt]\nonumber
&&{} + \P\bigl(\llvert D_{1n} \rrvert/V_n \geq1/ 4x,
\mathcal{G}_{n,x} \bigr) + \P\bigl(\llvert D_{2n} \rrvert
\geq1/ 4x^2 , \mathcal{E}_{n,x} \bigr) + \P\bigl(\mathcal
{G}_{n,x}^c \bigr)
\end{eqnarray}
for all $x\geq1$ satisfying (\ref{xcond1}) and
\begin{equation}
L_{n,x} \leq c_3 x^2 , \label{u-rc}
\end{equation}
where $R_{n,x}$ is given in (\ref{rn1}) but with $D_{2n}$ replaced by
$D_{3n}$. In particular, for $2<p\leq3$, we have $L_{n,x} \leq(\sigma
_p/\sigma)^p x^p n^{1-p/2}$, and thus the constraint (\ref{u-rc}) is
satisfied whenever
\begin{equation}
1\leq x \leq\bigl(c_3^{1/p}/2\bigr) ( \sigma/
\sigma_p )^{1/p} n^{1/2-1/p}. \label{xcond2}
\end{equation}

However, for $0\leq x\leq1$, it follows from (\ref{sc1}) that
\begin{eqnarray*}
\P\bigl( T_n^{\ast} \geq x \bigr) \leq\P\bigl(
\mathcal{G}_{n,x}^c \bigr) + \bigl\{1-\Phi(x) \bigr\} (1
+C_7 \breve{R}_{n,x} ),
\end{eqnarray*}
for $\breve{R}_{n,x}$ as in (\ref{rn2}) with $D_{2n}$ replaced with $D_{3n}$.

In view of (\ref{pbsetest}) and (\ref{ubddec}), (\ref{cmd-ubd-u}) follows
directly from the following two propositions.
\end{pf*}

%
\begin{proposition} \label{prop1}
Under condition (\ref{k-c}), there exists a positive constant $C$
independent of $n$ such that
\begin{eqnarray}\label{tr-pr}
&& \P\bigl( \llvert D_{1n} \rrvert/V_n \geq1/4x ,
\mathcal{G}_{n,x} \bigr) +\P\bigl( \llvert D_{2n} \rrvert\geq1/4x^2 , \mathcal{E}_{n,x} \bigr)
\nonumber\\[-8pt]\\[-8pt]\nonumber
&&\quad \leq C
\sqrt{a_m} \bigl\{1-\Phi(x)\bigr\} x^2 n^{-1/2} ,
\end{eqnarray}
holds for all $x\geq1$ satisfying (\ref{hypo1}), where $a_m = \max\{c_0
\tau, c_0+m\}$, $\mathcal{G}_{n,x}$ and $\mathcal{E}_{n,x}$ are
given in (\ref{setGnx}) and (\ref{setSnx}), respectively.
\end{proposition}

%
\begin{proposition} \label{prop2}
There is a positive constant $C $ independent of $n$ such that
\begin{equation}
R_{n,x} \leq C\sigma_h x^3 n^{-1/2}
\label{rn1-ub}
\end{equation}
for all $x \geq1$ and
\begin{equation}
\breve{R}_{n,x} \leq C \sigma_hn^{-1/2}
\label{rn2-ub}
\end{equation}
for $0\leq x\leq1$, where $\sigma_h$ is given in (\ref{vardef}).
\end{proposition}

\begin{pf*}{Proof of (\ref{cmd-lbd-u})}
Observe that
\begin{eqnarray*}
\P\bigl(T_n^{\ast} \geq x\bigr) &\geq& \P\bigl\{
W_n+D_{1n} \geq xV_n(1+D_{2n})^{1/2},
\mathcal{G}_{n,x} \bigr\}
\\
&\geq& \P\bigl\{ W_n+D_{1n} \geq x V_n(1+D_{3n})^{1/2}
\bigr\} - \P\bigl(\mathcal{G}_{n,x}^c \bigr).
\end{eqnarray*}
Then (\ref{cmd-lbd-u}) follows from (\ref{cmd-lbd}), Remark~\ref{r20},
(\ref{pbsetest}) and Proposition~\ref{prop2}. Finally, assembling
(\ref{xcond1}) and (\ref{xcond2}) yields (\ref{hypo1}) and
completes the proof of Theorem~\ref{t11}.
\end{pf*}

\subsection{Proof of Propositions \texorpdfstring{\protect\ref{prop1}}{6.1} 
and \texorpdfstring{\protect
\ref{prop2}}{6.2}}
\label{pfpropsec}
We begin with a technical lemma, the proof of which is presented in the
\hyperref[appe]{Appendix}.

%
\begin{lemma} \label{tr-lm}
There exist an absolute constant $C$ and constants $B_1$--$B_4$
independent of $n$, such that for all $y\geq0$,
\begin{eqnarray}
&& \P\bigl\{ \Lambda_n^2 \geq a_m y \bigl(
B_1 + B_2 V_n^2 \bigr)
n^{2m-2} \bigr\} \leq C e^{-y/4} \label{tr-ineq-1}
\end{eqnarray}
and
\begin{eqnarray}
&& \P\biggl\{ \frac{ \llvert \sum_{1\leq i_1<\cdots<i_m \leq n}
r(X_{i_1},\ldots,X_{i_m})\rrvert } {\sqrt{a_m} ( B_3 + B_4 V_n^2)^{1/2}
n^{m-1} } \geq y \biggr\} \leq C e^{-y/4},
\label{tr-ineq-2}
\end{eqnarray}
where $a_m= \max\{c_0 \tau, c_0+m\}$, and $V_n^2$ and $\Lambda_n^2$
are given in (\ref{WVxi}) and (\ref{PsiGa}), respectively.
\end{lemma}

The above lemma generalizes and improves Lemma~3.4 of \cite
{LaiShaoWang2011} where $m=2$ and the bound was of the order $ n
e^{-y/8}$ instead of $e^{-y/4}$. Lemma~\ref{sumineq} in the \hyperref[appe]{Appendix}
makes it possible to eliminate the factor $n$.

\begin{pf*}{Proof of Proposition~\ref{prop1}}
By (\ref{de-ubd})
and the definition of $\mathcal{E}_{n,x}$ in (\ref{setSnx}), we get
\begin{eqnarray*}
\P\bigl( \llvert D_{2n} \rrvert\geq1/ 4x^2 ,
\mathcal{E}_{n,x} \bigr) \leq\P\bigl( \La_n^2
\geq c_4 V_n^2 x^{-4}
n^{2m-1} , \mathcal{G}_{n,x} \bigr),
\end{eqnarray*}
provided that $1\leq x\leq c_5 n^{1/4}$. Because $V_n^2 \geq1/2$ on
$\mathcal{G}_{n,x}$, it is easy to see that
\[
V_n^2 \geq(2 B_1 + B_2)^{-1}
\bigl(B_1 + B_2V_n^2\bigr)
\]
for $B_1$ and $B_2$ as in Lemma~\ref{tr-lm}. Therefore, taking
\[
y = \frac{c_4}{ 2 B_1+ B_2 } \cdot\frac{n}{a_m x^4} %
\]
in (\ref{tr-ineq-1}) leads to
\begin{equation}
\P\bigl( \llvert D_{2n} \rrvert> 1/ 4x^2 ,
\mathcal{E}_{n,x} \bigr) \leq C\exp\bigl\{ - c_6 n/ \bigl(
a_m x^4\bigr) \bigr\}. \label{pbest3}
\end{equation}

Using (\ref{tr-ineq-2}), it can be similarly shown that
\begin{equation}
\P\bigl( \llvert D_{1n} \rrvert/V_n > 1/4x ,
\mathcal{G}_{n,x} \bigr) \leq C\exp\bigl\{ - c_7
n^{1/2} / \bigl( a_m^{1/2} x \bigr) \bigr\}.
\label{pbest4}
\end{equation}
Together, (\ref{pbest3}), (\ref{pbest4}) and (\ref{gaussiantails})
imply (\ref{tr-pr}) as long as
\begin{equation}
1\leq x \leq c_8 (n / a_m )^{1/6}.
\label{xcond3}
\end{equation}\upqed\vspace*{12pt}
\end{pf*}

\begin{pf*}{Proof of Proposition~\ref{prop2}}
For $x\geq0$ and
$1\leq i\leq n$, put $Y_i = x \xi_i-x^2\xi_i^2/2$, and let
\[
L_k:= \e\bigl( r_{1,\ldots,k} e^{Y_1+ \cdots+Y_k} \bigr),\qquad
\tilde{L}_k := \e\bigl( r_{1,\ldots,k} e^{Y_2+ \cdots+Y_k} \mid
X_1 \bigr) %
\]
for $2\leq k\leq m$, where $r_{1,\ldots,k} := \e\{ r(X_1,\ldots
,X_m)\mid X_1,\ldots, X_k \}$ for $ r(X_1,\ldots,X_m)$ as in (\ref{rm}).
In particular, put $r_{1, \ldots, m}:= r(X_1, \ldots, X_m)$ and note
that $\e r_{1, \ldots, m}^2 \leq\sigma_h^2$. The following lemma
provides the upper bounds for $L_m$ and $\tilde L_m$.

%
\begin{lemma} \label{l1}
For any $0\leq x\leq\sqrt{n}/2$, we have
%
\begin{eqnarray}
\llvert L_m\rrvert& \leq& C \sigma_h x^2
n^{-1} , \label{l1-1}
\\
\llvert\tilde{L}_m\rrvert& \leq& C \bigl\{ E \bigl(r_{1, \ldots, m}^2
\mid X_1\bigr) \bigr\}^{1/2} x n^{-1/2} .
\label{l1-2}
\end{eqnarray}
\end{lemma}

We postpone the proof of Lemma~\ref{l1} to the end of this section.
Recall the definition of $D_{1n}$ in~(\ref{D1}). Using H\"older's
inequality, we estimate
\[
\e\Bigl\{ \Bigl( \sum r_{i_1,\ldots,i_m} \Bigr)^2
e^{\sum_{j=1}^n
Y_j} \Bigr\} = \sum\sum\e\bigl( r_{i_1,\ldots,i_m}
r_{j_1,\ldots,j_m} e^{\sum_{j=1}^n Y_j} \bigr).
\]
Put
\begin{eqnarray}
\mc&=& \bigl\{ (i_1, j_1, \ldots, i_m,
j_m) : 1\leq i_1 \leq  \cdots \leq i_m \leq
n, 1\leq
j_1 < \cdot< j_m \leq n \bigr\}
\nonumber
\\
&=& \bigcup_{k=0}^m \bigl\{
(i_1, j_1, \ldots, i_m, j_m)
\in\mc: \bigl\llvert\{i_1, \ldots, i_m\} \cap
\{j_1, \ldots, j_m \} \bigr\rrvert= k \bigr\} :=
\bigcup_{k=0}^m \mc_k.
\nonumber
\end{eqnarray}
By (\ref{mgf-est}),
\begin{eqnarray*}
&& \e\Bigl\{ \Bigl( \sum r_{i_1,\ldots,i_m} \Bigr)^2
e^{\sum_{j=1}^n
Y_j} \Bigr\}
\\
&&\quad = \sum_{k=0}^m \sum
_{(i_1, j_1, \ldots, i_m, j_m) \in\mc_k} \e\bigl( r_{i_1,\ldots,i_m}
r_{j_1,\ldots,j_m}
e^{\sum_{j=1}^n Y_j} \bigr)
\\
&&\quad = \sum_{k=0}^m \pmatrix{n \cr m} \pmatrix{n-k
\cr m-k} \e\bigl( r_{1,\ldots,m} r_{1,\ldots,k,m+1,\ldots,2m-k} e^{
\sum
_{j=1}^{2m-k} Y_j} \bigr)
\cdot\bigl( \e e^{Y_1} \bigr)^{n-2m+k}
\\
&&\quad = \pmatrix{n \cr m}^2 \bigl(\e e^{Y_1}\bigr)^{-2m}
I_{n,x} L_m^2 + \pmatrix{n \cr m} \pmatrix{n-1 \cr m-1}
\bigl(\e e^{Y_1}\bigr)^{1-2m} I_{n,x} \e\bigl(
\tilde{L}_m^2 e^{Y_1} \bigr)
\\
&&\qquad{} + \sum_{k=2}^m \pmatrix{n \cr m} \pmatrix{n-k
\cr m-k} \bigl( \e e^{Y_1} \bigr)^{k-2m} I_{n,x}
\e\bigl( r_{1,\ldots,m} r_{1,\ldots,k,m+1,\ldots,2m-k} e^{ \sum
_{j=1}^{2m-k} Y_j } \bigr)
\\
&&\quad \leq C I_{n,x} n^{2m} \bigl( L_m^2
+ n^{-1} \e\tilde{L}_m^2 + \sigma
_h^2 n^{-2} \bigr),
\end{eqnarray*}
which together with Lemma~\ref{l1} yields for $x \geq1$,
\begin{eqnarray*}
\e\Bigl\{ \Bigl( \sum r_{i_1,\ldots,i_m} \Bigr)^2
e^{\sum_{j=1}^n
Y_j} \Bigr\} \leq C \sigma_h^2
I_{n,x} x^4 n^{2m-2} .
\end{eqnarray*}
This, together with (\ref{D1}) gives
\begin{equation}
\e\bigl( \llvert D_{1n} \rrvert e^{\sum_{j=1}^n Y_j } \bigr) \leq C
\sigma_h I_{n,x} x^2 n^{-1/2}.
\label{Dn1bd}
\end{equation}

Recall that $\psi_i= \sum_{1\leq\ell_1 \leq  \cdots\leq  \ell_{m-1} (\neq
i) \leq n} r(X_i, X_{\ell_1}, \ldots, X_{\ell_{m-1}})$. Then it can
be similarly derived that
\begin{equation}
\e\bigl( \psi_i^2 e^{\sum_{j=1}^n Y_j} \bigr) \leq C
\sigma_h^2 I_{n,x} x^2
n^{2m-3} . \label{pre2}
\end{equation}
Together with (\ref{D3}), this yields
\begin{equation}
\e\bigl( D_{3n} e^{\sum_{j=1}^n Y_j }\bigr) \leq C \sigma_h
I_{n,x} x n^{-1/2} . \label{rn1-0}
\end{equation}

Next, for each $1\leq i\leq n$, let $D_{1n}^{(i)}$ and $D_{3n}^{(i)}$
be obtained from $D_{1n}$ and $D_{3n}$, respectively, by throwing away
the summands that depend on $X_i$. Then, by (\ref{D1}) and (\ref
{D3}), we have
\begin{eqnarray*}
\bigl\llvert D_{1n} - D_{1n}^{(i)}\bigr\rrvert
\leq\frac{\sqrt{n} }{m {n \choose m}} \llvert\psi_i \rrvert
\end{eqnarray*}
and
\begin{eqnarray}
&& x\bigl\llvert D_{3n}- D_{3n}^{(i)}\bigr\rrvert
\nonumber
\\
&&\quad \leq C \sigma_h^{-1} n^{-2m+3/2} \biggl\{
\psi_i^2 + \sum_{j \neq i}
\biggl(\sum_{ 1 \leq j_1 <
\cdots< j_{m-2} (\neq i, j) \leq n } r_{i, j, j_1,\ldots, j_{m-2}}
\biggr)^2
\nonumber
\\
&&\qquad{} + 2 \sum_{j \neq i} \biggl\llvert\biggl( \sum
_{1\leq j_1 < \cdots
< j_{m-2} (\neq i, j) \leq n} r_{i, j, j_1, \ldots, j_{m-2}} \biggr)
\biggl( \sum
_{ 1\leq j_1 < \cdots< j_{m-1} (\neq j)\leq n } r_{j,
j_1,\ldots, j_{m-1}} \biggr) \biggr\rrvert\biggr
\} .
\nonumber
\end{eqnarray}

Using a conditional analogue of the argument that leads to (\ref
{pre2}) implies
\begin{equation}
\e\bigl( \psi_i^2 e^{\sum_{j\neq i} Y_j}\mid
X_i \bigr) \leq C I_{n,x} x^2
n^{2m-3} \times\e\bigl( r_{1,\ldots, m}^2\mid
X_i\bigr) , \label{pre3}
\end{equation}
as a consequence of which (recall that $\xi_{i,x}=x\xi_i$)
%
\begin{eqnarray} \label{rn1-1}
&& \sum_{i=1}^n\e\bigl\{ \min\bigl(
\llvert\xi_{i,x} \rrvert, 1\bigr) \bigl\llvert D_{1n}-D_{1n}^{(i)}
\bigr\rrvert e^{\sum_{j\neq i}^n Y_j} \bigr\}
\nonumber
\\
&&\quad \leq C n^{-m+1/2} \sum_{i=1}^n\e
\bigl[ \min\bigl(\llvert\xi_{i,x} \rrvert, 1\bigr) \bigl\{ \e\bigl(
\psi_i^2 e^{\sum_{j\neq i} Y_j} \mid X_i\bigr)
\bigr\}^{1/2} \bigl\{ \e\bigl( e^{\sum_{j\neq i} Y_j} \bigr) \bigr
\}^{1/2} \bigr]
\nonumber\\[-8pt]\\[-8pt]\nonumber
&&\quad \leq C I_{n,x} x^2 n^{-1} \sum
_{i=1}^n \bigl( \e\xi_i^2
\bigr)^{1/2} \bigl( E r_{1, \ldots, m}^2
\bigr)^{1/2}
\nonumber
\\
&&\quad \leq C \sigma_h I_{n,x} x^2
n^{-1/2} .\nonumber
\end{eqnarray}

For the contributions from $\llvert D_{3n}-D_{3n}^{(i)}\rrvert $, we have
\begin{eqnarray*}
\e\bigl\{ \min\bigl(\llvert\xi_{i,x} \rrvert, 1\bigr)
\psi_i^2 e^{\sum_{j \neq i} Y_j} \bigr\} & =& \e\bigl\{ \min
\bigl(\llvert\xi_{i,x} \rrvert, 1\bigr) \times\e\bigl(
\psi_i^2 e^{\sum_{j\neq
i}Y_j} \mid X_i \bigr)
\bigr\}
\\
& \leq& C I_{n,x} x^2 n^{2m-3} \times\e\bigl\{
\min\bigl(\llvert\xi_{i,x} \rrvert, 1\bigr) r_{1,\ldots,m}^2
\bigr\},
\end{eqnarray*}
and for each pair $(i,j)$ such that $1\leq i\neq j \leq n$,
\begin{eqnarray*}
&& \e\Bigl\{ \min\bigl(\llvert\xi_{i,x} \rrvert, 1\bigr) \Bigl\llvert
\Bigl( \sum\psi_{i , j, j_1,\ldots, j_{m-2}} \Bigr) \Bigl( \sum
\psi_{j, j_1,\ldots, j_{m-1}} \Bigr) \Bigr\rrvert e^{ \sum
_{k\neq i} Y_k }\Bigr\}
\\
&&\quad \leq\e\Bigl[ \min\bigl(\llvert\xi_{i,x} \rrvert, 1\bigr) \e\Bigl\{
\Bigl( \sum\psi_{i , j, j_1,\ldots, j_{m-2}} \Bigr)^2 e^{\sum_{k\neq i} Y_k}
\big| X_i \Bigr\}^{1/2}
\\
&&\qquad{} \times\e\Bigl\{ \Bigl(\sum
\psi_{j, j_1 , \ldots, j_{m-1}} \Bigr)^2 e^{\sum
_{k\neq i} Y_k} \Bigr\}^{1/2}
\Bigr]
\\
&&\quad \leq C I_{n,x} x^2 n^{2m-7/2} \times\e\llvert
\xi_i r_{1,\ldots,m}\rrvert\times\bigl( \e r^2_{1, \ldots, m}
\bigr)^{1/2}
\\
&&\quad \leq C \sigma_h^2 I_{n,x} x^2
n^{2m-4} ,
\end{eqnarray*}
where we used (\ref{pre2}) in the second step. Similarly, it can be
proved that
\begin{eqnarray*}
&& \e\Bigl\{ \min\bigl(\llvert\xi_{i,x} \rrvert, 1\bigr) \Bigl( \sum
r_{i , j, j_1, \ldots,
j_{m-2}} \Bigr)^2 e^{\sum_{k\neq i} Y_k} \Bigr\}
\\
&&\quad  = \e\Bigl[ \min\bigl(\llvert\xi_{i,x} \rrvert, 1\bigr) \e\Bigl\{
\Bigl( \sum r _{i, j,
j_1, \ldots, j_{m-2}} \Bigr)^2 e^{\sum_{k\neq i}Y_k}
\big| X_i \Bigr\} \Bigr] \leq C \sigma_h^2
I_{n,x} n^{2m-4} .
\end{eqnarray*}

Adding up the above calculations, we get
\begin{eqnarray*}
\sum_{i=1}^n\e\bigl\{ x \min\bigl(
\llvert\xi_{i,x} \rrvert, 1\bigr) \bigl\llvert D_{3n}-D_{3n}^{(i)}
\bigr\rrvert e^{\sum_{j\neq i} Y_j} \bigr\} \leq C \sigma_h
I_{n,x} x^2n^{-1/2} .
\end{eqnarray*}
This, together with (\ref{Dn1bd}), (\ref{rn1-0}) and (\ref{rn1-1})
implies (\ref{rn1-ub}).

Finally, we consider the case of $0\leq x\leq1$. By H\"older's inequality,
%
\begin{eqnarray}
\e\llvert D_{1n} \rrvert& \leq& C n^{1/2} \pmatrix{n \cr
m}^{-1} \Bigl\{ \e\Bigl(\sum r_{i_1,\ldots,i_m}
\Bigr)^2 \Bigr\}^{1/2} \leq C \sigma_h
n^{-1/2}
\label{rn2-1}
\end{eqnarray}
and
\begin{eqnarray}
\e D_{3n} & \leq& C \bigl( \sigma_h n^{-1/2}+
\sigma_h^{-1} n^{-2m+3/2} \e\La_n^2
\bigr) \leq C \sigma_h n^{-1/2} . \label{rn2-2}
\end{eqnarray}
Moreover, for any pair $(i,j)$ such that $1\leq i \neq j \leq n$,
\begin{eqnarray*}
\e\psi_i^2 \leq C \sigma_h^2
n^{2m-3} ,\qquad\e\Bigl( \sum\psi_{i, j, j_1, \ldots, j_{m-2}}
\Bigr)^2 \leq C \sigma_h^2 n^{2m-4}
\end{eqnarray*}
and
\begin{eqnarray*}
&& \e\Bigl\{ \Bigl\llvert\Bigl( \sum r_{i, j, \ell_1,\ldots, \ell
_{m-2}} \Bigr) \Bigl(
\sum r_{j, j_1, \ldots, j_{m-1}} \Bigr) \Bigr\rrvert \big|X_i \Bigr\}
\\
&&\quad \leq\Bigl[ \e\Bigl\{ \Bigl(\sum r_{i,j, \ell_1,\ldots, \ell
_{m-2}}
\Bigr)^2 \big|  X_i \Bigr\} \Bigr]^{1/2} \times
\Bigl\{ \e\Bigl(\sum\psi_{j, j_1,\ldots, j_{m-1}} \Bigr)^2 \Bigr
\}^{1/2}
\\
&&\quad \leq C \sigma_h n^{2m-7/2} \times\bigl\{ \e\bigl(
r^2_{1, \ldots, m} \mid X_i \bigr) \bigr
\}^{1/2} .
\end{eqnarray*}

Combining the above calculations, we obtain
\begin{equation}
\sum_{i=1}^n\e\bigl\llvert
\xi_i \bigl( D_{1n} - D_{1n}^{(i)}
\bigr) \bigr\rrvert\leq C n^{-m+1/2} \sum_{i=1}^n
\bigl( \e\xi_i^2 \bigr)^{1/2} \bigl( \e
\psi_i^2 \bigr)^{1/2} \leq C
\sigma_h n^{-1/2} \label{rn2-3}
\end{equation}
and
%
\begin{eqnarray} \label{rn2-4}
&& \sum_{i=1}^n \e\bigl\llvert x
\xi_i I\bigl\{\llvert\xi_i \rrvert\leq1/(1+x)\bigr\}
\bigl( D_{3n} - D_{3n}^{(i)} \bigr)\bigr\rrvert
\nonumber
\\
&&\quad \leq C \sigma_h^{-1} n^{-2m+3/2} \Biggl[ \sum
_{i=1}^n\e\psi_i^2
+ \sum_{i \neq j} \e\Bigl( \sum
\psi_{i,
j, j_1, \ldots, j_{m-2}} \Bigr)^2
\nonumber\\[-8pt]\\[-8pt]\nonumber
&&\qquad{} + 2 \sum_{i \neq j} \e\Bigl\{ \llvert
\xi_i\rrvert\times\Big|\Bigl( \sum r_{i, j, \ell_1, \ldots, \ell_{m-2}}
\Bigr) \Bigl( \sum r_{j, j_1,\ldots, j_{m-1}} \Bigr) \Big|\Bigr\} \Biggr]
\nonumber
\\
&&\quad \leq C \sigma_h n^{-1/2} .\nonumber
\end{eqnarray}

Assembling (\ref{rn2-1})--(\ref{rn2-4}) proves (\ref{rn2-ub}) and
completes the proof of Proposition~\ref{prop2}.
\end{pf*}

\begin{pf*}{Proof of Lemma~\ref{l1}}
We prove (\ref{l1-1}) by the method of induction, and (\ref{l1-2}) follows
a similar argument. First, for $m=2$, observe that
\begin{eqnarray*}
L_2 = \e\bigl( r_{1,2} e^{Y_1+Y_2} \bigr) = \e\bigl
\{ r_{1,2} \bigl(e^{Y_1}-1\bigr) \bigl(e^{Y_2}-1
\bigr) \bigr\}.
\end{eqnarray*}
Using the inequality
\begin{equation}
\bigl\llvert e^{t-t^2/2}-1\bigr\rrvert\leq2\llvert t\rrvert\qquad\mbox{for all } t \in\mathbb{R}, \label
{ineqtool1}
\end{equation}
we have (recall that $\xi_i=n^{-1/2} h_{1i}$)
\begin{eqnarray*}
\llvert L_2\rrvert\leq4 x^2 n^{-1} \e\llvert
r_{1,2} h_{1 1} h_{1 2}\rrvert\leq4 \sigma
_h x^2 n^{-1} .
\end{eqnarray*}

Similarly, noting that $ \tilde{L}_2= \e\{ r_{1,2} (e^{Y_2}-1)\mid X_1\}
$, we get
\begin{eqnarray*}
\llvert\tilde{L}_2\rrvert\leq2 \bigl\{ \e\bigl(
r_{1,2}^2 \mid X_1 \bigr) \bigr
\}^{1/2} x n^{-1/2} ,
\end{eqnarray*}
as desired.

For the general case where $m>2$, we derive
\begin{eqnarray*}
&& \e\bigl( r_{1,\ldots,m} e^{Y_1+\cdots+Y_m} \bigr)
\\
&&\quad
=  \e\bigl\{r_{1,\ldots,m} \bigl(e^{Y_1}-1\bigr)\cdots\bigl(e^{Y_m}-1
\bigr) \bigr\}
+ \sum_{1\leq i_1<\cdots<i_{m-1}\leq m}\e\bigl( r_{1,\ldots,m}
e^{Y_{i_1}+ \cdots+Y_{i_{m-1}}} \bigr)
\\
&&\qquad{}- \sum_{1 \leq i_1< \cdots<i_{m-2} \leq m} \e\bigl( r_{1,\ldots
,m}e^{Y_{i_1}+ \cdots+Y_{i_{m-2}}} \bigr)
+ \cdots\\
&&\qquad{}+(-1)^{m-1}\sum_{1\leq i_1<i_2\leq m} \e\bigl(
r_{1,\ldots,m} e^{Y_{i_1}+Y_{i_2}} \bigr)
\\
&&\quad = \e\bigl\{ r_{1,\ldots,m} \bigl(e^{Y_1}-1\bigr)\cdots
\bigl(e^{Y_m}-1\bigr) \bigr\} +m L_{m-1}
\\
&&\qquad{} - \pmatrix{m \cr m-2}L_{m-2} +\cdots+ (-1)^{m-1}\pmatrix{m
\cr 2} L_2,
\end{eqnarray*}
where for each $k$-tuple $(i_1, \ldots, i_k)$ ($2\leq k\leq m-1$)
satisfying $1\leq i_1 < \cdots< i_k\leq m$,
\begin{eqnarray*}
\e\bigl( r_{1,\ldots, m} e^{Y_{i_1}+\cdots+Y_{i_k}} \bigr) &=& \e\bigl[
e^{Y_{i_1}+
\cdots+Y_{i_k}} \e\bigl\{ r(X_1, \ldots, X_m)\mid
X_{i_1}, \ldots, X_{i_k} \bigr\} \bigr]
\\
&=& \e\bigl( r_{i_1,\ldots,i_k} e^{Y_{i_1}+ \cdots+Y_{i_k}} \bigr) = L_k,
\end{eqnarray*}
by definition. Using inequality (\ref{ineqtool1}) again gives
\begin{eqnarray*}
\bigl\llvert\e\bigl\{ r_{1,\ldots,m}\bigl(e^{Y_1}-1\bigr)\cdots
\bigl(e^{Y_m}-1\bigr) \bigr\} \bigr\rrvert\leq2^m
x^m n^{-m/2} \e\llvert r_{1,\ldots,m} h_{1 1}
\cdots h_{1 m}\rrvert\leq\sigma_h (2 x)^m
n^{-m/2} ,
\nonumber
\end{eqnarray*}
completing the proof of (\ref{l1-1}) by induction and under the condition
that $x \leq\sqrt{n}/2$.
\end{pf*}

\begin{appendix}\label{appe}

\section{Proof of Theorem~\texorpdfstring{\protect\ref{t23}}{2.2}} The main idea of the
proof is to first truncate $\xi_i$ at a suitable level, and then apply
the randomized concentration inequality to the truncated variables.

For $x \geq0$ and $i=1,\ldots, n$, define $Y_i = x\xi_i -x^2 \xi
_i^2/2$, and
\[
\bar{\xi_i}= \xi_{i} I\bigl\{ \llvert
\xi_i \rrvert\leq1/(1+x)\bigr\},\qquad\bar{Y_i}=
Y_i I\bigl\{\llvert\xi_i \rrvert\leq1/(1+x) \bigr
\}. %
\]
Moreover, put $S_Y=\sum_{i=1}^nY_i$ and $S_{\bar{Y}}= \sum_{i=1}^n\bar{Y_i}$.

We first consider the case of $x>0$. Proceeding as in (\ref{up}) and
(\ref{low}), we have
\begin{equation}
\P\bigl( S_Y \geq x^2/2+x \De_{2n} \bigr)
\leq\P(T_n \geq x) \leq\P\bigl( S_Y \geq
x^2/2 -x \De_{1n} \bigr), \label{BE1}
\end{equation}
where $\De_{1n}=x(V_n^2-1)^2+\llvert D_{1n} \rrvert + xD_{2n} \wedge0
$ and $\De
_{2n}=xD_{2n}/2-D_{1n}$. Replacing the $\xi_i^2$'s with their
truncated versions, we put
$\De_{3n} = x( \sum_{i=1}^n\bar{\xi_i}^2 -1 )^2+\llvert D_{1n}
\rrvert + x D_{2n}
\wedge0$, such that
%
\begin{eqnarray}\label{BE2}
&& \bigl\llvert\P\bigl( S_Y \geq x^2/2-x
\De_{1n} \bigr) - \P\bigl( S_{\bar{Y}} \geq x^2/2 -x
\De_{3n} \bigr) \bigr\rrvert
\nonumber\\[-9pt]\\[-9pt]\nonumber
&&\quad \leq\P\Bigl\{ \max_{1\leq i\leq n}\llvert\xi_i \rrvert
>1/(1+x) \Bigr\} \leq(1+x)^2 \sum_{i=1}^n
\e\xi_i^2 I\bigl\{ \llvert\xi_i \rrvert
>1/(1+x) \bigr\},
\end{eqnarray}
and the same bound holds for $ \llvert \P( S_Y \geq x^2/2+x \De_{2n}
)- \P(
S_{\bar{Y}} \geq x^2/2 +x \De_{2n} ) \rrvert $.

It suffices to estimate the probabilities of the truncated random
variables. Consider the following decomposition:
\begin{equation}
\P\bigl( S_{\bar{Y}} \geq x^2/2 -x \De_{3n} \bigr)
\leq\P\bigl( S_{\bar{Y}} \geq x^2/2 \bigr) + \P\bigl(
x^2/2- x \De_{3n} \leq S_{\bar{Y}} <
x^2/2 \bigr), \label{trun-decom}
\end{equation}
where\vspace*{2pt} $S_{\bar{Y}} =\sum_{i=1}^n\bar{Y}_i$ denotes the sum of the
truncated random variables. Write $\bar m_n = \sum_{i=1}^n\e\bar
Y_i$, $ \bar{\sigma}_n^2 =\sum_{i=1}^n\var( \bar{Y}_i )$ and $\bar
{v}_n= \sum_{i=1}^n\e\llvert \bar{Y}_i\rrvert ^3 $. By a similar
calculation to
that leading to (\ref{mgf-expan}),
\begin{eqnarray*}
\e\bar{Y}_i &=& - \bigl(x^2/2\bigr) \e
\xi_i^2 + O(1) \bigl(x+x^2\bigr) \e
\xi_i^2 I\bigl\{ \llvert\xi_i \rrvert
>1/(1+x)\bigr\} ,
\\
\e\bar{Y}_i^2 &=& x^2\e
\xi_i^2 +O(1) \bigl[ x^2 \e
\xi_i^2 I\bigl\{ \llvert\xi_i\rrvert
>1/(1+x)\bigr\} + x^3 \e\llvert\bar{\xi}_i\rrvert
^3 \bigr] ,
\\
\e\llvert\bar{Y}_i\rrvert^3 &=& O(1) x^3
\e\llvert\bar{\xi}_i \rrvert^3
\end{eqnarray*}
and
\begin{eqnarray*}
\var(\bar{Y}_i) &=& x ^2 \e\xi_i^2
+O(1) \bigl[ x^2 \e\xi_i^2 I\bigl\{ \llvert
\xi_i \rrvert>1/(1+x) \bigr\} + x^3 \e\llvert\bar{
\xi}_i\rrvert^3 \bigr],
\end{eqnarray*}
where $\llvert O(1) \rrvert \leq C_1$ for some absolute constant
$C_1$. Combining
these calculations, we have
%
\begin{eqnarray}\label{var-lbd}
\bar m_n &=& -x^2/2 + O(1 ) \bigl(x+x^2
\bigr) \sum_{i=1}^n\e\xi_i^2
I\bigl\{ \llvert\xi_i \rrvert>1/(1+x) \bigr\} ,
\nonumber\\[-8pt]\\[-8pt]\nonumber
\bar\sigma_n^2 &=& x^2+O(1) x^2
\sum_{i=1}^n \bigl[ \e
\xi_i^2 I\bigl\{ \llvert\xi_i\rrvert
>1/(1+x) \bigr\} + x \e\llvert\bar{\xi}_{i}\rrvert^3
\bigr] \geq x^2/2,
\end{eqnarray}
where the last inequality holds as long as $ (1+x)^{-2} L_{n,1+x} \leq
(2 C_1)^{-1}$. Otherwise, if this constraint is violated, then (\ref
{sc1}) is always true provided that $C>2 C_1$.

Applying the Berry--Esseen inequality to the first addend in (\ref
{trun-decom}) gives
\begin{eqnarray}\label{BE3}
\P\bigl( S_{\bar Y } \geq x^2/2 \bigr) &=& 1-\Phi( {\bar
\varepsilon}_n )+O(1) \bar v_n \bar{\sigma}_n^{-3}
\nonumber\\[-8pt]\\[-8pt]\nonumber
&=& 1-\Phi(x) +O(1) (1+x)^{-1}L_{n,1+x},
\end{eqnarray}
where $ \bar{\varepsilon}_n :=\bar{\sigma}_n^{-1}(x^2/2- \bar
{m}_n) = x+O(1) (1+x)^{-1}L_{n,1+x}$ by (\ref{var-lbd}).

For the second addend in (\ref{trun-decom}), applying the
concentration inequality (\ref{con-ineq-1}) to $\bar W_n = \bar
{\sigma
}_n^{-1} ( S_{\bar Y} - \bar{m}_n ) $ and noting that $\llvert \bar
Y_{i}\rrvert
\leq3 x \llvert \bar{\xi}_{i}\rrvert /2$, we obtain
%
\begin{eqnarray}\label{apprcineq1}
&& \P\bigl( x^2/2- x\llvert{\De}_{3n}\rrvert\leq
S_{\bar Y} < x^2/2 \bigr)\nonumber
\\
&&\quad =\P( \bar{\varepsilon}_n-
x {\De}_{3n} / \bar{\sigma}_n \leq\bar{W}_n
\leq\bar{\varepsilon}_n )
\nonumber\\[-8pt]\\[-8pt]\nonumber
&&\quad \leq 17 \bar{\sigma}_n^{-3} \sum
_{i=1}^n\e\llvert\bar Y_{i}\rrvert
^3 + 5 x \bar{\sigma}_n^{-1} \e\llvert{
\De}_{3n}\rrvert+ 2 x\bar{\sigma}_n^{-2}\sum
_{i=1}^n\e\bigl\llvert\bar
Y_{i} \bigl\{ {\De}_{3n}- {\De}_{3n}^{(i)}
\bigr\} \bigr\rrvert
\nonumber
\\
&&\quad \leq  C \Biggl[ \sum_{i=1}^n\e\llvert
\bar{\xi}_{i}\rrvert^3 + \e\llvert\De_{3n}
\rrvert+ \sum_{i=1}^n\e\bigl\llvert\bar
\xi_{i} \bigl\{ {\De}_{3n}- {\De}_{3n}^{(i)}
\bigr\} \bigr\rrvert\Biggr], \nonumber
\end{eqnarray}
where $\De_{3n} = x ( \sum_{i=1}^n\bar\xi_{i}^{ 2} -1 )^2+\llvert
D_{1n} \rrvert +
x\llvert D_{2n} \rrvert $. For $i=1,\ldots,n$, put
\begin{eqnarray*}
d_{ i} &=& \Biggl( \sum_{i=1}^n
\bar{\xi}_{i}^{ 2} -1 \Biggr)^2 - \biggl( \sum
_{ j\neq i} \bar{\xi}_{j}^{ 2}-1
\biggr)^2
\\
&=& \bar{\xi}_{i}^{ 2} \Biggl[ \bar{\xi}_{i}^{ 2}
+ 2\sum_{j\neq i} \bigl( \bar{\xi}_{j}^{ 2}
- \e\bar{\xi}_{j}^{ 2} \bigr) - 2 \e\bar{\xi
}_{i}^{ 2} - 2 \sum_{i=1}^n
\e{\xi}_{i}^2 I\bigl\{ \llvert\bar{\xi}_{i}
\rrvert>1/(1+x) \bigr\} \Biggr] .
\end{eqnarray*}
Direct calculation shows that
\begin{eqnarray*}
\e\Biggl( \sum_{i=1}^n \bar{
\xi}_{i}^2-1 \Biggr)^2 & \leq& C
(1+x)^{-4} \bigl(L_{n,1+x}+L^2_{n,1+x}
\bigr) ,
\\
\sum_{i=1}^n\e\llvert\bar
\xi_{i} d_{i}\rrvert& \leq& C (1+x)^{-5}
\bigl(L_{n,1+x}+L^2_{n,1+x}\bigr).
\end{eqnarray*}
Substituting this into (\ref{apprcineq1}), we get
\begin{eqnarray*}
&& \P\bigl( x^2/2- x\llvert{\De}_{3n}\rrvert\leq
S_{\bar Y} < x^2/2 \bigr)
\\
&&\quad \leq C \Biggl[ (1+x)^{-2} L_{n,1+x} + \e\llvert
D_{1n} \rrvert+ x \e\llvert D_{2n} \rrvert
\\
&&\qquad{} + \sum
_{i=1}^n\e\bigl\{ \llvert\bar\xi_{i}
\rrvert\bigl( \bigl\llvert D_{1n} - D_{1n}^{(i)}
\bigr\rrvert+ x\bigl\llvert D_{2n} - D_{2n}^{(i)}
\bigr\rrvert\bigr) \bigr\} \Biggr].
\end{eqnarray*}
This, together with (\ref{BE1}), (\ref{BE2}), (\ref{trun-decom})
and (\ref{BE3}) implies
\[
P(T_n \leq x ) \leq\Phi(x) + C \breve{R}_{n,x} %
\]
for all $x>0$, where $\breve{R}_{n,x}$ is given in (\ref{rn2}). A
lower bound can be similarly obtained by noting that $\P( {S}_{\bar
{Y}} \geq x^2/2 +x\De_{2n} ) \geq\P( {S}_{\bar{Y}} \geq x^2/2 ) -
\P( x^2/2 \leq{S}_{\bar{Y}} < x^2/2+x \De_{2n} )$.

We next consider the case of $x=0$. It is straightforward that
\begin{eqnarray}
&& \bigl\llvert P(T_n \leq0) - \Phi(0)\bigr\rrvert
\nonumber
\\
&&\quad  = \bigl\llvert\P(W_n+ D_{1n} \leq0 ) - \Phi(0)\bigr
\rrvert\leq\bigl\llvert\P(W_n \leq0) - \Phi(0)\bigr\rrvert+ \P
\bigl( -\llvert D_{1n} \rrvert\leq W_n \leq\llvert
D_{1n} \rrvert\bigr).
\nonumber
\end{eqnarray}
A uniform Berry--Esseen bound (see, e.g., \cite{ChenShao2001}) gives
$\llvert P(W_n \leq0) - \Phi(0)\rrvert \leq4.1 L_{n,1}$. As before,
we can use
the truncation technique and the concentration inequality (\ref
{con-ineq-1}) to upper bound the probability $ \P( -\llvert D_{1n}
\rrvert \leq W_n
\leq\llvert D_{1n} \rrvert )$. The rest of the proof is almost
identical to that for
the case of $x>0$ and is therefore omitted.

\section{Proof of Lemma \texorpdfstring{\protect\ref{l3}}{5.3}} Recall that $Z=X^2- \e
X^2$ and $Y=X-X^2/2$. Using the inequality $\llvert e^s -1 \rrvert
\leq\llvert s\rrvert e^{s \vee
0}$ implies
\begin{eqnarray*}
\e\bigl\{ Z e^{Y} I\bigl(\llvert X\rrvert\leq1 \bigr) \bigr\} &=& \e
\bigl[ Z \bigl\{ 1+ O(1)\llvert Y\rrvert e^{Y \vee0} \bigr\} I\bigl(
\llvert
X\rrvert\leq1\bigr) \bigr]
\\
&=& \e\bigl\{ Z I\bigl(\llvert X\rrvert>1\bigr) \bigr\} +O(1) \e\bigl\{
\llvert
Z \rrvert\cdot\llvert Y\rrvert e^{Y \vee0}I\bigl(\llvert X\rrvert\leq1
\bigr) \bigr\} ,
\end{eqnarray*}
where $\llvert O(1) \rrvert \leq1$. Because $ \llvert Y\rrvert e^{Y
\vee0}I(\llvert X\rrvert \leq1) \leq1.5
\llvert X\rrvert I(\llvert X\rrvert \leq1)$, we have
\begin{equation}
\e\bigl\{ \llvert Z \rrvert\times\llvert Y\rrvert e^{Y \vee0}I\bigl(
\llvert X\rrvert\leq1\bigr) \bigr\} \leq1.5 \e\bigl\{ \llvert X\rrvert
^3 I\bigl(\llvert X\rrvert\leq1\bigr) \bigr\}. \label{e1}
\end{equation}
Note that if both $f$ and $g$ are increasing functions, then $\e f(X)
\e g(X) \leq
\e\{ f(X) g(X) \}$. In particular, we have $\e X^2 \times\P(\llvert
X\rrvert >1)
\leq\e\{ \llvert X\rrvert ^2 I(\llvert X\rrvert >1)\}$, which
further implies
\[
\e\bigl\{ \llvert Z\rrvert e^Y I\bigl(\llvert X\rrvert>1 \bigr)
\bigr\} \leq\sqrt{e} \e\bigl\{ X^2 I\bigl(\llvert X\rrvert>1\bigr)
\bigr\}. %
\]
Together with (\ref{e1}), this yields (\ref{l3-a}).

For (\ref{l3-b}), it is straightforward that
\begin{eqnarray*}
\e\bigl( Z^2 e^Y \bigr) &=& \e\bigl\{ Z^2
e^Y I\bigl(\llvert X\rrvert\leq1\bigr) \bigr\} + \e\bigl\{
Z^2 e^Y I\bigl(\llvert X\rrvert> 1\bigr) \bigr\}
\\
&\leq& \sqrt{e} \bigl[ \e\bigl\{ X^4 I\bigl(\llvert X\rrvert\leq1
\bigr) \bigr\} + \bigl(\e X^2\bigr)^2 \P\bigl(\llvert X
\rrvert\leq1\bigr) -2 \e X^2 \times\e\bigl\{ X^2 I\bigl(
\llvert X\rrvert\leq1\bigr) \bigr\} \bigr]
\\
&&{} + \e\bigl\{ X^4 e^{X-X^2/2}I\bigl(\llvert X\rrvert> 1
\bigr) \bigr\} +\sqrt{e} \bigl(\e X^2\bigr)^2 \times\P
\bigl(\llvert X\rrvert>1\bigr)
\\
&\leq&\sqrt{e} \e\bigl\{ X^4 I\bigl(\llvert X\rrvert\leq1\bigr)
\bigr\} + 4 \e\bigl\{ X^2 I\bigl(\llvert X\rrvert> 1\bigr) \bigr\}
\\
&&{}+ \sqrt{e} \bigl(\e X^2\bigr)^2 - 2\sqrt{e} \e
X^2 \times\e\bigl\{ X^2 I\bigl(\llvert X\rrvert\leq1
\bigr) \bigr\}
\\
&\leq&\sqrt{e} \e\bigl\{ X^4I\bigl(\llvert X\rrvert\leq1\bigr) \bigr
\} + 4 \e\bigl\{ X^2 I\bigl(\llvert X\rrvert> 1\bigr) \bigr\}
\\
&&{}+ \sqrt{e} \e X^2 \times\e\bigl\{ X^2 I\bigl(
\llvert X\rrvert> 1\bigr) \bigr\} - \sqrt{e} \e X^2 \times\e\bigl\{
X^2 I\bigl(\llvert X\rrvert\leq1\bigr) \bigr\}
\\
&\leq&\sqrt{e} \e\bigl\{ \llvert X\rrvert^3 I\bigl(\llvert X\rrvert
\leq1\bigr) \bigr\} + 4 \e\bigl\{ X^2 I\bigl(\llvert X\rrvert> 1
\bigr) \bigr\} + \sqrt{e} \bigl\{ \e X^2 I\bigl(\llvert X\rrvert>1
\bigr) \bigr\}^2,
\end{eqnarray*}
where in the third inequality we use the inequality $\sup_{\llvert
x\rrvert >1} \{
x^2\exp( x-x^2/2) \} \leq4$.

Moreover, noting that
\begin{eqnarray*}
\sup_{\llvert x\rrvert \leq1} \bigl\{ (1-x/2)\exp\bigl(x-x^2/2\bigr)
\bigr\} \leq1\quad\mbox{and}\quad \sup_{x\in\mathbb{R}} \bigl\{ \bigl\llvert
x-x^2/2\bigr\rrvert\exp\bigl( x-x^2/2 \bigr) \bigr\}
\leq\sqrt{e}/2,
\end{eqnarray*}
we obtain
\begin{eqnarray*}
\e\bigl( \llvert Y Z\rrvert e^{Y} \bigr) & =& \e\bigl\{ \llvert Y Z
\rrvert e^Y I\bigl(\llvert X\rrvert\leq1\bigr) \bigr\} +\e\bigl\{
\llvert Y Z \rrvert e^Y I\bigl(\llvert X\rrvert>1\bigr) \bigr\}
\\
&\leq& \e\bigl\{ \bigl\llvert X^2-\e X^2\bigr\rrvert
\times\llvert X\rrvert I\bigl(\llvert X\rrvert\leq1\bigr) \bigr\} +
\frac{\sqrt
{e}}{2} \e\bigl\{ X^2 I\bigl(\llvert X\rrvert>1 \bigr)
\bigr\}
\\
&\leq& 2 \e\bigl\{ X^2 I\bigl(\llvert X\rrvert>1\bigr) \bigr\} + \e
\bigl\{ \llvert X\rrvert^3 I\bigl(\llvert X\rrvert\leq1\bigr) \bigr
\},
\end{eqnarray*}
which proves (\ref{l3-c}).

Finally, for (\ref{l3-d}), it follows from the inequality $\sup
_{\llvert x\rrvert >1} \{ \llvert x^3-x^4/2\rrvert \exp( x-x^2/2)
\} < 3.1$ that
\begin{eqnarray*}
&& \e\bigl( \llvert Y\rrvert Z^2 e^Y \bigr)
\\
&&\quad = \e\bigl\{
Z^2 \llvert Y\rrvert e^Y I\bigl(\llvert X\rrvert\leq1
\bigr) \bigr\} + \e\bigl\{ Z^2 \llvert Y\rrvert e^Y I
\bigl(\llvert X\rrvert> 1\bigr) \bigr\}
\\
&&\quad \leq \frac{\sqrt{e}}{2} \e\bigl\{ Z^2 I\bigl(\llvert X\rrvert\leq1
\bigr) \bigr\}+ \max\biggl[ 3.1 \e\bigl\{ X^2 I\bigl(\llvert X\rrvert
>1\bigr) \bigr\} , \frac{\sqrt{e}}{2} \bigl(\e X^2
\bigr)^2 P\bigl(\llvert X\rrvert>1\bigr) \biggr]
\\
&&\quad \leq\frac{\sqrt{e}}{2} \e\bigl\{ \llvert X\rrvert^3 I\bigl(\llvert
X\rrvert\leq1\bigr) \bigr\}
\\
&&\qquad{} + \max\biggl[ 3.1 \e\bigl\{ X^2 I\bigl(\llvert X\rrvert>1
\bigr) \bigr\} , \frac{ \sqrt{e}}{2} \e\bigl\{ X^2 I\bigl(\llvert X
\rrvert>1\bigr) \bigr\}+ \frac{\sqrt{e}}{2} \bigl\{ \e X^2I\bigl(
\llvert X\rrvert>1\bigr) \bigr\}^2 \biggr],
\end{eqnarray*}
as desired.

\section{Proof of Lemma \texorpdfstring{\protect\ref{tr-lm}}{6.1}} We start with two
technical lemmas. The first follows \cite{LaiShaoWang2011}.

%
\begin{lemma} \label{l00}
Let $\{\xi_i,\mathcal{F}_i,i\geq1\}$ be a sequence of martingale
differences with
$\e\xi_i^2<\infty$, and put
\begin{eqnarray*}
D_n^2=\sum_{i=1}^n
\bigl\{ \xi_i^2+2 \e\bigl(\xi_i^2
\mid\mathcal{F}_{i-1}\bigr)+3 \e\xi_i^2
\bigr\} .
\end{eqnarray*}
Then we have
\begin{eqnarray}
\P\Biggl( \Biggl\llvert\sum_{i=1}^n
\xi_i \Biggr\rrvert\geq x D_n \Biggr) \leq\sqrt{2}
\exp\bigl( -x^2/8 \bigr) \label{lsw}
\end{eqnarray}
for all $x>0$. In particular, if $\{\xi_i,i\geq1\}$ is a sequence of
independent random variables with zero means
and finite variances, write
\[
S_n=\sum_{i=1}^n
\xi_i,\qquad V_n^2=\sum
_{i=1}^n\xi_i^2 \quad\mbox{and} \quad B_n^2=\sum
_{i=1}^n\e\xi_i^2,
\]
such that $D_n^2 = V_n^2 + 5 B_n^2$. Then for any $x\geq0$,
\begin{eqnarray}
\P\bigl(\llvert S_n\rrvert\geq x D_n \bigr) \leq
\sqrt{2} \exp\bigl( -x^2/8 \bigr) \label{lsw-2}
\end{eqnarray}
and
\begin{eqnarray}
\e\bigl[ S_n^2 I\bigl\{ \llvert S_n\rrvert
\geq x (V_n+4B_n ) \bigr\} \bigr] \leq23
B_n^2 \exp\bigl( -x^2/4 \bigr).
\label{lsw-3}
\end{eqnarray}
\end{lemma}

The following result may be of independent interest.

%
\begin{lemma} \label{sumineq}
Let $\{ \xi_i, i\geq1\}$ and $\{\eta_i, i\geq1\}$ be two sequences
of arbitrary random variables.
Assume that the $\eta_i$'s are non-negative, and that for any $u> 0$,
\begin{equation}
\e\bigl\{ \xi_i I( \xi_i \geq u \eta_i
) \bigr\} \leq c_i e^{-c u}, \label{sin}
\end{equation}
where $\{c, c_i, i\geq1\}$ are positive constants. Then, for any
$u>0$, $v > 0$ and $n\geq1$,
\begin{equation}
\P\Biggl\{\sum_{i=1}^n
\xi_i \geq u \Biggl(v+ \sum_{i=1}^n
\eta_i \Biggr) \Biggr\} \leq\frac{e^{-c u}}{c u^2 v}\sum
_{i=1}^nc_i. \label{s-i}
\end{equation}
\end{lemma}

\begin{pf} 
For any $u >0$ and $ v>0$, applying Markov's and Jensen's inequalities gives
%
\begin{eqnarray} \label{s-d}
\mbox{ L.H.S. of (\ref{s-i}) } &\leq&\P\Biggl\{ \sum
_{i=1}^n(\xi_i - u
\eta_i) \geq u v \Biggr\}
\nonumber
\\
&\leq&\frac{1}{uv} \e\Biggl\{ \sum_{i=1}^n(
\xi_i-u \eta_i ) \Biggr\}_{+}
\\
&\leq& \frac{1}{uv} \sum_{i=1}^n\e(
\xi_i-u \eta_i )_{+},\nonumber
\end{eqnarray}
where $x_+ = \max(0, x)$ for all $x \in\mathbb{R}$. For each $1\leq
i\leq n$ fixed, it follows from (\ref{sin}) that
\begin{eqnarray*}
\e( \xi_i-u \eta_i )_{+} &=& \e\int
_{u \eta_i}^{\infty} I( \xi_i \geq s ) \,ds
\\
&=& \int_1^{\infty} u \e\bigl\{
\eta_i I( \xi_i \geq t u \eta_i ) \bigr
\} \,dt
\\
&\leq& \int_1^{\infty} t^{-1} \e\bigl\{
\xi_i I( \xi_i \geq t u \eta_i ) \bigr\}
\,dt
\\
&\leq& c_i \int_1^{\infty}
t^{-1} \exp( -c u t ) \,dt \leq\frac
{e^{-c u}}{c u} c_i ,
\end{eqnarray*}
which completes the proof of (\ref{s-i}) by (\ref{s-d}).
\end{pf}

To prove Lemma~\ref{tr-lm}, we use an inductive approach by
formulating the proof into three steps. Here, $C$ and $B_1, B_2, \ldots
$ denote positive constants that are independent of $n$. Recalling
(\ref{k-cprime}), it is easy to verify that
\begin{equation}
r^2(x_1,\ldots,x_m) \leq2 a_m
\bigl\{ 1 + h_1^2(x_1)+
\cdots+h_1^2(x_m) \bigr\} , \label{psi-m}
\end{equation}
where $a_m = \max\{c_0 \tau, c_0+m\} $. In line with (\ref{WVxi}),
let $W_n =n^{-1/2}\sum_{i=1}^nh_{1 i}$ and $V_n^2=n^{-1}\sum
_{i=1}^nh_{1 i}^2$. Here, and in the sequel, we write
\[
h_{1 i} = h_1(X_i),\qquad
h_{ j, i_1, \ldots, i_j} = \e\bigl\{ h(X_1, \ldots, X_m)\mid
X_{i_1}, \ldots, X_{i_j}\bigr\},\qquad 2\leq j\leq m, %
\]
for ease of exposition. The conclusion is obvious when $0\leq y\leq2$,
therefore we assume $y \geq2$ without loss of generality.

\begin{longlist}
\item[\textit{Step} 1.] Let $m=2$, then (\ref{psi-m}) reduces to
\begin{equation}
r^2(x_1,x_2) \leq2 a_2 \bigl\{
1+ h_1^2(x_1)+ h_1^2(x_2)
\bigr\}, \label{kdm2}
\end{equation}
where $a_2= \max\{ c_0 \tau, c_0+2 \}$. We follow the lines of the
proof of Lemma 3.4 in \cite{LaiShaoWang2011} with the help of
Lemma~\ref{sumineq}.

Retaining the notation in Section~\ref{proof2sec} for $m=2$, we have
\[
\Lambda_n^2=\sum_{i=1}^n
\psi_i^2 ,\qquad\psi_i = \sum
_{j=1,
j\neq i}^n r_{i,j} = \sum
_{j=1 , j\neq i}^n r(X_i,X_j),\qquad 1
\leq i\leq n. %
\]
Conditional on $X_i$, note that $\psi_i$ is a sum of independent
random variables with zero means. To apply inequality (\ref{lsw-3}), put
\begin{eqnarray*}
t_{i} = v_{i} +4 b_{i} ,\qquad
v_{i}^2= \sum_{ j\neq i}
r^2_{i, j},\qquad b_{i}^2 =\sum
_{ j\neq i} \e\bigl( r_{i,j}^2 \mid
X_i\bigr)
\end{eqnarray*}
for $1\leq i\leq n$. By (\ref{lsw-3}), $\e\{ \psi_i^2 I( \psi_i^2
\geq y t_{ i}^2 ) \mid X_i \} \leq23 b_{ i}^2 e^{-y/4}$. Taking
expectations on both sides yields
\[
\e\bigl\{ \psi_i^2 I\bigl( \psi_i^2
\geq y t_{i}^2 \bigr) \bigr\} \leq23 (n-1)e^{-y/4}
\e\bigl( r_{1,2}^2 \bigr) . %
\]

Applying Lemma~\ref{sumineq} with $\xi_i=\psi_i^2$, $\eta_i=t_{
i}$, $u=y$ and $v =a_2 n(n-1) $ gives
\begin{equation}
\P\Biggl\{ \La_n^2 \geq y \Biggl( \sum
_{i=1}^nt_{ i}^2 +
a_2 n(n-1) \Biggr) \Biggr\} \leq C \bigl(a_2
y^2\bigr)^{-1} e^{-y/4} \e\bigl(
r_{1,2}^2 \bigr) . \label{pbesta1}
\end{equation}
Direct calculation based on (\ref{kdm2}) shows
\begin{eqnarray*}
\sum_{i=1}^nv_{ i}^2
\leq a_2 (n-1)n\bigl( 2 + 4 V_n^2 \bigr) ,\qquad\sum_{i=1}^nb_{ i}^2
\leq a_2(n-1)n \bigl( 4+ 2 V_n^2 \bigr) ,
\end{eqnarray*}
which further implies
\begin{eqnarray*}
\sum_{i=1}^nt_{ i}^2
+ a_2 n(n-1) \leq17 \sum_{i=1}^n
\bigl(v_{ i}^2+b_{
i}^2 \bigr) +
a_2 n(n-1) \leq a_2 (n-1)n \bigl( 103 + 102
V_n^2 \bigr).
\end{eqnarray*}
Substituting this into (\ref{pbesta1}) with $y\geq2$ proves (\ref
{tr-ineq-1}).

As for (\ref{tr-ineq-2}), let $\mathcal{F}_j = \sigma\{X_i: i\leq j\}
$ and write
\begin{eqnarray*}
\sum_{1\leq i<j\leq n}r_{i,j} = \sum
_{j=2}^n \sum_{i=1}^{j-1}r_{i,j}
= \sum_{j=2}^n R_j,\qquad
R_j =\sum_{i=1}^{j-1}
r_{i,j},\qquad 2\leq j\leq n.
\end{eqnarray*}
Note that $\{ R_j, \mathcal{F}_j, j\geq2\}$ is a martingale
difference sequence. Then using the sub-Gaussian inequality (\ref
{lsw}) for self-normalized martingales yields
\begin{equation}
\P\Biggl\{ \biggl\llvert\sum_{1\leq i<j\leq n}
r_{i,j} \biggr\rrvert> \sqrt{2y} \Biggl( Q_n^2+2
\widehat{Q}_n^2+3\sum_{j=2}^n
\e R_j^2 \Biggr)^{1/2} \Biggr\} \leq\sqrt{2}
e^{- y /4}, \label{st0}
\end{equation}
where
\[
Q_n^2 = \sum_{j=2}^n
R_j^2 ,\qquad\widehat{Q}^2_n=
\sum_{j=2}^n \e\bigl( R_j^2
\mid\mathcal{F}_{j-1} \bigr) . %
\]

Observe that $Q_n^2$ and $\Lambda_n^2$ have same structure, thus it
can be similarly proved that
\begin{eqnarray}
\qquad\P\bigl\{ Q_n^2 \geq a_2 y
n^2 \bigl( 102 V_n^2 + 103 \bigr) \bigr\}
\leq C a_2^{-1} e^{-y/4} \e\bigl(
r_{1,2}^2 \bigr) . \label{st1}
\end{eqnarray}
For $\widehat{Q}_n^2$, write
\begin{equation}
\hat{t}_j= u_j + 4 d_j\qquad\mbox{where }
u_{j}^2 =\sum_{i=1}^{j-1}
r_{i,j}^2,\qquad d_{j}^2 = \sum
_{i=1}^{j-1} \e\bigl( r_{i,j}^2
\mid X_j \bigr),\qquad 2\leq j\leq n, \label{hattj}
\end{equation}
then it follows from a conditional analogue of (\ref{lsw-3}) that
\begin{equation}
\e\bigl\{ R_j^2 I\bigl( R_j^2
\geq y \hat{t}^{ 2}_j \bigr) \mid X_j \bigr
\} \leq23 d_j^2 e^{-y/4} . \label{temp1}
\end{equation}
Therefore, for $y \geq2$,
%
\begin{eqnarray} \label{temp2}
&& \P\Biggl[ \widehat{Q}_n^2 > y \Biggl\{ \sum
_{j=2}^n \e\bigl( \hat{t}^{ 2}_j
\mid\mathcal{F}_{j-1} \bigr) + a_2 n (n-1) \Biggr\}
\Biggr]
\nonumber
\\
&&\quad \leq\P\biggl[ \frac{ \sum_{j=2}^n \e\{ R_j^2 I( R_j^2 \leq y
\hat{t}_j^{ 2} ) \mid \mathcal{F}_{j-1} \} } { \sum_{j=2}^n \e( \hat
{t}^{ 2}_j \mid \mathcal{F}_{j-1} ) } > y \biggr]
\nonumber\\[-8pt]\\[-8pt]\nonumber
&&\qquad{} +\P\Biggl[ \sum_{j=2}^n \e\bigl
\{ R_j^2 I\bigl( R_j^2 > y
\hat{t}_j^{ 2} \bigr) \mid\mathcal{F}_{j-1}
\bigr\} \geq y a_2 n(n-1) \Biggr]
\nonumber
\\
&&\quad \leq\frac{1}{ a_2 y n (n-1) } \sum_{j=2}^n \e
\bigl\{ R_j^2 I\bigl( R_j^2 >
y \hat{t}_j^{ 2} \bigr) \bigr\} \leq C
a_2^{-1} e^{-y/4} \e\bigl( r_{1,2}^2
\bigr),\nonumber
\end{eqnarray}
where in the last step we used (\ref{temp1}).

For $d_j^2$ and $u_j^2$ given in (\ref{hattj}), we have
\begin{eqnarray*}
\e\bigl( u_j^2 \mid\mathcal{F}_{j-1} \bigr)
&=& \sum_{i=1}^{j-1} \e\bigl(
r_{i,j}^2\mid X_i \bigr) \leq4
a_2(j-1) + 2 a_2 \sum_{i=1}^{j-1}
h_{1 i}^2 ,
\\
\e\bigl( d_j^2 \mid\mathcal{F}_{j-1} \bigr)
&=& \sum_{i=1}^{j-1} r_{i,j}^2
\leq2a_2(j-1) + 2a_2 \sum_{i=1}^{j-1}
\bigl( h_{1 i}^2 + h_{1 j}^2 \bigr),
\end{eqnarray*}
leading to
\begin{eqnarray*}
\sum_{j=2}^n \e\bigl(
\hat{t}_j^2 \mid\mathcal{F}_{j-1} \bigr)
\leq17 \sum_{j=2}^n \bigl\{ \e\bigl(
u_j^2 \mid\mathcal{F}_{j-1} \bigr) + \e
\bigl( d_j^2 \mid\mathcal{F}_{j-1} \bigr)
\bigr\} \leq a_2 (n-1) n \bigl( 104 + 136 V_n^2
\bigr) .
\end{eqnarray*}
Substituting this into (\ref{temp2}) yields
\begin{equation}
\P\bigl\{ \widehat{Q}_n^2 > a_2 y
n^2 \bigl( 136 V_n^2+ 104 \bigr) \bigr\} \leq
C a_2^{-1} e^{-y/4} \e\bigl( r_{1,2}^2
\bigr) . \label{st2}
\end{equation}

Together, (\ref{st0}), (\ref{st1}), (\ref{st2}) and the identity
$\sum_{j=2}^n \e R_j^2 = \frac{1}{2}n(n-1) \e( r_{1,2}^2 ) $ prove
(\ref{tr-ineq-2}).

\item[\textit{Step} 2.] Assume $m=3$. By (\ref{psi-m}),
\begin{equation}
r^2(x_1,x_2,x_3) \leq2
a_3 \bigl\{ 1+h_1^2(x_1)+h_1^2(x_2)+h_1^2(x_3)
\bigr\} \label{kdm3}
\end{equation}
and for $r_2(x_1, x_2)=E\{r(X_1, X_2, X_3) \mid X_1=x_1, X_2=x_2\}$,
\begin{equation}
r_2^2(x_1,x_2) \leq2
a_3 \bigl\{ 2+h_1^2(x_1)+h_1^2(x_2)
\bigr\}. \label{kdm3prime}
\end{equation}

Again, starting from $\Lambda_n^{2}=\sum_{i=1}^n\psi_i^2$ with
%
\begin{eqnarray} \label{s12}
\psi_i &=& \mathop{\sum_{ 1\leq j<k \leq n}}_{j , k
\neq i}
r(X_i,X_j,X_k) := \mathop{\sum
_{1\leq j<k \leq n}}_{j , k \neq i } r_{i,
j, k}
\nonumber
\\
&=& \mathop{\sum_{j=2}}_{j\neq i}^n \mathop{\sum
_{k=1}}_{k\neq i}^{j-1} ( r_{i,j,k}-r_{i,j} )
+\mathop{\sum_{j=2}}_{j\neq i}^{n} \mathop{\sum
_{k=1}}_{k\neq i }^{j-1} r_{i,j}
\\
&:=& \mathop{\sum_{j=2}}_{j\neq i}^n R_{i ,j} +
\mathop{\sum_{j=2}}_{j\neq
i}^{n} \bigl\{ j-1-1(j>i) \bigr\}
r_{i,j}.\nonumber
\end{eqnarray}
Conditional on $(X_i, X_j)$, $R_{i ,j} $ is a sum of independent random
variables with zero means. Define $t_{i,j} = v_{ i,j} + 4 b_{i,j} $,
where
\begin{eqnarray*}
t_{ i, j}^2 &=& \mathop{\sum_{k=1}}_{k \neq i}^{j-1}
(r_{i,j,k}-r_{i,j} )^2 = \mathop{\sum
_{k=1}}_{k \neq i}^{j-1} (h_{3, ijk}- h_{2,
ij} -
h_{1 k})^2,
\\
b_{i,j}^2 &=& \mathop{\sum_{k=1}}_{k \neq i}^{j-1}
\e\bigl\{ (r_{i,j,k}-r_{i,j})^2 \mid
X_i, X_j \bigr\} =\mathop{\sum_{k=1}}_{k \neq i}^{j-1}
\bigl[ \e\bigl\{ (h_{3, ijk}- h_{1 k})^2 \mid
X_i, X_j \bigr\} - h_{2,ij}^2
\bigr].
\end{eqnarray*}
Applying (\ref{lsw-3}) conditional on $(X_i, X_j)$ gives
\[
\e\bigl\{ R_{i ,j}^2 I( R_{i ,j} \geq\sqrt{y}
t_{i,j} ) \mid X_i,X_j \bigr\} \leq23
b_{i,j}^2 e^{-y/4} . %
\]
Then it follows from Lemma~\ref{sumineq} that
\begin{eqnarray}
&& \P\Biggl\{ \sum_{i=1}^n \Biggl( \sum
_{j=2 , j\neq i}^n R_{i ,j}
\Biggr)^2 \geq y n \Biggl( \sum_{i=1}^n
\sum_{j=2 , j\neq i}^n t_{i,j}^2
+ a_3 n^3 \Biggr) \Biggr\}
\nonumber
\\
&&\quad \leq\P\Biggl\{ \sum_{i=1}^n \sum
_{j=2 , j\neq i}^n R_{i ,j}^2
\geq y \Biggl( \sum_{i=1}^n\sum
_{j=2, j\neq i}^n t^2_{i,j} +
a_3 n^3 \Biggr) \Biggr\}
\nonumber
\\
&&\quad \leq C \frac{ e^{- y/4}}{ a_3 n^3 } \sum_{i=1}^n
\sum_{j=2 , j\neq i}^n (j-1) \e\bigl(
r_{1,2, 3}^2 \bigr) \leq C a_3^{-1}
e^{-y/4} \e\bigl( r_{1,2,3}^2 \bigr) .
\nonumber
\end{eqnarray}
This, combined with the inequality $ \sum_{i=1}^n\sum_{j=2, j\neq
i}^n t^2_{i,j} \leq a_3 n^3 ( B_1 + B_2 V_n^2)$ implies
\begin{eqnarray}
\P\Biggl\{ \sum_{i=1}^n \Biggl( \sum
_{j=2 , j\neq i}^n R_{i ,j}
\Biggr)^2 \geq a_3 y n^4 \bigl(
B_1+ 1+ B_2 V_n^2 \bigr) \Biggr
\} \leq C a_3^{-1} e^{-y/4} \e\bigl(
r_{1,2,3}^2 \bigr). \label{dti1}
\end{eqnarray}

For the second addend in (\ref{s12}), consider $ \widetilde
{r}_{i,j}=\{ j-1-I(j>i ) \} r_{i,j} $ as a new (degenerate) kernel
satisfying $\e( \widetilde{r}_{i,j}\mid X_i )= \e( \widetilde
{r}_{i,j}\mid X_j )=0$. Then by similar arguments as in step~1, we obtain
%
\begin{eqnarray}\label{l51-02}
&& \P\Biggl( \sum_{i=1}^n \Biggl[ \sum
_{j=2 , j\neq i}^n \bigl\{ j-1-1(j>i ) \bigr\}
r_{i,j} \Biggr]^2 \geq a_3 y n^4
\bigl(B_3 + B_4 V_n^2\bigr)
\Biggr)
\nonumber\\[-8pt]\\[-8pt]\nonumber
&&\qquad  \leq C a_3^{-1} e^{-y/4} \e\bigl(
r_{1,2,3}^2 \bigr).
\end{eqnarray}

Together, (\ref{s12}), (\ref{dti1}) and (\ref{l51-02}) prove (\ref
{tr-ineq-1}).

To prove (\ref{tr-ineq-2}) for $m=3$, consider the following decomposition:
%
\begin{eqnarray}\label{p1de}
&& \sum_{1\leq i_1<i_2<i_3\leq n}r(X_{i_1},X_{i_2},X_{i_3})\nonumber
\\
&&\quad
=\sum_{1\leq i_1<i_2<i_3\leq n}r_{i_1,i_2,i_3}
\nonumber
\\
&&\quad = \sum_{k=3}^n\sum
_{1\leq i_1<i_2<k} ( r_{i_1,i_2,k}-r_{i_1,i_2}) +\sum
_{k=3}^n\sum_{1\leq
i_1<i_2<k}r_{i_1,i_2}\nonumber
\\
&&\quad = \sum_{k=3}^n \sum
_{1\leq i_1<i_2<k} ( r_{i_1,i_2,k}-r_{i_1,i_2} ) + \sum
_{j=2}^{n-1} \sum_{i=1}^{j-1}(n-j)
r_{i,j}
\\
&&\quad  = \sum_{k=3}^n\sum
_{j=2}^{k-1} \sum_{i=1}^{j-1}
( r_{i,j,k}-r_{i,j}-r_{j,k} ) + \sum
_{k=3}^n \sum_{j=2}^{k-1}
(j-1) r_{j,k} + \sum_{j=2}^{n-1}
\sum_{i=1}^{j-1}(n-j) r_{i,j}
\nonumber
\\
&&\quad := \sum_{k=3}^n \sum
_{j=2}^{k-1} r^*_{1, j k} + \sum
_{k=3}^n \sum_{j=2}^{k-1}
r^*_{2, jk} + \sum_{j=2}^{n-1}
r^*_{ j}, \nonumber
\end{eqnarray}
where
\[
r_{1, jk}^*= \sum_{i=1}^{j-1}(
r_{i,j,k}-r_{i,j}-r_{j,k} ),\qquad
r_{2, jk}^* = (j-1) r_{j,k}\quad\mbox{and}\quad r^*_{j}=
\sum_{i=1}^{j-1}(n-j) r_{i,j}.
\]

Put $R^*_k = R^*_{1,k} + R^*_{2,k}$, $R^*_{1,k}=\sum_{j=2}^{k-1}
r^*_{1, j k}$ and $R^*_{2,k} = \sum_{j=2}^{k-1} r^*_{2,jk} $. We see
that $\{R^*_k, \mathcal{F}_k, k \geq3\}$ is a sequence of martingale
differences, and by (\ref{lsw}),
%
\begin{eqnarray}
\P\Biggl( \Biggl\llvert\sum_{k=3}^n
R^*_k \Biggr\rrvert\geq\sqrt{2 y} \Biggl[ \sum
_{k=3}^n \bigl\{ R^*_k +2\e\bigl(
R^{* 2}_k \mid\mathcal{F}_{k-1} \bigr) +3 \e
R^{* 2}_k \bigr\} \Biggr]^{1/2} \Biggr) \leq
\sqrt{2} e^{-y/4}. \label{mi-2}
\end{eqnarray}
Note that conditional on $(X_j,X_k)$, $r^*_{1, j k}$ is a sum of
independent random variables with zero means, and given $X_k$, $
r^*_{2, jk}$ are independent with zero means. Then it is
straightforward to verify that
\begin{equation}
\sum_{k=3}^n \e R^{* 2}_{k}
\leq2\sum_{k=3}^n(k-2)\sum
_{j=2}^{k-1}\e r_{1,j k }^{* 2} +2
\sum_{k=3}^n R^{* 2}_{2,k}
\leq C a_3 n^4. \label{se-3}
\end{equation}

Moreover, by noting the resemblance in structure between $R^*_{ k}$ and
$\psi_i$ (see (\ref{s12})), it can be shown that
\begin{eqnarray}
\P\Biggl\{ \sum_{k=3}^n
R^{* 2}_{ k} \geq a_3 y n^4
\bigl(B_5 + B_6 V_n^2 \bigr)
\Biggr\} \leq C e^{-y/4}, \label{sr-3}
\end{eqnarray}
which is analogous to (\ref{tr-ineq-1}).

It remains to bound the tail probability of $\sum_{k=3}^n \e( R^{*
2}_{k} \mid\mathcal{F}_{k-1})$. In view of (\ref{p1de}), let $t_{j,k}^*
= v_{j,k }^* + 4 b_{j,k}^*$ for $2\leq j<k\leq n$, where
\begin{eqnarray*}
v_{j,k}^{* 2} = \sum_{i=1}^{j-1}
(r_{i,j,k} - r_{i,j} - r_{j,k} )^2 ,\qquad
b_{j,k}^{* 2} = \sum_{i=1}^{j-1}
\e\bigl\{ (r_{i,j,k} - r_{i,j} - r_{j,k}
)^2\mid X_j , X_k\bigr\} ,
\end{eqnarray*}
and for $3\leq k\leq n$, put
\begin{eqnarray*}
t^*_{ k} = v^*_{k} + 4 b^*_{ k },\qquad
v_{ k}^{* 2} =\sum_{j=2}^{k-1}
r_{2, j k}^{* 2},\qquad b^*_{ k}= \sum
_{j=2}^{k-1} \e\bigl( r_{2, j k}^{* 2}
\mid X_k \bigr) .
\end{eqnarray*}

Recall that $R^*_k = R^*_{1,k} + R^*_{2,k}= \sum_{j=2}^{k-1} ( r^*_{1,
j k} + r^*_{2,jk})$. We proceed in a similar manner as in~(\ref{temp2}):
\begin{eqnarray*}
&& \sum_{k=3}^n \e\bigl(
R^{* 2}_{k} \mid\mathcal{F}_{k-1} \bigr)
\\
&&\quad \leq2 \sum_{k=3}^n(k-2)\sum
_{j=2}^{k-1}\e\bigl( r_{1,jk}^{* 2}
\mid\mathcal{F}_{k-1} \bigr) + 2 \sum_{k=3}^n
\e\bigl( R^{* 2}_{2,k} \mid\mathcal{F}_{k-1}
\bigr)
\\
&&\quad  = 2 \sum_{k=3}^n\sum
_{j=2}^{k-1} (k-2) \e\bigl[ r_{1,j k }^{* 2}
\bigl\{ I\bigl( \bigl\llvert r_{1,j k}^{*} \bigr\rrvert\leq
\sqrt{y} t^{*}_{j,k} \bigr) + I\bigl( \bigl\llvert
r_{1,jk}^{* } \bigr\rrvert> \sqrt{y} t^{* }_{j,k}
\bigr) \bigr\} | \mathcal{F}_{k-1}\bigr]
\\
&&\qquad{} + 2\sum_{k=3}^n \e\bigl[
R^{* 2}_{2,k} \bigl\{ I\bigl(\bigl\llvert
R^{*}_{2,k}\bigr\rrvert\leq\sqrt{y} t^{* }_{k}
\bigr) + I\bigl( \bigl\llvert R^{*}_{2,k} \bigr\rrvert>
\sqrt{y} t_{k}^{*}\bigr) \bigr\} \mid
\mathcal{F}_{k-1} \bigr].
\end{eqnarray*}
By (\ref{lsw-3}) and the Markov inequality, we have (recall that
$y\geq2$)
%
\begin{eqnarray}\label{cse-2}
&& \P\Biggl[ \sum_{k=3}^n (k-2) \sum
_{j=2}^{k-1} \e\bigl\{ r_{1,j k}^{* 2}
I\bigl( \bigl\llvert r^*_{1,jk} \bigr\rrvert> \sqrt{y}
t^{*}_{j,k} \bigr) \mid\mathcal{F}_{k-1} \bigr\}
\geq a_3 y n^4 \Biggr]
\nonumber\\[-8pt]\\[-8pt]\nonumber
&&\quad \leq\bigl(a_3 y n^4 \bigr)^{-1} \sum
_{k=3}^n (k-2) \sum_{j=2}^{k-1}
\e\bigl\{ r_{1,j k}^{* 2} I\bigl( \bigl\llvert
r_{1,jk}^*\bigr\rrvert> \sqrt y t^{* }_{j,k}
\bigr) \mid\mathcal{F}_{k-1} \bigr\} \leq C e^{-y/4 }
\end{eqnarray}
and
%
\begin{eqnarray}\label{cse-3}
&& \P\Biggl[ \sum_{k=3}^n \e\bigl\{
R_{2,k}^{* 2} I\bigl( \bigl\llvert R_{2,k}^*\bigr
\rrvert> \sqrt y t^{* }_{k} \bigr) \mid
\mathcal{F}_{k-1} \bigr\} \geq a_3 y n^4 \Biggr]
\nonumber\\[-8pt]\\[-8pt]\nonumber
&&\quad \leq\bigl(a_3 y n^4\bigr)^{-1} \sum
_{k=3}^n \e\bigl\{ R_{2,k}^{* 2}
I\bigl( \bigl\llvert R_{2,k}^* \bigr\rrvert> \sqrt y
t^{* }_{k} \bigr) \mid\mathcal{F}_{k-1} \bigr\}
\leq C e^{-y/4}.
\end{eqnarray}
However, it follows from (\ref{kdm3}) and (\ref{kdm3prime}) that
%
\begin{eqnarray}
\sum_{k=3}^n(k-2) \sum
_{j=2}^{k-1} \e\bigl\{ r_{1,j k}^{* 2}
I\bigl( \bigl\llvert r_{1,jk}^{* }\bigr\rrvert\leq\sqrt y
t^{*}_{j,k} \bigr) \mid\mathcal{F}_{k-1} \bigr\}
&\leq& a_3 y n^4 \bigl( B_7 +
B_8 V_n^2 \bigr), \label{cse-4}
\\
\sum_{k=3}^n \e\bigl\{
R^{* 2}_{2,k} I\bigl( \bigl\llvert R^*_{2,k}\bigr
\rrvert\leq\sqrt y t_{ k}^{*
} \bigr) \mid
\mathcal{F}_{k-1} \bigr\} & \leq& a_3 y n^4
\bigl(B_9 + B_{10} V_n^2 \bigr).
\label{cse-5}
\end{eqnarray}

Assembling (\ref{mi-2})--(\ref{cse-5}), we obtain
\[
\P\Biggl\{ \Biggl\llvert\sum_{k=3}^n
R_{ k}^* \Biggr\rrvert\geq\sqrt{a_3} y n^2
\bigl( B_{11} + B_{12} V_n^2
\bigr)^{1/2} \Biggr\} \leq C e^{-y/4}. %
\]
By induction, a similar result holds for $\sum_{j=2}^{n-1} r_j^*$;
that is,
\[
\P\Biggl\{ \Biggl\llvert\sum_{j=2}^n
r_j^* \Biggr\rrvert\geq\sqrt{a_3} y n^2
\bigl( B_{13} + B_{14} V_n^2
\bigr)^{1/2} \Biggr\} \leq C e^{-y/4}. %
\]
This completes the proof of (\ref{tr-ineq-2}) for $m=3$.

\item[\textit{Step} 3.] For a general $3 < m < n/2$,
\begin{equation}
r_k^2(x_1,\ldots,x_k) \leq2
a_m \Biggl\{ m-k+1+\sum_{j=1}^k
h_1^2(x_j) \Biggr\}, \label{kdm}
\end{equation}
where $r_k(x_1,\ldots,x_k)=E\{r(X_1,\ldots,X_m)\mid X_1=x_1,\ldots
,X_k=x_k\}$ for $k=2,\ldots,m$.

To use the induction, we need the following string of equalities:
%
\begin{eqnarray}\label{de-m}
\psi_i &=& \mathop{\sum_{1\leq\ell_1 < \cdots<\ell_{m-1} \leq n}}_{\ell
_1, \ldots, \ell_{m-1} \neq i }
r_{ \ell_1,\ldots,\ell_{m-1},i}
\nonumber
\\
&=& \mathop{\sum_{\ell_{m-1}=m-1}}_{\ell_{m-1}\neq
i}^{n}\mathop{\sum
_{1\leq
{\ell}_1<\cdots< {\ell}_{m-2}< {\ell}_{m-1}}}_{{\ell}_1,\ldots, {\ell}_{m-2}\neq i } ( r_{ {\ell}_1,\ldots, {\ell
}_{m-2}, \ell_{m-1} , i}-r_{ {\ell}_2,\ldots, {\ell}_{m-1},i} )
\nonumber\\[-8pt]\\[-8pt]\nonumber
&&{} + \mathop{\sum_{2\leq{\ell}_2<\cdots< {\ell}_{m-1} \leq n}}_{{\ell}_2,\ldots, {\ell}_{m-1}\neq i } \bigl\{ {\ell}_2-1-
1( i < {\ell}_2 ) \bigr\} r_{ {\ell}_2,\ldots, {\ell
}_{m-1},i}
\nonumber
\\
&:=& \psi_{1,i}+\psi_{2,i}. \nonumber
\end{eqnarray}
Moreover,
\begin{eqnarray}
\psi_{1,i} &=& \mathop{\sum_{{\ell}_{m-1}=m-1}}_{{\ell
}_{m-1} \neq i }^{n}
\mathop{\sum_{1\leq{\ell}_1<\cdots<{\ell}_{m-2} <{\ell}_{m-1}}}_{{\ell
}_1,\ldots,{\ell}_{m-2} \neq i} ( r_{ {\ell}_1,\ldots,{\ell}_{m-2},{\ell
}_{m-1} ,i }- r_{ {\ell
}_2,\ldots,{\ell}_{m-1} ,i} )
\nonumber
\\
& =& \mathop{\sum_{{\ell}_{m-1}=m-1}}_{{\ell
}_{m-1}\neq i}^{n} \mathop{\sum
_{1\leq{\ell}_1<\cdots<{\ell}_{m-2}<{\ell}_{m-1}}}_{{\ell
}_1,\ldots,{\ell}_{m-2} \neq i} \breve{r}_{{\ell}_1,\ldots, \ell_{m-1},i}
\nonumber
\\
&=& \mathop{\sum_{{\ell}_{m-1}=m-1}}_{{\ell}_{m-1}
\neq i}^{n} \mathop{\sum_{{\ell}_{m-2}=m-2}}_{{\ell}_{m-2}\neq i
}^{{\ell
}_{m-1}-1}\ldots\mathop{\sum_{{\ell}_2=2}}_{{\ell}_2
\neq i }^{{\ell}_3-1}
\Biggl( \mathop{\sum_{{\ell}_1=1 }}_{{\ell}_1 \neq
i}^{{\ell}_2-1} \breve
{r}_{{\ell}_1,\ldots, \ell_{m-1},i} \Biggr)
\nonumber
\\
& =& \mathop{\sum_{{\ell}_{m-1}=m-1}}_{{\ell}_{m-1}
\neq i }^{n} \mathop{\sum
_{{\ell}_{m-2}=m-2}}_{{\ell}_{m-2} \neq i
}^{{\ell
}_{m-1}-1}\ldots \mathop{\sum_{{\ell}_2=2}}_{{\ell}_2
\neq i }^{{\ell}_3-1}
\breve{R}_{{\ell}_2, \ldots,\ell_{m-1},i}
\nonumber
\end{eqnarray}
with
\[
\breve{r}_{{\ell}_1,\ldots, \ell_{m-1}}= r_{ {\ell}_1,\ldots
,{\ell}_{m-2},{\ell}_{m-1} ,i}-r_{ {\ell}_2,\ldots,{\ell
}_{m-1},i},\qquad
\breve{R}_{{\ell}_2, \ldots,\ell_{m-1},i} =\mathop{\sum_{{\ell}_1=1
}}_{{\ell}_1 \neq i}^{{\ell}_2-1}
\breve{r}_{{\ell}_1,\ldots, \ell_{m-1},i}. %
\]

Conditional on $(X_i,X_{{\ell}_2},\ldots,X_{{\ell}_{m-1}})$, $\breve
{R}_{{\ell}_2, \ldots,\ell_{m-1},i}$ is a sum of independent random
variables with zero means. Also, it is straightforward to verify that
\begin{eqnarray}
\psi_{1,i}^2 \leq{n-1 \choose m-2} \mathop{\sum
_{{\ell}_{m-1}=m-1}}_{{\ell}_{m-1} \neq
i}^{n} \mathop{\sum_{{\ell}_{m-2}=m-2}}_{{\ell}_{m-2}
\neq i}^{{\ell
}_{m-1}-1}
\ldots\mathop{\sum_{{\ell}_2=2}}_{{\ell}_2\neq
i}^{{\ell
}_3-1} \breve{R}_{{\ell}_2, \ldots,\ell_{m-1},i}^2.
\nonumber
\end{eqnarray}
Next, let $\breve{t}_{\ell}=\breve{v}_{\ell} + 4 \breve{b}_{\ell
}$, where
\begin{eqnarray*}
\breve{v}_{\ell} = \sum_{{\ell_1}=1 , \ell_1 \neq i}^{{\ell}-1}
\breve{r}_{\ell_1, \ldots, \ell_{m-1}, i}^2 , \qquad\breve{b}_{\ell
}^2
= \sum_{\ell_1=1 , \ell_1\neq i}^{{\ell}-1} \e\bigl(
\breve{r}_{\ell_1, \ldots, \ell_{m-1}, i}^2 \mid X_i, X_{\ell},
X_{\ell_3}, \ldots, X_{\ell_{m-1}}\bigr) .
\end{eqnarray*}
Similar to the proof of (\ref{dti1}), we derive from Lemma~\ref{l00}
that for every $y\geq2$,
\begin{eqnarray*}
{n-1 \choose m-2}^{-1} \sum_{i=1}^n
\psi_{1,i}^2 \leq y \Biggl\{ a_m {n-1 \choose
m-1} + \sum_{i=1}^n \mathop{\sum
_{{\ell}_{m-1}=m-1}}_{{\ell
}_{m-1} \neq i}^{n}\ldots\mathop{\sum_{{\ell}_2=2}}_{
{\ell
}_2\neq i}^{{\ell}_3-1}
\breve{t}_{{\ell}_2}^{ 2} \Biggr\}
\end{eqnarray*}
holds with probability at least $1- C \exp(-y/4)$. This, together with
the following inequality
\begin{eqnarray*}
\sum_{i=1}^n\mathop{\sum
_{{\ell}_{m-1}=m-1}}_{{\ell}_{m-1} \neq
i}^{n}\ldots\mathop{\sum_{{\ell}_2=2}}_{{\ell}_2\neq i}^{{\ell}_3-1}
\breve{t }_{{\ell}_2}^{ 2} \leq a_m {n \choose m }
\bigl( B_{15} + B_{16} V_n^2 \bigr)
\end{eqnarray*}
which can be obtained by using (\ref{kdm}) repeatedly, gives
\begin{equation}
\P\Biggl\{ \sum_{i=1}^n
\psi_{1,i}^2 \geq a_m y n^{2m -2 }\bigl(
B_{17} +B_{18} V_n^2 \bigr) \Biggr
\} \leq C e^{-y/4 }. \label{pt1}
\end{equation}

For $\psi_{2,i}$, note that the summation is carried out over all
$(m-2)$-tuples and
\[
\bigl\llvert\bigl\{ {\ell}_2-1-1( i<{\ell}_2)\bigr\}
r_{ {\ell}_2,\ldots,{\ell}_{m-1},i} \bigr\rrvert\leq n \llvert r_{
{\ell}_2,\ldots,{\ell}_{m-1},i}\rrvert.
\]
Regarding $ \{ {\ell}_2-1-1( i<{\ell}_2) \} r_{ {\ell}_2,\ldots
,{\ell}_{m-1},i}$ as a (weighted) degenerate kernel with $(m-1)$
arguments, it follows from induction that
\begin{equation}
\P\Biggl\{ \sum_{i=1}^n
\psi_{2,i}^2 \geq a_m y n^{2m-2} \bigl(
B_{19} +B_{20} V_n^2 \bigr) \Biggr
\} \leq C e^{-y/4}. \label{pt2}
\end{equation}
Assembling (\ref{de-m}), (\ref{pt1}) and (\ref{pt2}) yields (\ref
{tr-ineq-1}).

Similarly, using the decomposition
\begin{eqnarray*}
&& \sum_{1\leq i_1<\cdots<i_m\leq n} r(X_{i_1},
\ldots,X_{i_m})
\\[-1pt]
&&\quad  =\sum_{1\leq i_1<\cdots<i_m\leq n}r_{i_1,\ldots,i_m}
\\[-1pt]
&&\quad = \sum_{k=m}^n\sum
_{1\leq i_1<\cdots<i_{m-1}<k} ( r_{i_1,\ldots,i_{m-1},k}-r_{i_1,\ldots
,i_{m-1}} )
+ \sum_{1\leq i_1<\cdots<i_{m-1}\leq
n-1}(n-i_{m-1})r_{i_1,\ldots,i_{m-1}}.
\end{eqnarray*}
Because $\e( r_{i_1,\ldots,i_{m-1},k}\mid\mathcal
{F}_{k-1})=r_{i_1,\ldots,i_{m-1}}$,
\[
\biggl\{ R^*_k:= \sum_{1\leq i_1<\cdots<i_{m-1}\leq k} (
r_{i_1,\ldots,i_{m-1},k}-r_{i_1,\ldots,i_{m-1}}), \mathcal{F}_{k} \biggr
\}_{ k \geq m} %
\]
is a martingale difference sequence, such that the following analogue
of (\ref{mi-2}) holds:
\begin{eqnarray*}
\P\Biggl( \Biggl\llvert\sum_{k=m}^n
R^*_{k} \Biggr\rrvert\geq\sqrt{2 y} \Biggl[ \sum
_{k=m}^n \bigl\{ R^{* 2}_k
+2 \e\bigl( R^{* 2}_k \mid\mathcal{F}_{k-1}
\bigr) +3 \e R^{* 2}_k \bigr\} \Biggr]^{1/2}
\Biggr) \leq\sqrt{2} e^{-y/4}.
\end{eqnarray*}
For $m\leq k\leq n$ fixed, extending (\ref{p1de}) gives
\begin{eqnarray*}
R^*_k &=& \sum_{1\leq i_1<\cdots<i_{m-1}<k} (
r_{i_1,\ldots
,i_{m-1},k}-r_{i_1,\ldots,i_{m-1}} )
\\[-1pt]
&=& \sum_{i_{m-1}=m-1}^{k-1}\ldots\sum
_{i_1=1}^{i_2-1} ( r_{i_1 ,
i_2,\ldots,i_{m-1},k}-r_{i_1,\ldots,i_{m-1}}
-r_{i_2,\ldots,i_{m-1},k}+r_{i_2,\ldots,i_{m-1}} )
\\[-1pt]
&&{} + \sum_{i_{m-1}=m-1}^{k-1}\ldots\sum
_{i_2=2}^{i_3-1}w_{2 } (
r_{i_2,\ldots,i_{m-1},k}-r_{i_2,\ldots,i_{m-1}} -r_{i_3,\ldots
,i_{m-1},k}+ r_{i_3,\ldots,i_{m-1}} )
\\[-1pt]
&&{} + \cdots+ \sum_{i_{m-1}=m-1}^{k-1}w_{m-1 }
r_{i_{m-1},k},
\end{eqnarray*}
where $w_{j}:={i_{j} -1 \choose j-2}$ for $2\leq j\leq m-1$, and set
$w_1 \equiv1$ for convention. Moreover, for $1\leq j\leq m-2$, put
\begin{eqnarray}
r^*_{j, i_{j+1} , \ldots, i_{m-1} , k} = \sum_{i_j=j}^{i_{j+1}-1}
w_j ( r_{i_j, \ldots,i_{m-1}, k}-r_{i_j, \ldots,i_{m-1}}
-r_{i_{j+1},\ldots, i_{m-1} , k}+r_{i_{j+1},\ldots,i_{m-1}}
)
\nonumber
\end{eqnarray}
and $r^*_{m-1, k}= \sum_{i_{m-1}=m-1}^{k-1}w_{m-1} r_{i_{m-1}, k}$,
such that
%
\begin{eqnarray}\label{decR*k}
R^*_{k} & =& \sum_{2\leq i_2 < \cdots< i_{m-1} \leq k-1 }
r^*_{1, i_2
, \ldots, i_{m-1} , k}
\nonumber\\[-8pt]\\[-8pt]\nonumber
&&{}+\sum_{ 3\leq i_3 <\cdots< i_{m-1} \leq k-1} r^*_{2, i_{3}
, \ldots, i_{m-1} , k} + \cdots+
r^*_{m-1, k} .
\end{eqnarray}
For $j=1,\ldots, m-2$, conditional on $(X_{i_{j+1}},\ldots
,X_{i_{m-1}}, X_k)$, $r^*_{j, i_{j+1} , \ldots, i_{m-1} , k} $ is a
sum of independent random variables with zero means, and so is
$r^*_{m-1,k}$ conditional on $X_k$.

In particular, we have
\begin{eqnarray*}
\sum_{k=m}^n \e R^{* 2}_k
& \leq&(m-1) \sum_{j=m}^n \biggl\{ \e
\biggl( \sum_{2\leq i_2 < \cdots< i_{m-1} \leq k-1 } r^*_{1, i_2
, \ldots, i_{m-1} , k}
\biggr)^2
\\
&&{}+\e\biggl( \sum_{ 3\leq i_3 <\cdots< i_{m-1} \leq k-1} r^*_{2,
i_{3} , \ldots, i_{m-1} , k}
\biggr)^2 + \cdots+ \e r^{* 2}_{m-1,k} \biggr\}
\\
&\leq&(m-1)\sum_{k=m}^n \biggl\{ \pmatrix{k-2
\cr m-2} \sum_{2 \leq i_2 < \cdots< i_{m-1} \leq k-1 } \e r^{*
2}_{1,i_2,\ldots,i_{m-1},k}
\\
&&{} + \pmatrix{k-3 \cr m-3}\sum_{3\leq i_3 < \cdots< i_{m-1} \leq
k-1 } \e
r^{* 2}_{2,i_3,\ldots,i_{m-1},k} + \cdots+ \e r^{* 2}_{m-1, k}
\biggr\}
\\
&\leq& C (m-1)\e\bigl\{ r^2(X_1,\ldots,X_m)
\bigr\} \sum_{k=m}^n \Biggl\{ \pmatrix{k-2 \cr
m-2} \pmatrix{k-1 \cr m-1}
\\
&&{}+ \pmatrix{k-3 \cr m-3} \sum_{2\leq i_2 < \cdots< i_{m-1} \leq
k-1 }
(i_2-1)^2 + \cdots+ \sum_{i=m-1}^{k-1}
\pmatrix{i-1 \cr m-2}^2 \Biggr\}
\\
&\leq& C a_m n^{2m-2},
\end{eqnarray*}
which extends inequality (\ref{se-3}). In view of (\ref{decR*k}),
inequalities (\ref{sr-3})--(\ref{cse-5}) can be similarly extended by
using Lemmas~\ref{l00} and~\ref{sumineq} in the same way as in
step~2. The proof of Lemma~\ref{tr-lm} is then complete.
\end{longlist}
\end{appendix}

\section*{Acknowledgements}
The authors sincerely thank the Editor, Associate Editor
and anonymous referees for their constructive comments that led to
substantial improvement of the paper.

Qi-Man Shao was supported by Hong Kong Research Grants Council GRF
603710 and 403513.
Wen-Xin Zhou was supported by NIH R01GM100474-4 and a grant from the
Australian Research Council.


%

\printhistory

\begin{thebibliography}{39}

\bibitem{AlberinkBentkus2001}
%
\begin{barticle}[auto]
\bauthor{\bsnm{Alberink},~\bfnm{I.~B.}\binits{I.B.}} \AND
\bauthor{\bsnm{Bentkus},~\bfnm{V.}\binits{V.}}
(\byear{2001}).
\btitle{Berry--{E}sseen bounds for von {M}ises and {$U$}-statistics}.
\bjournal{Lith. Math. J.}
\bvolume{41}
\bpages{1--16}.
\bid{doi={10.1023/A:1011066719481}, issn={0132-2818}, mr={1849804}}
\end{barticle}
%

\bptok{imsref}%
\endbibitem

\bibitem{AlberinkBentkus2002}
%
\begin{barticle}[auto]
\bauthor{\bsnm{Alberink},~\bfnm{I.~B.}\binits{I.B.}} \AND
\bauthor{\bsnm{Bentkus},~\bfnm{V.}\binits{V.}}
(\byear{2002}).
\btitle{Lyapunov type bounds for {$U$}-statistics}.
\bjournal{Theory Probab. Appl.}
\bvolume{46}
\bpages{571--588}.
\bid{doi={10.1137/S0040585X97979299}, issn={0040-361X}, mr={1971830}}
\end{barticle}
%

\bptok{imsref}%
\endbibitem

\bibitem{Arvesen1969}
%
\begin{barticle}[mr]
\bauthor{\bsnm{Arvesen},~\bfnm{James~N.}\binits{J.N.}}
(\byear{1969}).
\btitle{Jackknifing {$U$}-statistics}.
\bjournal{Ann. Math. Statist.}
\bvolume{40}
\bpages{2076--2100}.
\bid{issn={0003-4851}, mr={0264805}}
\end{barticle}
%

\bptok{imsref}%
\endbibitem

\bibitem{BentkusGotze1996}
%
\begin{barticle}[mr]
\bauthor{\bsnm{Bentkus},~\bfnm{V.}\binits{V.}} \AND
\bauthor{\bsnm{G{\"o}tze},~\bfnm{F.}\binits{F.}}
(\byear{1996}).
\btitle{The {B}erry--{E}sseen bound for student's statistic}.
\bjournal{Ann. Probab.}
\bvolume{24}
\bpages{491--503}.
\bid{doi={10.1214/aop/1042644728}, issn={0091-1798}, mr={1387647}}
\end{barticle}
%

\bptok{imsref}%
\endbibitem

\bibitem{Bickel1974}
%
\begin{barticle}[mr]
\bauthor{\bsnm{Bickel},~\bfnm{P.~J.}\binits{P.J.}}
(\byear{1974}).
\btitle{Edgeworth expansions in nonparametric statistics}.
\bjournal{Ann. Statist.}
\bvolume{2}
\bpages{1--20}.
\bid{issn={0090-5364}, mr={0350952}}
\end{barticle}
%

\bptok{imsref}%
\endbibitem

\bibitem{BorovskikhWeber2003a}
%
\begin{barticle}[auto]
\bauthor{\bsnm{Borovskikh},~\bfnm{Y.~V.}\binits{Y.V.}} \AND
\bauthor{\bsnm{Weber},~\bfnm{N.~C.}\binits{N.C.}}
(\byear{2003}).
\btitle{Large deviations of {$U$}-statistics. {I}}.
\bjournal{Lith. Math. J.}
\bvolume{43}
\bpages{11--33}.
\bid{doi={10.1023/A:1022911005006}, issn={0132-2818}, mr={1996751}}
\end{barticle}
%

\bptok{imsref}%
\endbibitem

\bibitem{BorovskikhWeber2003b}
%
\begin{barticle}[auto]
\bauthor{\bsnm{Borovskikh},~\bfnm{Yu.~V.}\binits{Y.V.}} \AND
\bauthor{\bsnm{Weber},~\bfnm{N.~C.}\binits{N.C.}}
(\byear{2003}).
\btitle{Large deviations of {$U$}-statistics. {II}}.
\bjournal{Lith. Math. J.}
\bvolume{43}
\bpages{241--261}.
\bid{doi={10.1023/A:1026185217832}, issn={0132-2818}, mr={2019542}}
\bptnote{check pages}%
\end{barticle}
%

\bptok{imsref}%
\endbibitem

\bibitem{CallaertJanssen1978}
%
\begin{barticle}[mr]
\bauthor{\bsnm{Callaert},~\bfnm{Herman}\binits{H.}} \AND
\bauthor{\bsnm{Janssen},~\bfnm{Paul}\binits{P.}}
(\byear{1978}).
\btitle{The {B}erry--{E}sseen theorem for {$U$}-statistics}.
\bjournal{Ann. Statist.}
\bvolume{6}
\bpages{417--421}.
\bid{issn={0090-5364}, mr={0464359}}
\end{barticle}
%

\bptok{imsref}%
\endbibitem

\bibitem{ChanWierman1977}
%
\begin{barticle}[mr]
\bauthor{\bsnm{Chan},~\bfnm{Y.-K.}\binits{Y.-K.}} \AND
\bauthor{\bsnm{Wierman},~\bfnm{John}\binits{J.}}
(\byear{1977}).
\btitle{On the {B}erry--{E}sseen theorem for {$U$}-statistics}.
\bjournal{Ann. Probab.}
\bvolume{5}
\bpages{136--139}.
\bid{mr={0433551}}
\end{barticle}
%

\bptok{imsref}%
\endbibitem

\bibitem{ChenGoldsteinShao2010}
%
\begin{bbook}[auto:parserefs-M02]
\bauthor{\bsnm{Chen},~\bfnm{L.~H.~Y.}\binits{L.H.Y.}},
\bauthor{\bsnm{Goldstein},~\bfnm{L.}\binits{L.}} \AND
\bauthor{\bsnm{Shao},~\bfnm{Q.-M.}\binits{Q.-M.}}
(\byear{2010}).
\btitle{Normal Approximation by Stein's Method}.
\blocation{Berlin}:
\bpublisher{Springer}.
\end{bbook}
%

\bptok{imsref}%
\endbibitem

\bibitem{ChenShao2001}
%
\begin{barticle}[mr]
\bauthor{\bsnm{Chen},~\bfnm{Louis~H.~Y.}\binits{L.H.Y.}} \AND
\bauthor{\bsnm{Shao},~\bfnm{Qi-Man}\binits{Q.-M.}}
(\byear{2001}).
\btitle{A non-uniform {B}erry--{E}sseen bound via {S}tein's method}.
\bjournal{Probab. Theory Related Fields}
\bvolume{120}
\bpages{236--254}.
\bid{doi={10.1007/PL00008782}, issn={0178-8051}, mr={1841329}}
\end{barticle}
%

\bptok{imsref}%
\endbibitem

\bibitem{ChenShao2007}
%
\begin{barticle}[mr]
\bauthor{\bsnm{Chen},~\bfnm{Louis~H.~Y.}\binits{L.H.Y.}} \AND
\bauthor{\bsnm{Shao},~\bfnm{Qi-Man}\binits{Q.-M.}}
(\byear{2007}).
\btitle{Normal approximation for nonlinear statistics using a
concentration inequality approach}.
\bjournal{Bernoulli}
\bvolume{13}
\bpages{581--599}.
\bid{doi={10.3150/07-BEJ5164}, issn={1350-7265}, mr={2331265}}
\end{barticle}
%

\bptok{imsref}%
\endbibitem

\bibitem{CsorgoSzyszkowiczWang2003}
%
\begin{barticle}[mr]
\bauthor{\bsnm{Cs{\"o}rg{\H{o}}},~\bfnm{Mikl{\'o}s}\binits{M.}},
\bauthor{\bsnm{Szyszkowicz},~\bfnm{Barbara}\binits{B.}} \AND
\bauthor{\bsnm{Wang},~\bfnm{Qiying}\binits{Q.}}
(\byear{2003}).
\btitle{Donsker's theorem for self-normalized partial sums processes}.
\bjournal{Ann. Probab.}
\bvolume{31}
\bpages{1228--1240}.
\bid{doi={10.1214/aop/1055425777}, issn={0091-1798}, mr={1988470}}
\end{barticle}
%

\bptok{imsref}%
\endbibitem

\bibitem{PenaLaiShao2009}
%
\begin{bbook}[mr]
\bauthor{\bsnm{de~la Pe{\~n}a},~\bfnm{Victor~H.}\binits{V.H.}},
\bauthor{\bsnm{Lai},~\bfnm{Tze~Leung}\binits{T.L.}} \AND
\bauthor{\bsnm{Shao},~\bfnm{Qi-Man}\binits{Q.-M.}}
(\byear{2009}).
\btitle{Self-Normalized Processes: Limit Theory and Statistical Applications}.
\bseries{Probability and Its Applications (New York)}.
\blocation{Berlin}:
\bpublisher{Springer}.
\bid{doi={10.1007/978-3-540-85636-8}, mr={2488094}}
\end{bbook}
%

\bptok{imsref}%
\endbibitem

\bibitem{EichelsbacherLowe1995}
%
\begin{barticle}[mr]
\bauthor{\bsnm{Eichelsbacher},~\bfnm{Peter}\binits{P.}} \AND
\bauthor{\bsnm{L{\"o}we},~\bfnm{Matthias}\binits{M.}}
(\byear{1995}).
\btitle{A large deviation principle for {$m$}-variate von
{M}ises-statistics and {$U$}-statistics}.
\bjournal{J. Theoret. Probab.}
\bvolume{8}
\bpages{807--824}.
\bid{doi={10.1007/BF02410113}, issn={0894-9840}, mr={1353555}}
\end{barticle}
%

\bptok{imsref}%
\endbibitem

\bibitem{Filippova1962}
%
\begin{barticle}[auto]
\bauthor{\bsnm{Filippova},~\bfnm{A.~A.}\binits{A.A.}}
(\byear{1962}).
\btitle{Mises' theorem on the asymptotic behavior of functionals of empirical
distribution functions and its statistical applications}.
\bjournal{Theory Probab. Appl.}
\bvolume{7}
\bpages{24--57}.
\bptnote{check pages}%
\end{barticle}
%

\bptok{imsref}%
\endbibitem

\bibitem{Friedrich1989}
%
\begin{barticle}[mr]
\bauthor{\bsnm{Friedrich},~\bfnm{Karl~O.}\binits{K.O.}}
(\byear{1989}).
\btitle{A {B}erry--{E}sseen bound for functions of independent random
variables}.
\bjournal{Ann. Statist.}
\bvolume{17}
\bpages{170--183}.
\bid{doi={10.1214/aos/1176347009}, issn={0090-5364}, mr={0981443}}
\end{barticle}
%

\bptok{imsref}%
\endbibitem

\bibitem{GineGotzeMason1997}
%
\begin{barticle}[mr]
\bauthor{\bsnm{Gin{\'e}},~\bfnm{Evarist}\binits{E.}},
\bauthor{\bsnm{G{\"o}tze},~\bfnm{Friedrich}\binits{F.}} \AND
\bauthor{\bsnm{Mason},~\bfnm{David~M.}\binits{D.M.}}
(\byear{1997}).
\btitle{When is the {S}tudent {$t$}-statistic asymptotically standard normal?}
\bjournal{Ann. Probab.}
\bvolume{25}
\bpages{1514--1531}.
\bid{doi={10.1214/aop/1024404523}, issn={0091-1798}, mr={1457629}}
\end{barticle}
%

\bptok{imsref}%
\endbibitem

\bibitem{GramsSerfling1973}
%
\begin{barticle}[mr]
\bauthor{\bsnm{Grams},~\bfnm{William~F.}\binits{W.F.}} \AND
\bauthor{\bsnm{Serfling},~\bfnm{R.~J.}\binits{R.J.}}
(\byear{1973}).
\btitle{Convergence rates for {$U$}-statistics and related statistics}.
\bjournal{Ann. Statist.}
\bvolume{1}
\bpages{153--160}.
\bid{issn={0090-5364}, mr={0336788}}
\end{barticle}
%

\bptok{imsref}%
\endbibitem

\bibitem{GriffinKuelbs1991}
%
\begin{barticle}[mr]
\bauthor{\bsnm{Griffin},~\bfnm{Philip}\binits{P.}} \AND
\bauthor{\bsnm{Kuelbs},~\bfnm{James}\binits{J.}}
(\byear{1991}).
\btitle{Some extensions of the LIL via self-normalizations}.
\bjournal{Ann. Probab.}
\bvolume{19}
\bpages{380--395}.
\bid{issn={0091-1798}, mr={1085343}}
\end{barticle}
%

\bptok{imsref}%
\endbibitem

\bibitem{GriffinKuelbs1989}
%
\begin{barticle}[mr]
\bauthor{\bsnm{Griffin},~\bfnm{Philip~S.}\binits{P.S.}} \AND
\bauthor{\bsnm{Kuelbs},~\bfnm{James~D.}\binits{J.D.}}
(\byear{1989}).
\btitle{Self-normalized laws of the iterated logarithm}.
\bjournal{Ann. Probab.}
\bvolume{17}
\bpages{1571--1601}.
\bid{issn={0091-1798}, mr={1048947}}
\end{barticle}
%

\bptok{imsref}%
\endbibitem

\bibitem{Hoeffding1948}
%
\begin{barticle}[mr]
\bauthor{\bsnm{Hoeffding},~\bfnm{Wassily}\binits{W.}}
(\byear{1948}).
\btitle{A class of statistics with asymptotically normal distribution}.
\bjournal{Ann. Math. Statist.}
\bvolume{19}
\bpages{293--325}.
\bid{issn={0003-4851}, mr={0026294}}
\end{barticle}
%

\bptok{imsref}%
\endbibitem

\bibitem{JingShaoWang2003}
%
\begin{barticle}[mr]
\bauthor{\bsnm{Jing},~\bfnm{Bing-Yi}\binits{B.-Y.}},
\bauthor{\bsnm{Shao},~\bfnm{Qi-Man}\binits{Q.-M.}} \AND
\bauthor{\bsnm{Wang},~\bfnm{Qiying}\binits{Q.}}
(\byear{2003}).
\btitle{Self-normalized {C}ram\'er-type large deviations for
independent random variables}.
\bjournal{Ann. Probab.}
\bvolume{31}
\bpages{2167--2215}.
\bid{doi={10.1214/aop/1068646382}, issn={0091-1798}, mr={2016616}}
\end{barticle}
%

\bptok{imsref}%
\endbibitem

\bibitem{KeenerRobinsonWeber1998}
%
\begin{barticle}[mr]
\bauthor{\bsnm{Keener},~\bfnm{Robert~W.}\binits{R.W.}},
\bauthor{\bsnm{Robinson},~\bfnm{John}\binits{J.}} \AND
\bauthor{\bsnm{Weber},~\bfnm{Neville~C.}\binits{N.C.}}
(\byear{1998}).
\btitle{Tail probability approximations for {$U$}-statistics}.
\bjournal{Statist. Probab. Lett.}
\bvolume{37}
\bpages{59--65}.
\bid{doi={10.1016/S0167-7152(97)00100-4}, issn={0167-7152}, mr={1622662}}
\end{barticle}
%

\bptok{imsref}%
\endbibitem

\bibitem{KoroljukBorovskich1994}
%
\begin{bbook}[mr]
\bauthor{\bsnm{Koroljuk},~\bfnm{V.~S.}\binits{V.S.}} \AND
\bauthor{\bsnm{Borovskich},~\bfnm{Yu.~V.}\binits{Yu.V.}}
(\byear{1994}).
\btitle{Theory of {$U$}-Statistics}.
\bseries{Mathematics and Its Applications}
\bvolume{273}.
\blocation{Dordrecht}:
\bpublisher{Kluwer Academic}.
\bid{doi={10.1007/978-94-017-3515-5}, mr={1472486}}
\end{bbook}
%

\bptok{imsref}%
\endbibitem

\bibitem{LaiShaoWang2011}
%
\begin{barticle}[mr]
\bauthor{\bsnm{Lai},~\bfnm{Tze~Leng}\binits{T.L.}},
\bauthor{\bsnm{Shao},~\bfnm{Qi-Man}\binits{Q.-M.}} \AND
\bauthor{\bsnm{Wang},~\bfnm{Qiying}\binits{Q.}}
(\byear{2011}).
\btitle{Cram\'er type moderate deviations for Studentized \mbox
{U-}statistics}.
\bjournal{ESAIM Probab. Stat.}
\bvolume{15}
\bpages{168--179}.
\bid{doi={10.1051/ps/2009014}, issn={1292-8100}, mr={2870510}}
\end{barticle}
%

\bptok{imsref}%
\endbibitem

\bibitem{LoganMallowsRiceShepp1973}
%
\begin{barticle}[mr]
\bauthor{\bsnm{Logan},~\bfnm{B.~F.}\binits{B.F.}},
\bauthor{\bsnm{Mallows},~\bfnm{C.~L.}\binits{C.L.}},
\bauthor{\bsnm{Rice},~\bfnm{S.~O.}\binits{S.O.}} \AND
\bauthor{\bsnm{Shepp},~\bfnm{L.~A.}\binits{L.A.}}
(\byear{1973}).
\btitle{Limit distributions of self-normalized sums}.
\bjournal{Ann. Probab.}
\bvolume{1}
\bpages{788--809}.
\bid{mr={0362449}}
\end{barticle}
%

\bptok{imsref}%
\endbibitem

\bibitem{Petrov1965}
%
\begin{barticle}[auto]
\bauthor{\bsnm{Petrov},~\bfnm{V.~V.}\binits{V.V.}}
(\byear{1965}).
\btitle{On the probabilities of large deviations for sums of
independent random variables}.
\bjournal{Theory Probab. Appl.}
\bvolume{10}
\bpages{287--298}.
\bid{issn={0040-361X}, mr={0185645}}
\end{barticle}
%

\bptok{imsref}%
\endbibitem

\bibitem{Serfling1980}
%
\begin{bbook}[mr]
\bauthor{\bsnm{Serfling},~\bfnm{Robert~J.}\binits{R.J.}}
(\byear{1980}).
\btitle{Approximation Theorems of Mathematical Statistics}.
\blocation{New York}:
\bpublisher{Wiley}.
\bid{mr={0595165}}
\end{bbook}
%

\bptok{imsref}%
\endbibitem

\bibitem{Shao1997}
%
\begin{barticle}[mr]
\bauthor{\bsnm{Shao},~\bfnm{Qi-Man}\binits{Q.-M.}}
(\byear{1997}).
\btitle{Self-normalized large deviations}.
\bjournal{Ann. Probab.}
\bvolume{25}
\bpages{285--328}.
\bid{doi={10.1214/aop/1024404289}, issn={0091-1798}, mr={1428510}}
\end{barticle}
%

\bptok{imsref}%
\endbibitem

\bibitem{Shao1999}
%
\begin{barticle}[mr]
\bauthor{\bsnm{Shao},~\bfnm{Qi-Man}\binits{Q.-M.}}
(\byear{1999}).
\btitle{A {C}ram\'er type large deviation result for {S}tudent's
{$t$}-statistic}.
\bjournal{J. Theoret. Probab.}
\bvolume{12}
\bpages{385--398}.
\bid{doi={10.1023/A:1021626127372}, issn={0894-9840}, mr={1684750}}
\end{barticle}
%

\bptok{imsref}%
\endbibitem

\bibitem{ShaoZhangZhou2014}
\begin{barticle}[mr]
\bauthor{\bsnm{Shao},~\bfnm{Qi-Man}\binits{Q.-M.}},
\bauthor{\bsnm{Zhang},~\bfnm{Kan}\binits{K.}} \AND
\bauthor{\bsnm{Zhou},~\bfnm{Wen-Xin}\binits{W.-X.}}
(\byear{2016}).
\btitle{Stein's method for nonlinear statistics: {A} brief survey and recent progress}.
\bjournal{J. Statist. Plann. Inference}
\bvolume{168}
\bpages{68--89}.
\bid{doi={10.1016/j.jspi.2015.06.008}, issn={0378-3758}, mr={3412222}}
\bptnote{check year}%
\end{barticle}
%

\bptok{imsref}%
\endbibitem

\bibitem{ShaoZhou2014}
%
\begin{barticle}[mr]
\bauthor{\bsnm{Shao},~\bfnm{Qi-Man}\binits{Q.-M.}} \AND
\bauthor{\bsnm{Zhou},~\bfnm{Wen-Xin}\binits{W.-X.}}
(\byear{2014}).
\btitle{Necessary and sufficient conditions for the asymptotic
distributions of coherence of ultra-high dimensional random matrices}.
\bjournal{Ann. Probab.}
\bvolume{42}
\bpages{623--648}.
\bid{doi={10.1214/13-AOP837}, issn={0091-1798}, mr={3178469}}
\end{barticle}
%

\bptok{imsref}%
\endbibitem

\bibitem{Stein1986}
%
\begin{bbook}[auto:parserefs-M02]
\bauthor{\bsnm{Stein},~\bfnm{C.}\binits{C.}}
(\byear{1986}).
\btitle{Approximation Computation of Expectations}.
\blocation{Hayward, CA}:
\bpublisher{IMS}.
\end{bbook}
%

\bptok{imsref}%
\endbibitem

\bibitem{vanZwet1984}
%
\begin{barticle}[mr]
\bauthor{\bparticle{van} \bsnm{Zwet},~\bfnm{W.~R.}\binits{W.R.}}
(\byear{1984}).
\btitle{A {B}erry--{E}sseen bound for symmetric statistics}.
\bjournal{Z. Wahrsch. Verw. Gebiete}
\bvolume{66}
\bpages{425--440}.
\bid{doi={10.1007/BF00533707}, issn={0044-3719}, mr={0751580}}
\end{barticle}
%

\bptok{imsref}%
\endbibitem

\bibitem{VandemaeleVeraverbeke1985}
%
\begin{barticle}[mr]
\bauthor{\bsnm{Vandemaele},~\bfnm{M.}\binits{M.}} \AND
\bauthor{\bsnm{Veraverbeke},~\bfnm{N.}\binits{N.}}
(\byear{1985}).
\btitle{Cram\'er type large deviations for Studentized \mbox{$U$-}statistics}.
\bjournal{Metrika}
\bvolume{32}
\bpages{165--179}.
\bid{doi={10.1007/BF01897811}, issn={0026-1335}, mr={0824452}}
\bptnote{check pages}%
\end{barticle}
%

\bptok{imsref}%
\endbibitem

\bibitem{Wang1998}
%
\begin{barticle}[mr]
\bauthor{\bsnm{Wang},~\bfnm{Qiying}\binits{Q.}}
(\byear{1998}).
\btitle{Bernstein type inequalities for degenerate {$U$}-statistics
with applications}.
\bjournal{Chin. Ann. Math. Ser. B}
\bvolume{19}
\bpages{157--166}.
\bid{issn={0252-9599}, mr={1655931}}
\end{barticle}
%

\bptok{imsref}%
\endbibitem

\bibitem{WangJingZhao2000}
%
\begin{barticle}[mr]
\bauthor{\bsnm{Wang},~\bfnm{Qiying}\binits{Q.}},
\bauthor{\bsnm{Jing},~\bfnm{Bing-Yi}\binits{B.-Y.}} \AND
\bauthor{\bsnm{Zhao},~\bfnm{Lincheng}\binits{L.}}
(\byear{2000}).
\btitle{The {B}erry--{E}sseen bound for {S}tudentized statistics}.
\bjournal{Ann. Probab.}
\bvolume{28}
\bpages{511--535}.
\bid{doi={10.1214/aop/1019160129}, issn={0091-1798}, mr={1756015}}
\end{barticle}
%

\bptok{imsref}%
\endbibitem

\bibitem{WangWeber2006}
%
\begin{barticle}[mr]
\bauthor{\bsnm{Wang},~\bfnm{Qiying}\binits{Q.}} \AND
\bauthor{\bsnm{Weber},~\bfnm{Neville~C.}\binits{N.C.}}
(\byear{2006}).
\btitle{Exact convergence rate and leading term in the central limit
theorem for {$U$}-statistics}.
\bjournal{Statist. Sinica}
\bvolume{16}
\bpages{1409--1422}.
\bid{issn={1017-0405}, mr={2327497}}
\end{barticle}
%

\bptok{imsref}%
\endbibitem

\end{thebibliography}
\end{document}